\documentclass[pdftex, 11.5pt, a4paper, oneside, reqno]{article}
\usepackage[english]{babel}
\usepackage{natbib}
\usepackage{appendix}
\usepackage{caption}
\usepackage[usenames,dvipsnames]{xcolor}
\usepackage{subcaption}
\usepackage{chngcntr}
\usepackage{amsfonts,amssymb,amsmath,amsthm}
\usepackage{booktabs}
\usepackage{enumerate}
\usepackage{pgfplots}
\usepackage{hyperref}
\hypersetup{
    bookmarks=true,         % show bookmarks bar?
    unicode=false,          % non-Latin characters in Acrobat’s bookmarks
    pdftoolbar=true,        % show Acrobat’s toolbar?
    pdfmenubar=true,        % show Acrobat’s menu?
    pdffitwindow=false,     % window fit to page when opened
    pdfstartview={FitH},    % fits the width of the page to the window
    pdfnewwindow=true,      % links in new PDF window
    colorlinks=true,       % false: boxed links; true: colored links
    linkcolor=magenta,          % color of internal links (change box color with linkbordercolor)
    citecolor=blue,        % color of links to bibliography
    filecolor=magenta,      % color of file links
    urlcolor=blue           % color of external links
}
\usepackage{float}
\usepackage{setspace}
\usepackage{amsfonts,amssymb,amsmath,amsthm}
\usepackage{fullpage}
\usepackage{import}
\usepackage{rotating}
\usepackage{tikz}
\usetikzlibrary{arrows,chains,matrix,positioning,scopes}
\makeatletter
\tikzset{join/.code=\tikzset{after node path={%
			\ifx\tikzchainprevious\pgfutil@empty\else(\tikzchainprevious)%
			edge[every join]#1(\tikzchaincurrent)\fi}}}
\makeatother
\tikzset{>=stealth',every on chain/.append style={join},
	every join/.style={->}}
\tikzstyle{labeled}=[execute at begin node=$\scriptstyle,
execute at end node=$]
\usepackage{pgfplots}
\usetikzlibrary{arrows}
\DeclareFontFamily{U}{mathx}{\hyphenchar\font45}
\DeclareFontShape{U}{mathx}{m}{n}{<-> mathx10}{}
\DeclareSymbolFont{mathx}{U}{mathx}{m}{n}
\DeclareMathAccent{\widebar}{0}{mathx}{"73}
\usepackage{mathtools,amscd}
\usepackage{xr}
\theoremstyle{theorem}
\newtheorem{remark}{Remark}[section]
\newtheorem{lemma}{Lemma}[section]

\newtheorem{proposition}{Proposition} [section]
\newtheorem{corollary}{Corollary} [section]

\newtheorem{theorem}{Theorem} [section]
\newtheorem{assumption}{Assumption} [section]

\providecommand{\keywords}[1]{{\textit{Keywords: }} #1}

\newcommand{\reals}{\mathbb{R}}

\newcommand{\diff}{{\rm d}}

\newcommand{\convergeL}{\overset{\mathcal{L}}{\longrightarrow}}

\newcommand{\myp}[1]{\left(#1\right)}

\newcommand{\expect}[1]{\mathbb{E}{\left(#1\right)}}
\newcommand{\expectsigma}[1]{\mathbb{E}_{\sigma}{\left(#1\right)}}

\newcommand{\idfun}[1]{\mathbf{1}_{\{#1\}}}

\newcommand{\conexp}[2]{\mathbb{E}\left(#1\left|#2\right.\right)}

\newcommand{\infor}[1]{\mathcal{F}_{#1}}
\newcommand{\abs}[1]{\left|#1\right|}

\newcommand{\var}[1]{\mathbf{Var}{\left(#1\right)}}

\newcommand{\prob}[1]{\mathbb{P}{\left(#1\right)}}
\newcommand{\cov}[1]{\mathbf{Cov}{\left(#1\right)}}

\newcommand{\normdist}[2]{\mathcal{N}\myp{#1,#2}}

\newcommand{\Pconverge}{\overset{\mathbb{P}}{\rightarrow}}

\newcommand{\Pbv}{{\rm PAV}}
\newcommand{\preavg}[1]{\widebar{#1}^{n}_m}

\newcommand{\RV}[1]{\left\langle #1\right\rangle}
\newcommand{\LsConverge}{\overset{\mathcal{L}-s}{\longrightarrow}}

\usepackage{url}
\title{Dependent Microstructure Noise and Integrated Volatility Estimation from High-Frequency Data
}

\author{Z. Merrick Li\thanks{Corresponding author. University of Amsterdam, Amsterdam School of Economics, PO Box 15867, 1001 NJ Amsterdam, The Netherlands. Email: \href{mailto:Z.Merrick.Li@gmail.com}{Z.Merrick.Li@gmail.com}. Phone: +31 (0)20 5254252.}\\
{\small Erasmus University Rotterdam}\\
{\small University of Amsterdam}\\
{\small and Tinbergen Institute}
\and
Roger J. A. Laeven\thanks{University of Amsterdam, Amsterdam School of Economics, PO Box 15867, 1001 NJ Amsterdam,
The Netherlands. Email: \href{mailto:R.J.A.Laeven@uva.nl}{R.J.A.Laeven@uva.nl}. Phone: +31 (0)20 5254219.}\\
{\small Amsterdam School of Economics}\\
{\small University of Amsterdam, EURANDOM}\\
{\small and CentER}
\and
Michel H. Vellekoop\thanks{University of Amsterdam, Amsterdam School of Economics, PO Box 15867, 1001 NJ Amsterdam, The Netherlands.
Email: \href{mailto:M.H.Vellekoop@uva.nl}{M.H.Vellekoop@uva.nl}. Phone: +31 (0)20 5254210.}\\
{\small Amsterdam School of Economics}\\
{\small University of Amsterdam}}

\date{\today }

\begin{document}

\def\spacingset#1{\renewcommand{\baselinestretch}%
	{#1}\small\normalsize} \spacingset{1}

\spacingset{1.45} % DON'T change the spacing!
	
\maketitle

\begin{abstract}
In this paper, we develop econometric tools to analyze the integrated volatility of the efficient price
and the dynamic properties of microstructure noise in high-frequency data
under general dependent noise.
We first develop consistent estimators of the variance and autocovariances of noise using a variant of realized volatility.
Next, we employ these estimators to adapt the pre-averaging method and derive a consistent estimator of the integrated volatility, which converges stably to a mixed Gaussian distribution at the optimal rate $n^{1/4}$.
To refine the finite sample performance, we propose a two-step approach that corrects the finite sample bias,
which turns out to be crucial in applications.
Our extensive simulation studies demonstrate the excellent performance of our two-step estimators.
In an empirical study, %Empirically,
 we characterize the dependence structures of microstructure noise in several popular sampling schemes and provide intuitive economic interpretations; we also illustrate the importance of accounting for both the serial dependence in noise and the finite sample bias
when estimating integrated volatility.
%   study the impacts on integrated volatility estimation.

%	we find strong evidence of positively autocorrelated noise in the examined stocks
%   and we illustrate the considerable importance of taking both the dependence in noise and the interlocked bias
%   into account.
\end{abstract}

\medskip
	
	
\newpage

\doublespacing

\section{Introduction}
Over the past decade and a half, high-frequency financial data have become increasingly available.
In tandem, the development of econometric tools to study the dynamic properties of high-frequency data
has become an important subject area %realm
in economics and statistics.
A major challenge is provided by the accumulation of market microstructure noise at higher frequencies,
which can be attributed to various market microstructure effects including, for example, information asymmetries (see~\cite{glosten1985bid}), inventory controls (see~\cite{ho1981optimal}), discreteness of the data (see~\cite{harris1990estimation}), and transaction costs (see~\cite{garman1976market}).

It has been well-established (see, e.g., \cite{black1986noise}) that the observed transaction price\footnote{In this paper, ``price'' always refers to the ``logarithmic price''.} $Y$ can be decomposed into the unobservable ``efficient price'' (or ``frictionless equilibrium price'') $X$ plus a noise component $U$ that captures %various
market microstructure effects.
That is, it is natural to assume that
\begin{equation}
Y_t = X_t+U_t,
\label{eq:Y=X+U}
\end{equation}
where further assumptions on $X$ and $U$ need to be stipulated.
%Therefore, the specification of the dynamic properties of  microstructure noise plays a key role to conduct statistical inference, e.g, the estimation of the integrated volatility of the efficient price process.
While estimating the integrated volatility of the efficient price is the emblematic problem in high-frequency financial econometrics (see, for example,~\cite{Ait-Sahalia&Jacod2014HFE}), the study of microstructure noise, e.g., its magnitude, dynamic properties, etc., is the main focus of the market microstructure literature (see, for example,~\cite{hasbrouck2007empirical}).
A common challenge, however, is that the two components of the observed price $Y$ in \eqref{eq:Y=X+U} are latent.
Therefore, distributional features of one component, say, of %, such as
the microstructure noise,
will affect the estimation of characteristics of the other, such as the integrated volatility of the efficient price.\footnote{Indeed, while high-frequency data in principle facilitate the asymptotic and empirical analysis of volatility estimators,
the pronounced presence of microstructure noise at high frequency subverts the desirable properties of traditional estimators such as realized volatility.}

While the semimartingale framework provides the natural class to model the efficient price (see, e.g.,~\cite{duffie2010dynamic}),
the statistical assumptions on %properties of
noise
induced by microeconomic financial models range from simple to very complex,
depending on which phenomena the model aims to capture.
For example, the classic Roll model (see~\cite{Roll1984simple}) postulates an i.i.d. bid-ask bounce
resulting from uncorrelated order flows;~\cite{hasbrouck1987order},~\cite{choi1988estimation}, and~\cite{stoll1989inferring}
introduce autocorrelated order flows, yielding autoregressive microstructure noise; and~\cite{gross2013predicting} model microstructure noise with long-memory properties.
Therefore, being able to account for the potentially complex statistical behavior of microstructure noise
that contaminates our observations of %
the semimartingale efficient price dynamics,
would be an appealing property of any method that aims at disentangling the efficient price and microstructure noise.
%is to estimate the integrated volatility of the efficient price process $X$.
%While high-frequency data in principle facilitate the asymptotic and empirical analysis of volatility estimators,
%the pronounced presence of microstructure noise at high frequency subverts the desirable properties of traditional estimators such as realized volatility.

%To account for the contamination of the efficient price by market microstructure noise when estimating integrated volatility,
To estimate the integrated volatility of the efficient price, several de-noise methods have been developed, mostly assuming i.i.d. microstructure noise.
Examples include the two-scale and multi-scale realized volatility estimators developed in~\cite{zhang2005TSRV} and~\cite{zhang2006MSRV},
the realized kernel methods developed in~\cite{barndorff2008RealizedKernels},
the likelihood approach initiated by~\cite{ait2005often} and~\cite{xiu2010},
and the pre-averaging method developed in a series of papers by~\cite{podolskij2009pre-averaging-1} and~\cite{jacod2009pre-averaging-2,jacod2010pre-averaging-3}, see also~\cite{podolskij2009bipower}.
The variance of noise is usually obtained as a by-product. %, ~\cite{hautsch2013preaveraging} and~\cite{christensen2014jump-high-fre}.

In this paper, we allow the microstructure noise to be serially dependent in a general setting, nesting many special cases (including independence).
We do not impose any parametric restrictions on the distribution of the noise,
except for some rather general mixing conditions that guarantee the existence of limit distributions, hence our approach is essentially nonparametric.
In this setting, we first derive the stochastic limit of the realized volatility of observed prices after $j$ lags.
%(see~\eqref{eq:Pconverge_jth_RV}) has a stochastic limit $\var{U}-\gamma(j)$ where $\var{U}$ is the variance of noise and $\gamma(j)$ is the $j$-th order %covariance.
%  . We generalize the result by taking the differences after more lags (rather than 1) to compute the realized volatility.  This simple strategy enables us to
Using %Based on
this limit result, we develop consistent estimators of the variance and covariances of noise.
% However, the realized volatility estimators of the second moments of noise are not performing well in finite samples.
The aim of estimating the second moments of noise is twofold.
On the one hand, we would like to explore the dynamic properties of microstructure noise.
In particular, we would like to compare these properties to those induced by various parametric models of microstructure noise
based on leading microstructure theory,
and obtain corresponding economic interpretations to achieve a better understanding of the microstructure effects in high-frequency data.
On the other hand, the second moments of noise become nuisance parameters in estimating the integrated volatility,
which is a prime objective in the analysis of high-frequency financial data.
	
To estimate the integrated volatility, we next adapt the pre-averaging estimator (PAV) to allow for serially dependent noise in our general setting.
We find that the stochastic limit of the adapted PAV estimator is a function of the volatility and the variance and covariances of noise,
and the latter, constituting an \emph{asymptotic bias},
can be consistently estimated by our realized volatility estimator.
Hence, we can correct the asymptotic bias, resulting in centered estimators of the integrated volatility.

A key interest in this paper is to unravel the %important
interplay between asymptotic and finite sample biases when estimating integrated volatility.
In a finite sample analysis, we find that the realized volatility estimator has a finite sample bias that is %a fraction of
proportional to the integrated volatility.
The bias term becomes significant when the number of lags (in computing the variant of realized volatility) is large,
or the noise-to-signal ratio\footnote{The ratio of the variance of noise and the integrated volatility.} is small.
Therefore, we are in a situation in which the integrated volatility generates a \emph{finite sample bias} to the estimators of the second moments of noise,
while the latter become the \emph{asymptotic bias} in estimating the former.
This ``feedback effect'' in the bias corrections motivates us to develop \textit{two-step estimators}.
First, we simply ignore the dependence in noise and proceed with the pre-averaging method
to obtain an estimator of the integrated volatility.
Next, we use this estimator to obtain \emph{finite sample bias} corrected estimators of the second moments of noise,
which can then be used to correct the asymptotic bias yielding the second-step estimator of the integrated volatility.
Repeating this process leads to three-step estimators (and beyond) which may further improve the two-step estimators on average,
but at the cost of higher standard deviations.
Figure~\ref{fig:description_two_step_estimator} gives a simple graphical illustration of the implementation of the two-step estimators.
	%. Subsequently, we obtain the the second step estimator of the integrated volatility after removing the asymptotic bias using the finite sample bias corrected estimates of the second moments of noise.

We conduct extensive Monte Carlo experiments to examine the performance of our estimators,
which proves to be excellent.
We demonstrate in particular that they can accommodate both serially dependent and independent noise
and perform well in finite samples with realistic data frequencies and sample sizes.
The experiments reveal the importance of
a unified treatment of asymptotic and finite sample biases when
%correcting the finite sample bias in
estimating
%the moments of noise.
integrated volatility.
		
Empirically, we apply our new estimators to a sample of Citigroup transaction data.
We find that the associated microstructure noise tends to be positively autocorrelated.
This is in line with earlier %empirical
findings in the microstructure literature, see~\cite{hasbrouck1987order},~\cite{choi1988estimation}, and~\cite{huang1997components}.
Attributing %interpret
this positive autocorrelation to order flow continuation, the estimated probability that a buy (or sell) order follows another buy (or sell) order is 0.87.
Furthermore, microstructure noise turns out to be negatively autocorrelated under tick time sampling.
This is consistent with inventory models, in which dealers alternate quotes to maintain their inventory position.
We obtain an estimate of the probability of reversed orders equal to 0.84.
Turning to the estimators of integrated volatility, we find that with positively autocorrelated noise
the commonly adopted methods %, for example, the pre-averaging method,
that hinge on the i.i.d. assumption of noise
tend to overestimate the integrated volatility.
Under two alternative (sub)sampling schemes --- regular time sampling and tick time sampling --- our estimators also appear to work well.
This testifies %We show in particular
to the critical relevance of the %finite sample
bias corrections embedded in our two-step estimators.
%Thus our empirical investigation offers some insights on the choices of sampling schemes in practice.
%	that noise is close to be independent in this subsample. The estimates of integrated volatility by our 	two-step estimator and the classical pre-averaging estimator with i.i.d. noise are very close. %The robustness to the (mis)specification of noise is a great advantage of our estimators.
	
%\cite{ait2005often} and, especially,~\cite{hansen2006realized} analyze several \emph{continuous-time} specifications
%of noise that allow dependence in calendar time, assess their implications for volatility estimation,
%and conduct an interesting empirical analysis.\footnote{\cite{hansen2006realized} also consider the case of dependence between noise
%and the efficient price, which is excluded in this paper,
%and analyze empirically a specific tick time dependence.}
In earlier literature, \cite{Ait-Sahalia2011DependentNoise} show that the two-scale and multi-scale realized volatility estimators
are robust to exponentially decaying dependent noise. %, characterized by strongly mixing sequences with exponentially decaying dependence.
In this paper, we provide explicit estimators of the second moments of noise
and analyze their asymptotic behavior,
develop bias-corrected estimators of the integrated volatility based on these moments of noise,
and empirically assess the noise characteristics under different sampling schemes.
%Our general setting is essentially non-parametric; see also Remarks \ref{rmk:Sampling_Schemes} and \ref{rmk:q-dependence}
%below for connections between the various noise assumptions.
Furthermore, \cite{hautsch2013preaveraging} study $q$-dependent microstructure noise,
develop consistent estimators of the first $q$ autocovariances of microstructure noise and define %
the associated pre-averaging estimators.
An appealing feature of their approach is that their autocovariance-type estimators of $q$-dependent noise consider non-overlapping increments which avoids finite sample bias.
%To distinguish our paper from the two, we first note that
We allow for more general assumptions on the dependence structure of microstructure noise.
%\footnote{In particular, we allow the mixing sequence to decay at a polynomial rate, which is weaker than an exponential rate.
%The relevance of this generality is that without further specifications, one cannot distinguish the two decaying rates.}.
%%While the generality brings mathematical and econometric challenges %(see also~\cite{jacod2015IVDependentNoise}),
Owing to its generality our setting incorporates many microstructure models as special cases.
%beyond $q$ dependence model or exponentially decaying noise
%This warrants %emphasis
%that
We therefore do not need to %
advocate any particular model of microstructure noise
and this enables us to obtain economic interpretations of our empirical results
under multiple sampling schemes.
%., though our (limited) empirical studies turn out to support weakly dependent noise.
% impose similar assumptions on the structure of the noise as we do
%and develop estimators of integrated volatility based on multiple time scales.

%\footnote{We show that our integrated volatility estimator is also consistent with irregular sampling, see Appendix~\ref{sec:IrregularSampling}.}

In two contemporaneous and independent works,
~\cite{jacod2015IVDependentNoise,jacod2013StatisticalPropertyMMN} also study dependent noise in high-frequency data.
In~\cite{jacod2013StatisticalPropertyMMN}, they develop a novel local averaging method to ``recover'' the noise and they %
can, in principle, estimate any finite (joint) moments of noise with diurnal features.
Moreover, they also allow observation times to be random.
Empirically, they find some interesting statistical properties of noise.
In particular, they find that noise is strongly serially dependent with polynomially decaying autocorrelations.
Employing this local averaging method,~\cite{jacod2015IVDependentNoise} develop an %pre-averaging
estimator of integrated volatility that allows for dependent noise.
To distinguish our work from these two papers, we first note that our assumptions on noise are slightly different:
we assume that the noise process constitutes a strongly mixing sequence while they require a $\rho$-mixing sequence
(see~\cite{bradley2005StrongMixing} for a discussion of mixing sequences).
%Moreover, we obtain consistency theory for the pre-averaging bi-power variation, a more general functional form that also permits finite activity jumps.
Furthermore, the local averaging method differs from, and allows to analyze more general noise characteristics than,
the simpler realized volatility method developed here.
The key difference is our explicit treatment of the feedback effect between the asymptotic and finite sample biases:
we show that in a finite sample, the integrated volatility and second moments of microstructure noise should be estimated in a unified way, since they induce biases in % to
 each other.
We design novel and easily implementable two-step estimators to correct for the intricate biases.
Our two-step estimators of the integrated volatility, which are designed to allow for dependent noise, also perform well
in the special case of independent noise, and in a sample of reasonable size as encountered in practice.
This robustness to (mis)specification of noise and to sampling frequencies is an important advantage of our two-step estimators.
%\footnote{Obviously,
%a polynomially decaying rate for the mixing coefficients as considered in~\cite{jacod2013StatisticalPropertyMMN,jacod2015IVDependentNoise}
%is mathematically more general
%than the exponentially decaying rate considered in this paper.
%Our exponentially decaying rate is motivated in part by %extant
%market microstructure pricing theory; see, for example,~\cite{amihud1987trading}
%for a %simple
%model of price adjustments.
%Moreover, this fast-decaying rate is also supported by our empirical evidence; see Section~\ref{sec:EmpiricalStudy}.
%We suspect that the lack of
Our unified treatment of the asymptotic and finite sample biases
may help explain why %
%in
the empirical studies in~\cite{jacod2013StatisticalPropertyMMN}
render the strong dependence in noise they find (and question themselves); see our
%simulation studies in Section~\ref{sec:simulation} and
empirical analysis in Section~\ref{sec:EmpiricalStudy}.
%}

In another independent paper,~\cite{da2017moving} introduce a novel quasi maximum likelihood approach
to estimate both the volatility and the autocovariances of moving-average microstructure noise.
They also extend their estimators to general settings that allow for irregular observation times,
intraday patterns of noise and jumps in asset prices.
Their approach treats ``large'' and ``small'' microstructure noise in a uniform way which leads to a potential improvement in the convergence rate.
Our approach is essentially of a nonparametric nature
and provides unified estimators of a class of volatility functionals (see Theorem~\ref{thm:consistency}) including the asymptotic variance,
which account for the feedback between finite sample and asymptotic biases. %\footnote{We note that in~\cite{da2017moving}, the asymptotic variance estimators are based on the pre-averaging method.}
Our empirical study also has a different focus.
Our investigation is not as extensive as in~\cite{da2017moving},\footnote{Da and Xiu maintain a website to provide up-to-date daily annualized volatility estimates for all S\&P 1500 index constituents, see~\url{http://dachxiu.chicagobooth.edu/\#risklab}.}
but we explicitly consider different sampling schemes,\footnote{In their empirical studies,~\cite{da2017moving} only consider tick time sampling.}
analyzing the autocovariance patterns of noise in connection to microstructure noise models
and their impact on integrated volatility estimation.
	
The remainder of this paper is organized as follows.
In Section~\ref{sec:framework}, we introduce the basic setting and notation.
In Section~\ref{sec:VarianceCovarianceEst}, we analyze realized volatility with dependent noise
and develop consistent estimators of the second moments of noise.
The pre-averaging method with dependent noise is studied in Section~\ref{sec:pre-averaging}.
Section~\ref{sec:two-step estimators} introduces our two-step estimators. %to correct the finite sample biases.
Section~\ref{sec:simulation} reports extensive simulation studies.
Our empirical study is presented in Section~\ref{sec:EmpiricalStudy}.
Section~\ref{sec:conclusion} concludes the paper.
All proofs and some additional Monte Carlo simulation and empirical results
are collected in %relegated to
an online appendix, see \cite{LLVsupp2018}.
%	
%	\begin{figure}[h]
%		\centering
%		\input{ACFCiti_Tick_Second_Ret.tex}
%		\caption{Autocorrelation function of Citi log-return for January, 2011. In the top panel, the data is sampled at 1 second time scale; in the bottom panel, the data is sampled at tick time, i.e., time when price changes.}
%		\label{fig:ACFGE0401}
%	\end{figure}
%	

\section{Framework and Assumptions}\label{sec:framework}
We assume that the efficient log-price process $X$ is represented by a continuous It\^o semimartingale
defined on a filtered probability space $(\Omega,\mathcal{F},(\mathcal{F}_t)_{t\geq 0},\mathbb{P})$:
\begin{equation}
X_t = X_0 + \int_{0}^{t}a_s\diff s + \int_{0}^{t}\sigma_s\diff W_s,
\label{eq:Eff_Price_Ito_Diffu}
\end{equation}
where $W$ is a standard Brownian motion, the drift process $a_s$ is optional
and %
locally bounded,
and the volatility process $\sigma_s$ is adapted with c\`adl\`ag paths.
The probability space also supports the noise process $U$.
We assume that all observations are collected in the fixed time interval $[0,T]$,
where without losing generality we let $T=1$.
At stage $n$, the observation times are given by $0=t^n_0<t^n_1<\dots<t^n_n=1$.

%	We also assume regular discretization scheme $t^n_i = i/n, i=0,...,n$. For any process $V$, we let $V^n_i = V_{t^n_i}$.

\begin{assumption}[Market microstructure noise]\label{assumption:dependent_noise}
The noise process $(U_i)_{i\in\mathbb{N}}$ satisfies the following assumptions:
	\begin{enumerate}
		\item $U$ is symmetrically distributed around 0;
		\item The noise process $U$ is independent of the efficient log-price process $X$; %defined in~\eqref{eq:Eff_Price_Ito_Diffu};
		\item \label{assumption:noiseAssump3}$U$ is stationary and strongly mixing and the mixing coefficients\footnote{The \emph{mixing coefficients} constitute a sequence satisfying
			\[
			\abs{\prob{A\cap B}-\prob{A}\prob{B}}\leq \alpha_h,
			\]
            for all $A\in\sigma\myp{U_0,\dots,U_k},B\in\sigma\myp{U_{k+h},U_{k+h+1},\dots}$, where $\sigma(A)$ is the $\sigma$-algebra generated by $A$.
            We refer to~\cite{bradley2007strongMixing} or Chapter VIII of~\cite{jacod1987limit}
            for further details on and properties of mixing sequences.} $\{\alpha_h\}_{h=1}^\infty$ decay at a polynomial rate,
            i.e., there exist some constants $C>0,v>0$ such that
        \begin{equation}\label{eq:alpha_mixing_coeff_v}
        \alpha_h \leq\frac{C}{h^v}.
        \end{equation}Moreover, we assume $U$ has bounded moments of all orders.
	\end{enumerate}
\end{assumption}
 %\footnote{This notation indicates that the distribution of noise is independent of the number of observations $n$. This is consistent with the general principle in high-frequency analysis that the number of observations (or data frequency) does not alter the distribution of noise. A natural interpretation is to consider another probability space $(\Omega',\mathcal{F}',(\mathcal{F}'_i)_{i\in\mathbb{N}},\mathbb{P})$ that supports the noise process $(U_i)_{i\in\mathbb{N}}$. $\mathbb{N}$ can be replaced by any countable set corresponding to the transaction times thus could be irrelevant to the calender time. In that case, the probability space $(\Omega,\mathcal{F}, (\mathcal{F}_{t^n_i})_{0\leq i\leq n}, \mathbb{P})$ after any realization of $n$ observations will the probability extension of $(\Omega',\mathcal{F}',(\mathcal{F}'_i)_{i\in\mathbb{N}},\mathbb{P})$ and the probability space that supports the efficient price. But to keep notation consistent, we do not make this extension. One is referred to~\cite{ait2014high} and~\cite{jacod2013StatisticalPropertyMMN} for a further discussion about the compatibility of colored noise. }.

%\begin{remark}[Mixing conditions]\label{rmk:rho_strong_mixing}
\noindent The mixing conditions in Assumption~\ref{assumption:dependent_noise} item (3.)
ensure that the noise process evaluated at different time instances, say, $i,i+h$,
is increasingly limited in dependence when the lag $h$ increases.
In particular, there exists some $C'>0$ such that
\begin{equation}
\abs{\gamma(h)}\leq \frac{C'}{h^{v/2}},
\label{eq:rho_strong_mixing}
\end{equation}
where $\gamma(h) = \cov{U_{i}, U_{i+h}}$ is the autocovariance function of $U$.
Assuming $U$ to have bounded moments of all orders is not strictly necessary.
Depending on the targeted moments, this assumption can be relaxed via the choice of $v$ in~\eqref{eq:alpha_mixing_coeff_v}, see Lemma VIII 3.102 in~\cite{jacod1987limit}.
Throughout the paper we maintain the assumption of bounded moments of all orders
and only specify the %respective
restrictions on $v$.
%\end{remark}

At stage $n$, we will denote $U_i$ by $U^n_i$, $\forall i\leq n$.
The $i$-th observed price is thus given by
\begin{equation}\label{eq:Y^n=X^n+U^n}
Y^n_i = X^n_i + U^n_i,
\end{equation}where $X^n_i = X_{t^n_i}$.
In the remainder of the main text,
we assume $t^n_i = i/n, i=0,\dots, n$; see Appendix~\ref{sec:IrregularSampling} for an analysis of irregular sampling schemes.

\begin{remark}[Microstructure noise and sampling schemes]
\label{rmk:Sampling_Schemes}
We allow the noise process $U$ to generate dependencies in \emph{sampling time},
including \emph{transaction time},\footnote{Under this sampling scheme, $Y^n_i$ (resp. $X^n_i,U^n_i$) is the observed log-price (resp. efficient log-price, microstructure noise) associated with the $i$-th trade.
The observation times $(t_{i}^{n})_{0\leq i\leq n}$ can, in general, be deterministic or random, and regular or irregular.
%Throughout the main text we assume $t_{i}^{n}=i/n$ and we refer to Appendix~\ref{sec:IrregularSampling} for the deterministic irregular case.
}
\emph{calendar time},\footnote{Under this sampling scheme, $Y^n_i$ (resp. $X^n_i,U^n_i$) is the observed log-price (resp. efficient log-price, microstructure noise) at regular time $i\Delta_n$, with $\Delta_n = 1/n$ in the main text.}
and \emph{tick time}.\footnote{Tick time sampling removes all zero returns; see~\cite{Ait-Sahalia2011DependentNoise} and~\cite{griffin2008sampling}.
Hence, $Y^n_i$ is by definition different from $Y^n_{i-1}$ and $Y^n_{i+1}$ under this sampling scheme.}
Hence, our noise process essentially constitutes a \emph{discrete-time model} --- it does not depend explicitly on the time %in
between successive observations. ~\cite{ait2005often},~\cite{hansen2006realized}, and ~\cite{hansen2008moving} study various \emph{continuous-time} models of dependent microstructure noise.
In these continuous-time models, the noise component of a log-return over a time interval $\Delta$ is of order $O_p(\sqrt{\Delta})$,
the same order as the
logarithmic return of the %
efficient price.
%\footnote{~\cite{hansen2006realized} also consider a specific tick time dependence.}
%Thus it is not a proper model for various market microstructure effects, for instance, the bid-ask spread.
%which leads to different asymptotic properties. Calender-time dependent noise has much smaller stochastic orders,\footnote{Under regularity conditions, the variance of the noise component in the log-returns will be of the same order as the variance of the efficient prices. The intuition is that as the data frequency becomes higher, the calender-dependent noise terms in successive observed prices are closely correlated, hence their difference has much less variation.} hence is less appropriate to fit some empirical facts like the volatility smile.
\end{remark}
%It is well known that the sample mean of i.i.d. random variables converges to a Gaussian distribution under some regularity conditions. The following lemma generalizes this result --- asymptotic normality is still attainable on strongly mixing sequence. Later, we will use this lemma to derive the asymptotic distribution of the pre-averaging noise, see Proposition~\ref{prop:pre-averaged noise}.

%This is valid since the process is stationary.

\begin{remark}[General dynamic properties of microstructure noise]
\label{rmk:q-dependence}
Our assumptions on the dependence of noise are quite general, nesting many models as special cases including, for example,
i.i.d. noise, $q$-dependent noise (under which $\gamma(h) = 0,\:\forall h>q$), ARMA($p,q$) noise (see~\cite{mixingARMA}) and some long-memory processes (see~\cite{tsay2005analysis}).
We note that AR(1) and AR(2) noise are studied in~\cite{barndorff2008RealizedKernels} and~\cite{hendershott2013implementation} respectively, $q$-dependent noise is considered by~\cite{hansen2008moving} and~\cite{hautsch2013preaveraging}, while~\cite{gross2013predicting} study long-memory bid-ask spreads.
\end{remark}

\section{Estimation of the Variance and Covariances of Noise}\label{sec:VarianceCovarianceEst}
In this section, we develop consistent estimators of the second moments of noise under Assumption~\ref{assumption:dependent_noise}.
These estimators will later serve as important inputs to adapt the pre-averaging method.
We also analyze our estimators' finite sample properties.
% From here onwards, for any process $V$, we denote $V^n_i = V_{{{i}/{n}}}$.

\subsection{Realized volatility with dependent noise}
We start with the following preliminary result:
\begin{proposition}
\label{prop:RV_Estimate_var+cov(1)}
Assume that the efficient log-price follows~\eqref{eq:Eff_Price_Ito_Diffu},
the observations follow~\eqref{eq:Y^n=X^n+U^n},
and the noise process satisfies Assumption~\ref{assumption:dependent_noise}.
Furthermore, let $j$ be a fixed integer and assume the sequence $j_n$ and the exponent $v$ satisfy the following conditions:
\begin{equation}\label{eq:Asy_condi_RV_consistency}
v>2, \quad j_n\rightarrow\infty,\quad  j_n/n\rightarrow 0.
\end{equation}
%Furthermore, assume that there exists an $\epsilon>0$ such that $\expect{\abs{U}^{4+\epsilon}}<\infty.$
Then we have the following convergences in probability as $n\rightarrow\infty$:
\begin{equation}\label{eq:Pconverge_jth_RV}
\widehat{\RV{Y,Y}}_{n}(j) :=\frac{\sum_{i=0}^{n-j} (Y^n_{i+j}-Y^n_{i})^2}{2(n-j+1)} \Pconverge \var{U} - \gamma(j),
\end{equation}
\begin{equation}\label{eq:consistent_estimate_var_U_dependent}
\widehat{\var{U}}_n:=\frac{\sum_{i=0}^{n-j_n} (Y^n_{i+j_n}-Y^n_{i})^2}{2(n-j_n+1)}\Pconverge \var{U},
\end{equation}
\begin{equation}
\widehat{\gamma(j)}_n := \widehat{\var{U}}_n - \widehat{\RV{Y,Y}}_{n}(j)\Pconverge \gamma(j).
\label{eq:gamma(j)_hat}
\end{equation}
%where $Y^n_k = Y_{t^n_k}$.
\end{proposition}
\begin{proof}
See Appendix~\ref{appendix:prop_RV_Estimate_var+cov(1)}.
\end{proof}
	
%	\begin{remark}
%		This result is motivated by the studies of realized volatility (when $j=1$) in~\cite{Ait-Sahalia2011DependentNoise} with similar dependent noise as ours (strongly mixing with an exponential decaying mixing rate). They show that the realized volatility --- a consistent estimator of the variance of noise when noise is i.i.d. --- becomes a consistent estimator of  $\var{U}-\gamma(1)$. We notice in an updated version of~\cite{jacod2013StatisticalPropertyMMN}, they also report this result under different settings of noise  ($\alpha$-mixing and polynomially decaying rate) without proofs. They are interested in estimating the moments $\var{U}-\gamma(j)$ while we are interested in estimating the variances and covariances separately.
%	\end{remark}
%	
%Thus, $\widehat{\RV{Y,Y}}_n(j)$ provides a consistent estimator of $\var{U}-\gamma(j)$.
The special case of~\eqref{eq:Pconverge_jth_RV} that occurs when $j=1$ appears in~\cite{Ait-Sahalia2011DependentNoise}
assuming exponential decay.
We also note that in the most recent version of~\cite{jacod2013StatisticalPropertyMMN}
similar estimators as $\widehat{\RV{Y,Y}}_{n}(j) $ are mentioned but without formal analysis of their limiting behavior.
%Proposition \ref{prop:RV_Estimate_var+cov(1)}, however, only provides a consistent estimator of
%the \emph{differences} between the variance and the covariances of noise and,
%as we will see later, we also need to estimate $\var{U}$ and $\gamma(j)$ \emph{separately}.
To our best knowledge, our paper is the first to estimate the variance and covariances of noise
using realized volatility under a general dependent noise setting.

\subsection{Finite sample bias correction}\label{subsec:Finite_Sample_Bias_Correction}
The theoretical validity of our realized volatility estimators in~\eqref{eq:Pconverge_jth_RV}--\eqref{eq:gamma(j)_hat}
hinges on the increasing availability of observations in a fixed time interval, the so-called \emph{infill asymptotics}.
In general, an estimator derived from asymptotic results can, however, behave very differently in finite samples.
Our realized volatility estimators of the second moments of noise are an example for which
the asymptotic theory provides a poor representation of the estimators' finite sample behavior.\footnote{This applies to
the local averaging estimators developed in~\cite{jacod2013StatisticalPropertyMMN} as well; see Footnote \ref{fn:lafs} for further details.}
	
Intuitively, the finite sample bias stems from the diffusion component,
when computing the realized volatility $\widehat{\RV{Y,Y}}_n(j)$ over large lags $j$ in a finite sample,
and we will explain later (e.g., in Remark \ref{rmk:why_correct_bias}) why it is critically relevant to account for it in real applications.
In the sequel, we assume the drift $a_t$ in~\eqref{eq:Eff_Price_Ito_Diffu} to be zero.
According to, for example,~\cite{bandi2008microstructure} and~\cite{lee2012jumps} this is not restrictive in high-frequency analysis.
This will be confirmed in our Monte Carlo simulation studies in Section \ref{sec:simulation} and Appendix \ref{sec:SVsimu}.
%%
%Since we are taking all $j$ lags to compute $[Y,Y]^j_n$,  some starting and ending observations will contribute less to $[Y,Y]^j_n$.
%We will impose some regularity conditions on the volatility path around 0 and 1 so that we still get a good approximation of the integrated volatility $\int_{0}^{1}\sigma^2_s\diff s$, after neglecting some observations around 0 and 1.\footnote{}
	
\begin{proposition}
\label{prop:Finite_Sample_Bias_Correction}
Assume that the efficient log-price follows~\eqref{eq:Eff_Price_Ito_Diffu} with $a_{s} = 0\:\forall s$,  and assume there is some $\delta>0$ so that $\sigma_t$ is bounded for all $t\in[0,\delta]\cup [1-\delta,1]$.
Furthermore, assume the observations follow~\eqref{eq:Y^n=X^n+U^n}, and the noise process satisfies Assumption~\ref{assumption:dependent_noise}. Then, conditional on the volatility path,
\begin{align}
\expectsigma{\widehat{\RV{Y,Y}}_n(j)} = \frac{j\int_{0}^{1}\sigma^2_t\diff t}{2(n-j+1)} +\var{U} - \gamma(j) + O_p\myp{j^2/n^2}.
\label{eq:Finite_Sample_Bias_Correction}		
\end{align}
Here, $\expectsigma{\cdot}$ is the expectation conditional on the entire path of volatility.
\end{proposition}
\begin{proof}
See Appendix~\ref{sec:prop:Finite_Sample_Bias_Correction}.
\end{proof}
\begin{remark}
The regularity conditions with respect to $\sigma_{t}$ in Proposition \ref{prop:Finite_Sample_Bias_Correction}
trivially hold if the volatility is assumed to be continuous.
(Volatility is usually assumed to be continuous when making finite sample bias corrections.)
%see, e.g.,~\cite{Ait-Sahalia2011DependentNoise} and~\cite{jacod2013StatisticalPropertyMMN}.
\end{remark}
\begin{remark}
Let $j=1$ and let us restrict attention to sampling in calendar time.
In that special case the result in Proposition \ref{prop:Finite_Sample_Bias_Correction} bears similarities with Theorem 1 in~\cite{hansen2006realized}.
Contrary to~\cite{hansen2006realized} we assume that the efficient log-price $X$ is independent of the noise $U$.
Therefore, any correlations between the two drop out.
%Moreover, we assume transaction-time dependent noise while~\cite{hansen2006realized} work with calender-time dependent noise. 	
\end{remark}
	
Proposition \ref{prop:Finite_Sample_Bias_Correction} reveals that $\widehat{\RV{Y,Y}}_n(j) - \frac{j\int_{0}^{1}\sigma^2_t\diff t}{2(n-j+1)}$
will be a better estimator of $\var{U} - \gamma(j)$ in finite samples, and it motivates the following finite sample bias corrected estimators:
\begin{align}\label{eq:RV_SSBC}
	\widehat{\RV{Y,Y}}^{\rm (adj)}_{n}(j) & := \widehat{\RV{Y,Y}}_n(j) -  \frac{\hat{\sigma}^2 j}{2(n-j+1)};\\
	\label{eq:var_noise_dependent_SSBC}
	\widehat{\var{U}}^{\rm (adj)}_n  & := \widehat{\var{U}}_n -  \frac{\hat{\sigma}^2 j_n}{2(n-j_n+1)};\\
	\label{eq:covs_dependent_SSBC}
	\widehat{\gamma(j)}_n^{\rm (adj)}  & := \widehat{\var{U}}_n^{\rm (adj)} - \widehat{\RV{Y,Y}}^{\rm (adj)}_{n}(j);
	\end{align}
where $\hat{\sigma}^2$ is an estimator of $\int_{0}^{1}\sigma^2_s\diff s$.
We note that the bias corrected estimators are still consistent,
as the fraction $\frac{j}{n-j+1}$ is negligible when $j$ is much smaller than $n$.%,i.e., when $j/n=o(1)$.
	
\begin{remark}[Why the finite sample bias matters]
\label{rmk:why_correct_bias}
We now explain why the finite sample bias correction is crucial in applications.
We first rewrite~\eqref{eq:Finite_Sample_Bias_Correction}:
\begin{equation}
\begin{split}
\expectsigma{\widehat{\RV{Y,Y}}_n(j)}& = \frac{j\int_{0}^{1}\sigma^2_t\diff t}{2(n-j+1)} +\var{U} - \gamma(j) + O_p\myp{j^2/n^2} \\
 &= \myp{\var{U}-\gamma(j)} \myp{1+\frac{\frac{j}{2(n-j+1)}}{\frac{\var{U} - \gamma(j)}{\int_{0}^{1}\sigma^2_t\diff t}} }+ O_p\myp{j^2/n^2}.
\label{eq:why_correct_bias}
\end{split}
\end{equation}
Observe that the finite sample bias is determined by the ratio of
the two terms $\frac{j}{2(n-j+1)}$ and $\frac{\var{U} - \gamma(j)}{\int_{0}^{1}\sigma^2_t\diff t}$.
The first term, $\frac{j}{2(n-j+1)}$, depends on the data frequency $(n)$ and ``target parameters'' $(j)$;
the second term, $\frac{\var{U} - \gamma(j)}{\int_{0}^{1}\sigma^2_t\diff t}$, is %related to
the (latent) noise-to-signal ratio.
If the second term is ``relatively larger (smaller)'' than the first one,
then the finite sample bias will be small (large).
In other words, the finite sample bias is not only determined by the data frequency and target parameters,
but also by other properties of the underlying efficient price and noise processes.
	
In high-frequency financial data, the noise-to-signal ratio $\frac{\var{U}}{\int_{0}^{1}\sigma^2_t\diff t}$ is typically small,
but it can vary from $O(10^{-2})$ (see~\cite{bandi2006separating}) to $O(10^{-6})$ (see~\cite{christensen2014jump-high-fre}) in empirical studies.
The ratio $\frac{ j}{2(n-j+1)}$, while typically small as well, can still be \emph{relatively} large,
depending on the specific situation.
Consider the following two scenarios:
\begin{enumerate}[1)]
\item We have ultra high-frequency data with $n=O(10^5)$ (recall that the number of seconds in a business day is 23,400),
and we select $j_{n}=20$.
Then, the ratio $\frac{j_n}{2(n-j_n+1)} = O(10^{-4})$.
\item We have i.i.d. noise and we would like to estimate the variance of noise by $\widehat{\RV{Y,Y}}_{n}(1)$
using high-frequency data with average duration of 20 seconds (thus $n\approx 10^3$);
see, e.g.,~\cite{bandi2006separating}.
Hence, $\frac{j}{2(n-j+1)} = O(10^{-3})$.
\end{enumerate}
In both scenarios, the ratio of $\frac{j}{2(n-j+1)}$ and $\frac{\var{U} - \gamma(j)}{\int_{0}^{1}\sigma^2_t\diff t}$ can vary widely,
depending on the magnitude of the latent noise-to-signal ratio.
It is then clear from the first line of~\eqref{eq:why_correct_bias} that the finite sample bias term,
which is proportional to %a fraction of
the integrated volatility,
may well wipe out the variance of noise, depending on the specific situation.
%Therefore to estimate noise related parameters, we should be serious about the finite sample bias, which is not only related to the data frequency but also to the noise to signal ratio.
\end{remark}
\begin{remark}
Note that increasing the sample size by extending the time horizon %, say,
to $[0,T]$ with large $T$ will not remove the finite sample bias.
%To see this, let $n_T$ be the number of observations in $[0,T]$.
%Then,
%\begin{equation*}
%\begin{split}
%		\myp{\frac{j}{2(n_T-j+1)} + \frac{\var{U} - \gamma(j)}{\int_{0}^{T}\sigma^2_t\diff t}}\int_{0}^{T}\sigma^2_t\diff t %&\approx \myp{\frac{j}{2(nT-j+1)} + \frac{\var{U} - \gamma(j)}{T\int_{0}^{1}\sigma^2_t\diff t}}T\int_{0}^{1}\sigma^2_t\diff t\\
%		   & \approx \myp{\frac{j}{2\myp{n-\frac{j-1}{T}}} + \frac{\var{U} - \gamma(j)}{\int_{0}^{1}\sigma^2_t\diff t}}\int_{0}^{1}\sigma^2_t\diff t,
%\end{split}
%\end{equation*}
%where we approximate $n_T$ by $nT$ and $\int_{0}^{T}\sigma^2_t\diff t$ by $T\int_{0}^{1}\sigma^2_t\diff t.$
Hence, the finite sample bias may be viewed as a \emph{low frequency bias}.
\end{remark}
	
\section{The Pre-Averaging Method with Dependent Noise}\label{sec:pre-averaging}

In this section, we adapt a popular ``de-noise'' method --- the pre-averaging method --- to allow for
serially dependent noise in our general setting.
The pre-averaging method was originally introduced by~\cite{podolskij2009pre-averaging-1}
(see also~\cite{jacod2009pre-averaging-2},~\cite{jacod2010pre-averaging-3},~\cite{podolskij2009bipower}, and~\cite{hautsch2013preaveraging}). % We develop the consistency theory for the pre-averaging estimators and a central limit theorem for the integrated volatility.

\subsection{Setup and notation}
For a generic process $V$,
we denote its pre-averaged version by
\begin{equation}
\preavg{V} := \frac{1}{k_n+1}\sum_{i=(2m-2)k_n}^{(2m-1)k_n}\myp{V^n_{i+k_n} - V^n_{i}},
\label{eq:V_m^k_n}	
\end{equation}
for $1\leq m\leq M_n$ with $M_n=\lfloor \frac{\sqrt{n}}{2c}\rfloor$,
where $k_n\in\mathbb{N}$ satisfies
\begin{equation}
k_n = c\sqrt{n} + o(n^{1/4}),
\label{eq:k_nM_condition}
\end{equation}
for some positive constant $c$ and where $\lfloor\cdot\rfloor$ is the floor function. %\footnote{To be precise, we should denote $k_n =\lfloor c\sqrt{n} \rfloor$, where $\lfloor x\rfloor$ is the largest integer not exceeding $x$. However, for simplicity we keep the notation $k_n=c\sqrt{n}$. This is not restrictive for our analysis since for arbitrary $c>0$,  $c\sqrt{n}$ and $\lfloor c\sqrt{n} \rfloor$ are  asymptotically equivalent.}. %, where $[c\sqrt{n}]$ is the integer part of $c\sqrt{n}$.
For any real $r\geq 2$, the pre-averaged statistics of the log-price process $Y$ are defined as follows: %\footnote{It's called \emph{modulated bi-power variation} (MBV) in~\cite{podolskij2009pre-averaging-1} to indicate its close connection with the bi-power variation statistics in~\cite{barndorff2004bipower,barndorff2006bipower}.}
\begin{equation}
\Pbv(Y,r)_n: = n^{\frac{r-2}{4}}\sum_{m=1}^{M_n}\abs{\preavg{Y}}^r,\quad r\geq 2.
\label{eq:MBV_Y}
\end{equation}
\begin{remark}
Equation \eqref{eq:V_m^k_n} invokes a simple version of the pre-averaging method.
In particular, we take a simple weighting function to compute the pre-averages in the $m$-th \emph{non-overlapping} interval.
We refer to~\cite{jacod2009pre-averaging-2,jacod2010pre-averaging-3} and~\cite{podolskij2009bipower}
for the pre-averaging method with general weighting functions and pre-averaged values based on %statistics defined on
 overlapping intervals.	
\end{remark}
	
%where $M_n=\lfloor \frac{\sqrt{n}}{2c}\rfloor$ for the same $c$ as in~\eqref{eq:k_nM_condition}.
%Here, 	

We first present the following proposition, which provides the asymptotic distribution of the pre-averaged noise:
\begin{proposition}
\label{prop:pre-averaged noise}
Assume that the noise satisfies Assumption~\ref{assumption:dependent_noise} with $v>2$ and that $\sigma^2_U$ defined below is strictly positive.
%Furthermore, assume that $\expect{{|U|}^{2+\delta}}<\infty$ for some $0<\delta<\infty$.
Then, the following central limit theorem holds for $\preavg{U}$:
\begin{equation}
n^{1/4}\preavg{U}\convergeL \normdist{0}{\frac{2\sigma^2_U}{c}},
\label{eq:asy_distri_pre-averaged_noise}
\end{equation}
where
\begin{align}
\sigma^2_U = \var{U} + 2\sum_{j=1}^{\infty}\gamma(j),
\label{eq:sigma2_U}
\end{align}
and $c$ is defined in \eqref{eq:k_nM_condition}.
\end{proposition}
\begin{proof}
See Appendix~\ref{appendix:prop:pre-averaged noise}.
\end{proof}
%\begin{remark}
%In particular, if the noise exhibits $q$-dependence, we have $\sigma^2_U = \var{U} + 2\sum_{j=1}^{q}\gamma(j)$.
%\end{remark}
For i.i.d. noise, $\sigma^2_U$ reduces to $\var{U}$, and it is known (see~\cite{zhang2005TSRV} and~\cite{bandi2008microstructure})
that the variance of noise can be consistently estimated by the standardized realized volatility of observed returns.
However, when noise is dependent we face a much more complex situation:
all variance and covariance terms constitute $\sigma^2_U$.
Nevertheless, we can provide a consistent estimator of $\sigma^2_U$, as follows:  %Nevertheless, we provide a consistent estimator of $\sigma^2_U$.

\begin{proposition}
\label{prop:consistent_nonpar_sigma2U}
Let $v>2$ and $j_n^3/n\rightarrow 0$. Define
\begin{equation}
\widehat{\sigma^2_U}: = \widehat{\var{U}}_n + 2\sum_{j=1}^{i_n}\widehat{\gamma{(j)}}_n,
\label{eq:consistent_nonpar_sigma2U}
\end{equation}
where $i_n$ satisfies the conditions $i_n\rightarrow\infty, i_n\leq j_n$,
and $\widehat{\var{U}}_n$ and $\widehat{\gamma(j)}_n$ are defined in~\eqref{eq:consistent_estimate_var_U_dependent} and~\eqref{eq:gamma(j)_hat}.
%Suppose the assumptions of Proposition \ref{prop:RV_Estimate_var+cov(1)} hold.
Then,
\begin{equation}
\widehat{\sigma^2_U}\Pconverge \sigma^2_U.
\end{equation}
\end{proposition}
\begin{proof}
See Appendix~\ref{appendix:prop:consistent_nonpar_sigma2U}.
\end{proof}
%\begin{remark}[\textcolor{red}{need revision}]
%\label{rmk:converge_rate_HAC_estimator}
%The asymptotic condition $\frac{n}{{i_n}^2} =O(n^\ell), \ell >\frac{1}{2}$,
%can be weakened to $\frac{n}{{j_n^*}^2} \rightarrow \infty$, as one may observe from the proof of this proposition.
%The additional condition is needed to achieve a faster convergence rate (than $n^{1/4}$) in order to establish the associated central limit theorem;
%see Theorem~\ref{thm:CLT} below.
%However, we impose unified assumptions for ease of presentation.
%\end{remark}

\subsection{Asymptotic theory: Consistency}

The following results establish consistency and a central limit theorem
for the pre-averaged log-price process under dependent noise in our general setting.
\begin{theorem}%[Stochastic Convergence]
\label{thm:consistency}
Assume that the efficient log-price follows \eqref{eq:Eff_Price_Ito_Diffu},
the observations follow~\eqref{eq:Y^n=X^n+U^n},
and the noise process satisfies Assumption~\ref{assumption:dependent_noise}.
%and that there exists an $\varepsilon>0$ such that $\expect{|U|^{2r+\varepsilon}}<\infty$.
Then, for any even integer $r\geq 2$,
\begin{equation}
\Pbv(Y,r)_n\Pconverge \Pbv(Y,r) : =\frac{\mu_r}{2c}\int_{0}^{1}\myp{\frac{2c}{3}\sigma^2_s + \frac{2}{c}\sigma^2_U}^{\frac{r}{2}}\diff s,
\label{eq:LLN}
\end{equation}
where $\sigma^2_U$ is defined in~\eqref{eq:sigma2_U}
and $\mu_r = \expect{Z^r}$ for a standard normal random variable $Z$.
\end{theorem}
\begin{proof}
See Appendix~\ref{appendix:thm_consistency}.
\end{proof}
\begin{corollary}\label{corollary:consistency_IV}
Under the assumptions of Proposition~\ref{prop:consistent_nonpar_sigma2U} and Theorem~\ref{thm:consistency},
we have the following consistency result for the integrated volatility: %for any $r^*_2+l^*_2 = 2$
%		\begin{equation}
%		\label{eq:consistency_SV}
%		{\rm \widehat{IV}}^{(q)}_n(r^*_2,l^*_2):=3\myp{\frac{\Pbv(Y,r^*_2,l^*_2)_n}{\mu_{r^*_2}\mu_{l^*_2}} - \frac{\widehat{\sigma^2_U(q)}_n}{c^2}}\Pconverge \int_{0}^{1}\sigma^2_s\diff s.
%		\end{equation}
%	With a nonparametric estimator of $\sigma^2_U$, we have
\begin{equation}
{\rm \widehat{IV}}_n:=3\myp{\Pbv(Y,2)_n - \frac{\widehat{\sigma^2_U}}{c^2}}\Pconverge \int_{0}^{1}\sigma^2_s\diff s,
\label{eq:consistency_SV_nonpar}
\end{equation}
where $\widehat{\sigma^2_U}$ is defined in~\eqref{eq:consistent_nonpar_sigma2U}.
\end{corollary}

\subsection{Asymptotic theory: The central limit theorem}

\begin{theorem}
\label{thm:CLT}
Assume that the efficient log-price follows \eqref{eq:Eff_Price_Ito_Diffu},
the observations follow~\eqref{eq:Y^n=X^n+U^n},
and the noise process satisfies Assumption~\ref{assumption:dependent_noise}.
%and that there exists an $\varepsilon>0$ such that $\expect{|U|^{8+\varepsilon}}<\infty$.
Furthermore, assume that the process $\sigma$ is a continuous It\^o semimartingale, and the assumptions of Proposition~\ref{prop:consistent_nonpar_sigma2U} hold with $v>4$.
Then,
\begin{equation}
n^{1/4}\myp{{\rm \widehat{IV}}_n-\int_{0}^{1}\sigma^2_s\diff s}\LsConverge \int_{0}^{1}\myp{2\sqrt{c}\sigma^2_s + \frac{6\sigma^2_U}{c^{3/2}}}\diff W_s',
\label{eq:LS}
\end{equation}
where $\LsConverge$ denotes stable convergence in law
and where $W'$ is a standard Wiener process independent of $\mathcal{F}$.
Moreover, letting $\tau_n^2 := 6\Pbv(Y,4)_n$, we have that
\begin{align}
\frac{n^{1/4}\myp{{\rm \widehat{IV}}_n-\int_{0}^{1}\sigma^2_s\diff s}}{\tau_n}
\label{eq:stdGau}
\end{align}
converges stably in law to a standard normal random variable, which is independent of $\mathcal{F}$.
%The convergence rate $n^{1/4}$ is optimal.
\end{theorem}
\begin{proof}
See Appendix~\ref{appendix:thm_CLT}.
\end{proof}
\begin{remark}
\label{rmk:how_to_choose_c}
The limit result in \eqref{eq:LS} provides a simple rule to select $c$ conditional on the volatility path: $c$ can be chosen to minimize the asymptotic variance. The optimal $c$ thus obtained is given by
\begin{equation}
c^* = 3\sqrt{\frac{\sigma^2_U}{\int_{0}^{1}\sigma^2_s\diff s}}.
\label{eq:how_to_choose_c}
\end{equation}
This result is intuitive:
if the noise-to-signal ratio is large, we should pick a large $c$,
hence include more observations in a local pre-averaging window to reduce the noise effect.
With typical noise-to-signal ratios that range from $10^{-2}$ to $10^{-4}$ as encountered in practice,
the optimal $c^*\in [0.03,0.3]$.
In our simulation and empirical studies, we throughout fix $c=0.2$.
%In practice, we can initially select some moderate $c$, e.g., $c=0.1$,
%to obtain some prior information on $\sigma^2_U$ and $\int_{0}^{1}\sigma^2_s\diff s$.
%Next, we can update the choice of $c$ accordingly.
\end{remark}

\section{Two-Step Estimators and Beyond}
\label{sec:two-step estimators}
In this section, we present our two-step estimators of the integrated volatility and the second moments of noise
based on both our asymptotic theory and finite sample analysis.% finite sample bias.
	
We observe from Corollary~\ref{corollary:consistency_IV} that the second moments of noise contribute to an \emph{asymptotic bias} in the estimation of the integrated volatility.
But our finite sample analysis indicates that we need an estimator of the integrated volatility to correct the \emph{finite sample bias} when estimating the second moments of noise.
Our two-step estimators are specifically designed for the purpose of correcting the ``interlocked'' bias.
	
In the first step, we ignore the dependence in noise and estimate the variance of noise by realized volatility.
Hence, our first-step estimators of the second moments of noise are given by
\begin{equation}
\widehat{\var{U}}_{\rm step1} := \widehat{\RV{Y,Y}}_n(1);\quad\widehat{\gamma(j)}_{\rm step1} := 0;\quad\widehat{\sigma}^2_{U,{\rm step1}} := \widehat{\RV{Y,Y}}_n(1).
\end{equation}
Next, we proceed with the pre-averaging method to obtain the first-step estimator of the integrated volatility:
\begin{equation}
\widehat{\rm IV}_{\rm step1}:=3\myp{\Pbv(Y,2)_n -\frac{\widehat{\sigma}^2_{U,{\rm step1}}}{c^2}}.
\label{eq:1stStepIV}
\end{equation}
	
To initiate the second step, we first replace $\hat{\sigma}^2$ by $\widehat{\rm IV}_{\rm step1}$ in~\eqref{eq:RV_SSBC} and~\eqref{eq:var_noise_dependent_SSBC} and obtain the second-step estimators of the variance and covariances of noise as follows:%\footnote{Note
%that $j_n$ and $j_n^*$ are different quantities.
%They appear in Proposition~\ref{prop:consistent_estimate_var_U_dependence} and Proposition~\ref{prop:consistent_nonpar_sigma2U}, respectively.}
\begin{align}
	\widehat{\RV{Y,Y}}_{\rm step2}(j) & := \widehat{\RV{Y,Y}}_n(j)-  \frac{j\widehat{\rm IV}_{\rm step1}}{2(n-j+1)};\\
	\widehat{\var{U}}_{\rm step2}  & := \widehat{\var{U}}_n -  \frac{j_n\widehat{\rm IV}_{\rm step1} }{2(n-j_n+1)};\label{eq:2ndStepVarU}\\
	\widehat{\gamma(j)}_{\rm step2}  & := \widehat{\var{U}}_{\rm step2} - \widehat{\RV{Y,Y}}_{\rm step2}(j);\label{eq:2ndStepGammaj}\\
	\widehat{\sigma}^2_{U,\rm step2} & := \widehat{\var{U}}_{\rm step2} + 2\sum_{j=1}^{i_n} \widehat{\gamma(j)}_{\rm step2}\label{eq:2ndStepSigma2U}.
\end{align}
Then, the second-step estimator of the integrated volatility is given by
\begin{equation}
\widehat{\rm IV}_{\rm step2}:=3\myp{\Pbv(Y,2)_n -\frac{\widehat{\sigma}^2_{U,{\rm step2}}}{c^2}}.
\label{eq:2ndStepIV}
\end{equation}

The asymptotic properties of the two-step estimators are inherited
from the asymptotic properties derived in the previous section.
Of course, one can iterate beyond the two steps to obtain $k$-step estimators, for example, $\widehat{\rm IV}_{\rm step3}$.
The next section will present simulation evidence to compare the performances of the proposed estimators.
% $k$-step estimators.
%But
As the results in the following section reveal,
the two-step estimators already perform very well.
%	and obtain a bias corrected $\widehat{\sigma^2_U}$
	
%We propose two-step estimators to correct the bias. In the first step, we obtain some reasonable estimates of the bias term
%--- this is crucial since in real applications the amount of data we work on is always finite.

\section{Simulation Study}
\label{sec:simulation}

\subsection{Simulation design}
%\subsection{Model setup}

We consider an autoregressive noise process $U$ %described in~\cite{Ait-Sahalia2011DependentNoise} and
given by the following dynamics:
\begin{equation}
U_{t} = V_{t} + \epsilon_{t},
\label{eq:ARnoise}
\end{equation}
where $V$ is centered i.i.d. Gaussian and $\epsilon$ is an AR(1) process with first-order coefficient $\rho$, $|\rho|<1$.
The processes $V$ and $\epsilon$ are assumed to be statistically independent.
As benchmark parameters, we use the GMM estimates of the noise parameters from~\cite{Ait-Sahalia2011DependentNoise}
given by $\expect{V^2} = 2.9\times 10^{-8}$, $\expect{\epsilon^2} = 4.3\times 10^{-8}$, and $\rho = -0.7$.
We also allow for different dependence structures by varying our choice of $\rho$.
Furthermore, the efficient log-price $X$ is assumed to follow an Ornstein-Uhlenbeck process:
\begin{equation}
\diff X_t  = -\delta(X_t-\mu) \diff t + \sigma\diff W_t,\qquad\:\delta>0,\ \sigma>0.
\label{eq:OUprice}
\end{equation}
%\subsection{Parametric settings}
We set $\sigma^2=6\times 10^{-5}$, $\delta = 0.5$, and $\mu = 1.6$,
and assume the processes $X$ and $U$ to be mutually independent.
The signal-to-noise ratio induced by this model for $Y_{t}=X_{t}+U_{t}$ is realistic,
according to empirical studies; see, e.g.,~\cite{bandi2006separating,bandi2008microstructure}.
For all the experiments in this section, we conduct $1\mathord{,}000$ simulations.
Each simulated sample consists of $23\mathord{,}400$ observations in our fixed time interval $[0,1]$
representing one trading day of data sampled at the 1-sec time scale with 6.5 trading hours per day.
The ultra high-frequency case with sampling at the 0.05-sec time scale is also considered.
We take $c=0.2$.

%	We simulate $23,400$ equally spaced observations, corresponding to data on second time scale with 6.5 trading hours per day.
	%\subsection{Simulation results}

\subsection{Realized volatility estimators of the second moments of noise}
%	\subsubsection{Consistency and finite sample bias correction of realized volatility}

To get a first impression of the properties of our estimator $\widehat{\RV{Y,Y}}_n(j)$ defined in~\eqref{eq:Pconverge_jth_RV},
we plot %a ``sample path'' of
$\widehat{\RV{Y,Y}}_n(j)$ against the number of lags $j$ in Figure~\ref{fig:RVpath}.
In addition to $\widehat{\RV{Y,Y}}_n(j)$, we also plot the bias adjusted version $\widehat{\RV{Y,Y}}_n^{(\rm adj)}(j)$ defined in~\eqref{eq:RV_SSBC},
in which we employ three ``approximations'' to the integrated volatility that $\widehat{\RV{Y,Y}}_n^{(\rm adj)}(j)$ depends on:
$\hat{\sigma}^2_H = 1.2\sigma^2$, $\hat{\sigma}^2_M = \sigma^2$, and $\hat{\sigma}^2_L = 0.8\sigma^2$.
% The exact yet infeasible finite sample bias correction with $\hat{\sigma}^2_{M} = \sigma^2$ is also considered.
Figure~\ref{fig:RVpath} shows that a prominent feature of our realized volatility estimator $\widehat{\RV{Y,Y}}_n(j)$ is
that it deviates from its stochastic limit $\var{U}-\gamma{(j)}$ almost linearly in the number of lags $j$,
as predicted by Proposition~\ref{prop:Finite_Sample_Bias_Correction}.
The deviation, induced by the finite sample bias,
can be  corrected  to a large extent %largely
when only rough ``estimates'' of the integrated volatility are available.
In the ideal but infeasible situation that we know %exactly
 the true volatility ($\hat{\sigma}^2_M = \sigma^2$),
the bias corrected estimators almost perfectly match the underlying true values.
	%The bottom panel of Figure~\ref{fig:RVpath} shows that the finite sample bias is reduced when the sampling frequency increases.
	% that motivates our finite sample bias corrected estimators. By allowing for a wide range of errors in estimating the integrated volatility, which plays a key role to correct the bias, we show the greatly improved accuracy of the bias corrected estimators. In the ideal but infeasible situation that we know exactly the true volatility, the bias corrected estimator almost perfectly fits the underlying true values. In the bottom panel of , we observe that the deviations of the estimators from its stochastic limits are reduced when the data frequency increases, and this is expected from our asymptotic theory.   	
	
%	\subsubsection{Estimating second moments of noise by realized volatility and local averaging estimators in finite samples}
Next, we estimate the second moments of noise by our realized volatility estimators (RV)
and, for comparison purposes, by the local averaging estimators (LA)
proposed by~\cite{jacod2013StatisticalPropertyMMN}.
We demonstrate the importance of the finite sample bias correction to obtain accurate estimates,
and this applies to both estimators.\footnote{\label{fn:lafs}The finite sample bias corrected local averaging estimators
of the noise covariances are given by
\begin{equation*}
\widehat{R}(j)_n = \frac{1}{n}U((0,j))_n - \frac{K_n}{n}\myp{\frac{4}{3}\hat{\sigma}^2},
\end{equation*}
where $U((0,j))_n/n$ is the local averaging estimator of the $j$-th covariance without bias correction
and $\hat{\sigma}^2$ is an estimator of the integrated volatility;
see~\cite{jacod2013StatisticalPropertyMMN} for more details.
While~\cite{jacod2013StatisticalPropertyMMN} provide a finite sample bias correction when
developing their local averaging estimators of noise covariances,
they don't consider the feedback between, and unified treatment of, asymptotic and finite sample biases,
which is a key interest in this paper.}
%the realized volatility estimators and the local averaging estimators.% We will show that our realized volatility estimators perform better in a finite sample. The estimators are also quite robust to some tuning parameters. We consider the accuracy and robustness the great advantages over the local averaging method\footnote{In terms of computing time, our realize volatility estimator is more efficient. This is another advantage for practical purpose.}.
In Figure~\ref{fig:RVvsLA}, we plot the means of the autocorrelations of noise estimated by RV and LA based on $1\mathord{,}000$ simulations.
In the top panel we plot the estimators without finite sample bias correction
and we plot the estimators with finite sample bias correction in the bottom panel, %Recall that both methods require an estimate of the integrated volatility to correct the finite sample bias. In this experiment, we take an extreme step to make the bias corrections ---
in which we use the true $\sigma^2$ %instead of any approximation/estimator
to make the bias correction.
We will analyze the case in which $\sigma^2$ is estimated in the next subsection.
	
We observe that both estimators (RV and LA) perform poorly without finite sample bias correction.
In particular, the noise autocorrelations estimated by the LA estimators decay slowly and hover above 0 up to 25 lags,
from which we might conclude that the noise exhibits strong and long memory dependence,
while the underlying noise is, in fact, only weakly dependent.
However, both estimators perform well after the finite sample bias correction.
In Figure~\ref{fig:RVvsLAWithCIkn6}, we also plot the 95\% simulated confidence intervals of the two bias corrected estimators.
In terms of mean squared errors, both estimators, after bias correction, yield accurate estimates.
%of the autocorrelations of noise. %Further comparison of RV and LA can be found in Appendix~\ref{sec:TuningPara}.
We note that the results for our RV estimator are robust to the choice of $j_{n}$.
	
Figures~\ref{fig:RVpath}-\ref{fig:RVvsLAWithCIkn6} reveal that the finite sample bias correction is crucial
to obtain reliable estimates of noise moments.
The key ingredient of this correction, however, is (an estimate of) the integrated volatility.
Yet, to obtain an estimate of the integrated volatility,
we need to estimate the second moments of noise first --- whence the feedback loop of bias corrections.
This is where our two-step estimators come into play.
	
%	The important lesson we learn from the simulation study is that one should always take into account the finite sample bias to get a reliable estimate of variance or covariance of noise.  To make a finite sample bias correction, we need an estimate of the integrated volatility. Therefore the estimators of integrated volatility and estimators of variance and covariances of noise will reinforce each other: better estimates of the integrated volatility will provide more accurate finite sample bias correction, resulting better estimates of variance and covariances of noise, which in turn will provide a more accurate estimate of the integrated volatility. And this is the basic principle of our two-step estimators of the second moments of noise and the integrated volatility.

\subsection{Two-step estimators of integrated volatility
 and beyond
}\label{subsec:simu_est_IV}

In this subsection, we examine the performance of our two-step estimators of integrated volatility.
We will compare $\widehat{\rm IV}_{\rm step1}$ to $\widehat{\rm IV}_{\rm step2}$
(cf. \eqref{eq:1stStepIV} and \eqref{eq:2ndStepIV})
to assess the gained accuracy by dropping the possibly misspecified assumption of independent noise,
and compare $\widehat{\rm IV}_{n}$ to $\widehat{\rm IV}_{\rm step2}$
(cf. \eqref{eq:LS} and \eqref{eq:2ndStepIV})
to assess the accuracy gains from the unified treatment of asymptotic and finite sample biases.
We also illustrate the increased accuracy achieved by iterating one more step, yielding the estimator $\widehat{\rm IV}_{\rm step3}$.
	
In Table~\ref{tab:IV_est_delta=1s},
we report the means of our estimators, with standard deviations between parentheses,
based on $1\mathord{,}000$ simulations.\footnote{The numbers are multiplied by $10^5$.}
Throughout this subsection, $j_{n}$ is fixed at $20$.
Upon comparing the first and the third rows,
we observe the important advantage of our two-step estimators over the pre-averaging method that assumes independent noise,
since %where
 our estimators yield strongly improved accuracy.
Furthermore, a comparison between the results in the second and third rows leads to a striking conclusion:
ignoring the finite sample bias yields even more inaccuracy than ignoring the dependence in noise!
Thus one should be cautious in applying estimators without appropriate bias corrections
even with data on a 1-sec time scale.
The ``cost'' of applying our two-step estimators is the slightly larger standard deviations they induce.
The increased uncertainty is introduced by correcting the ``interlocked'' bias.
However, the reduction in bias strictly dominates the slight increase in standard deviations when noise is dependent.
Therefore, the two-step estimator has smaller mean-squared errors than the other two estimators.
The last row of Table~\ref{tab:IV_est_delta=1s} shows that another iteration of bias corrections
yields even more accurate estimates,
although the respective standard deviations increase slightly.
	%estimating the integrated volatility, which corrects the finite sample bias in the two-step estimators.
	%The advantage of the two-step estimators of the integrated volatility is quite significant.
%	

In Table~\ref{tab:IV_est_delta=0.1s}, we replicate the results of Table~\ref{tab:IV_est_delta=1s} but now with higher data frequency
(sampling at the 0.05-sec time scale).
We clearly observe the inconsistency caused by the misspecification of the dependence structure in noise
embedded in $\widehat{\rm IV}_{\rm step1}$ in the first row.
The improved accuracy achieved by the estimator $\widehat{\rm IV}_{n}$ in the second row
compared to the estimator $\widehat{\rm IV}_{\rm step1}$ in the first row confirms our asymptotic theory.
However, interestingly we observe that, even with such ultra high-frequency data, the two-step estimator $\widehat{\rm IV}_{\rm step2}$ in the third row
still performs better than the other two estimators --- with smaller biases in most cases and only slightly larger standard deviations.
In this scenario, one more iteration of bias corrections leads to little improvement.

Our results remain qualitatively the same when we increase the variance of noise.
%,
%although the estimates naturally become somewhat less accurate.
The relative improvement due to the 2-step estimator is even more pronounced in this case
and a 3-step estimator may yield further improvements.
As another robustness check, we also changed the exponentiated Ornstein-Uhlenbeck process
for the efficient price process into a Geometric Brownian Motion.
This only impacts the third digits of the estimates and the second digits of the standard deviations reported above.

To numerically ``verify'' the central limit theorem,
we plot the quantiles of the normalized estimators $\frac{n^{1/4}\myp{\widehat{\rm IV}_n-\int_{0}^{1}\sigma^2_s\diff s}}{\tau_n}$,
see \eqref{eq:stdGau},
and the bias corrected version $\frac{n^{1/4}\myp{\widehat{\rm IV}_{\rm step2}-\int_{0}^{1}\sigma^2_s\diff s}}{\tau_n}$
against standard normal quantiles in Figure~\ref{fig:CLT_QQplots}.
We observe that the limit distribution established in Theorem~\ref{thm:CLT} is clearly verified.

In Appendix~\ref{sec:SVsimu}, we provide additional Monte Carlo simulation evidence
based on \textit{stochastic volatility} models,
using realistic parameters motivated by our empirical studies,
and we find that our two-step estimator retains its advantage over the other two estimators,
$\widehat{\rm IV}_{\rm step1}$ and $\widehat{\rm IV}_{n}$.

\section{Empirical Study}\label{sec:EmpiricalStudy}

\subsection{Data description}

We analyze the NYSE TAQ transaction prices of Citigroup (trading symbol: C) over the month January 2011.
We discard all transactions before 9:30 and after 16:00.
We retain a total of $4\mathord{,}933\mathord{,}059$ transactions over 20 trading days,
thus on average 10.5 observations per second.
%While it is not uncommon practice to perform data cleaning to remove some unpleasant features of the data before estimation,
The estimation is first performed on the full sample, and then on subsamples obtained by different sampling schemes. We demonstrate how the sampling methods affect the properties of the noise,
and thus affect the estimation of the integrated volatility.
Throughout this section, the tuning parameter of the RV estimator is fixed at $j_n=30$ and $c=0.2$.
%	It is worth mentioning that we do not perform any data preprocessing to eliminate
	%The same dataset is used in the empirical study in~\cite{jacod2013StatisticalPropertyMMN}.
	
\subsection{Estimating the second moments of noise}\label{subsec:estimate_2nd_moments_noise_Citi}

We estimate the $j$-th autocovariance and autocorrelation of microstructure noise with $j=0,1,\dots,30$ by three estimators:
our realized volatility (RV) estimators in~\eqref{eq:consistent_estimate_var_U_dependent} and~\eqref{eq:gamma(j)_hat},
the local averaging (LA) estimators proposed by~\cite{jacod2013StatisticalPropertyMMN},
and the bias corrected realized volatility (BCRV) estimators in~\eqref{eq:2ndStepVarU} and~\eqref{eq:2ndStepGammaj}.
We perform the estimation over each trading day and end up with 20 estimates (of the 30 lags of autocovariances or autocorrelations) for each estimator.
In Figure~\ref{fig:RV_LA_2step_fulldata} we plot the average of the 20 estimates (over the month)
as well as the approximated confidence intervals that are two sample standard deviations away from the mean.
	
% To present the estimation results, we first compute the mean and standard deviation of the 20 estimates. Then we plot the mean and in .

We observe that the three estimators yield quite close %and accurate
estimates by virtue of the high data frequency.
Noise in this sample tends to be positively autocorrelated --- %the autocorrelations are positive up to 10 lags,
with the BCRV estimators yielding the fastest decay.
%Empirically
This is consistent with the finding that the arrivals of buy and sell orders are positively autocorrelated, see~\cite{hasbrouck1987order}.
This corresponds to the trading practice that informed traders split their orders over (a short period of) time and trade on one side of the market, rendering continuation in their orders.
	
We emphasize that
the finite sample bias can be much more pronounced %is not in general minor
than what we observe in Figure~\ref{fig:RV_LA_2step_fulldata},
even if we perform estimation on a full transaction data sample.
In Appendix~\ref{sec:GE_Empirical}, we analyze the transaction prices of General Electric (GE)
and show that, when the data frequency is very high,
the finite sample bias correction is particularly important % essential
when the noise-to-signal ratio is very small
(recall Remark~\ref{rmk:why_correct_bias}).

\subsection{Estimating the integrated volatility}\label{subsec:IV_Citi}

Turning to the estimation of the integrated volatility,
we mimic our simulation experiments and study three estimators: $\widehat{\rm IV}_{\rm step1}$, $\widehat{\rm IV}_{n}$, and $\widehat{\rm IV}_{\rm step2}$.
In the top panel of Figure~\ref{fig:CitiIVsFullData}, we plot the three estimators of the integrated volatility for each trading day.
We note that the estimator $\widehat{\rm IV}_{n}$ and the two-step estimator $\widehat{\rm IV}_{\rm step2}$ yield quite close results.
However, the estimator $\widehat{\rm IV}_{\rm step1}$, which ignores the dependence in noise, yields very different estimates,
and the differences are one-sided %persistent
--- $\widehat{\rm IV}_{\rm step1}$ yields higher estimates over each trading day.
Moreover, the differences are statistically significant by virtue of Theorem~\ref{thm:CLT} ---
19 out of the 20 estimates fall outside of the 95\% confidence intervals,
as the bottom panel of Figure~\ref{fig:CitiIVsFullData} reveals.
	
%	By virtue of Theorem~\ref{thm:CLT} we can
	
%		To examine the significance of the differences, we plot the asymptotic confidence intervals of $\widehat{\rm IV}_{n}$ based on the limit distribution developed in Theorem~\ref{thm:CLT}. We find that the differences are statistically significant

\subsection{Decaying rate of autocorrelation}

Figure~\ref{fig:RV_LA_2step_fulldata} shows that the positive autocorrelations of noise drop to zero rapidly.
To assess the rate of decay, we perform a logarithmic transformation of the autocorrelations estimated by BCRV.\footnote{We restrict attention to the lags up to $j=15$. The logarithmic autocorrelations at higher lags are very volatile since the autocorrelations are close to zero.}
In the top panel of Figure~\ref{fig:CitiLogCors}, we plot the logarithmic autocorrelations for each trading day,
revealing clear support for a linear trend.
To better visualize the linear relationship, we plot the means of the logarithmic autocorrelations over the 20 trading days
and fit a regression line to it; see the bottom panel of Figure~\ref{fig:CitiLogCors}.
The nearly perfect fit indicates that the logarithmic autocorrelation is approximately a linear function of the number of lags,
i.e., the autocorrelation function is decaying at an exponential rate.\footnote{The autocorrelation decay rate would be slower without unified treatment of the bias corrections,
which may explain the %relatively strong
polynomial dependence in noise found in~\cite{jacod2013StatisticalPropertyMMN}
and questioned by these authors themselves.}
	
%	\subsection{Autocorrelation function of log-returns}
%	Recall that if noise is independent, the log-returns will exhibit an MA(1) pattern, and this is commonly perceived as the support of i.i.d. noise. Obviously noise is not independent in this empirical study. It is interesting to examine the patterns of the log-returns.
%	
%	In Figure~\ref{fig:CitiACFLogRet.tex}, we plot the sample autocorrelation function of Citi log-returns for January, 2011.
%		\begin{figure}[h]
%			\centering
%			\input{CitiACFLogRet.tex}
%			\caption{Autocorrelation function of Citi log-returns in transaction time for January, 2011.}
%			\label{fig:CitiACFLogRet.tex}
%		\end{figure}
		
%\subsection{Regular sampling and tick time sampling}
\subsection{Robustness check ---  estimation under other sampling schemes}

It is interesting to analyze how our estimators perform when the data is sampled at different time scales.
In this section, we consider two alternative (sub)sampling schemes:
regular time sampling and tick time sampling (recall Remark~\ref{rmk:Sampling_Schemes} for details on the sampling schemes).

\subsubsection{Regular time sampling}\label{subsubsec:regular_sampling}
%we perform estimation in a regularly spaced subsample of the transaction data.% In Appendix~\ref{sec:TickTimeEmpirical}, we report the estimation results based on tick time sampling.
%We consider two popular sampling schemes: regular sampling and tick time sampling.
	
The prices in this sample are recorded on a 1-second time scale.
If there were multiple prices in a second, we select the first one; and we do not record a price if there is no transaction in a second.
We end up with $21\mathord{,}691$ observations on average per trading day. %This subsample accounts for less than 10\% of the complete sample.
%	 We already investigate this problem by simulation studies, in which we show that without bias correction the LA and RV estimators may lead to misleading conclusions about the dependence of noise when the sampling frequency is not that extremely high.
%	
%	We consider a subsample of the Citi transaction data. In particular, we record prices at 1-second time scale. Thus, on average we neglect 9 transactions out of 10. Since noise is assumed to be dependent in transaction time, we expect noise to be independent in this subsample.
Figure~\ref{fig:RV_LA_2step_seconddata} is analogous to Figure~\ref{fig:RV_LA_2step_fulldata}.
The three estimators, RV, LA, and BCRV, now produce very different patterns.
Both the RV and LA estimators indicate that noise is strongly autocorrelated in this subsample, even stronger than in the original full sample.
This would be counterintuitive since we eliminate more than 90\% of the full sample in a fairly random way ---
the elimination should if anything have weakened the serial dependence of noise in the remaining sample.
However, the estimates by BCRV reveal that in fact the noise is approximately uncorrelated ---
it is the finite sample bias that makes the autocorrelations of noise seem strong and persistent
if not taken into account.
	
%The finite sample bias in estimating the second moments of noise also affects the estimation of the integrated volatility. However, our two-step estimator of the integrated volatility is quite robust.

If the noise is close to being independent,
$\widehat{\rm IV}_{\rm step1}$, which assumes i.i.d. noise, would be a valid estimator of the integrated volatility.
An alternative estimator, e.g., $\widehat{\rm IV}_{\rm step2}$ or $\widehat{\rm IV}_{n}$,
would be robust if it delivered similar estimates.
	 %Turning to the estimation of the integrated volatility, we still employ the three estimators $\widehat{\rm IV}_{\rm step1},\widehat{\rm IV}_{\rm step2}$ and $\widehat{\rm IV}_{n}$ to estimate the integrated volatility in this subsample.
In the top panel of Figure~\ref{fig:CitiIVsSecondData}, we observe that $\widehat{\rm IV}_{\rm step1}$ and $\widehat{\rm IV}_{\rm step2}$
yield virtually identical estimates.
The estimator $\widehat{\rm IV}_{n}$, however, yields lower estimates on each trading day.
If we rely on the asymptotic theory only, we would conclude that the estimates by $\widehat{\rm IV}_{\rm step1}$ (or $\widehat{\rm IV}_{\rm step2}$)
are significantly higher than those by $\widehat{\rm IV}_{n}$ in the statistical sense ---
all the 20 estimates by $\widehat{\rm IV}_{\rm step1}$ (or $\widehat{\rm IV}_{\rm step2}$) are
outside the 95\% asymptotic confidence intervals of $\widehat{\rm IV}_{n}$, as we observe from the bottom panel of Figure~\ref{fig:CitiIVsSecondData}.
We conclude that Figures~\ref{fig:CitiIVsFullData} and~\ref{fig:CitiIVsSecondData} jointly reveal the importance of our multi-step approach. %$\widehat{\rm IV}_{\rm step2}$.
Indeed, $\widehat{\rm IV}_{\rm step1}$ shows unreliable behaviour %``fails''
 in Figure~\ref{fig:CitiIVsFullData},
while $\widehat{\rm IV}_{n}$ shows unreliable behaviour %``fails''
 in Figure~\ref{fig:CitiIVsSecondData}.

\subsubsection{Tick time sampling}		
\label{subsec:tick_sampling}
%In this section, we report our estimation results based on a subsample of the transaction prices of Citigroup for January, 2011.
In a tick time sample,  prices are collected with each price change, i.e., all zero returns are suppressed, see, e.g.,~\cite{da2017moving},~\cite{Ait-Sahalia2011DependentNoise},~\cite{griffin2008sampling},~\cite{kalnina2011subsampling} and~\cite{zhou1996high}.
For the Citigroup transaction data, 70\% of the returns are zero.
The corresponding average number of prices per second in our tick time sample is 3.2.
Figure~\ref{fig:RV_LA_2step_tickdata} shows that the microstructure noise has a different dependence pattern in the tick time sample ---
its autocorrelation function is alternating. Masked by alternating noise, the observed returns at tick time have a similar pattern; see~\cite{Ait-Sahalia2011DependentNoise} and~\cite{griffin2008sampling}.
This dependence structure of noise is perceived to be due to the discreteness of price changes,
irrespective of the distributional features of noise in the original transactions or quotes data.
%\footnote{To understand this alternating pattern,
%let's assume that the efficient price process is a Brownian motion.
%In a short period of time, the efficient price process changes only little.
%The local variation of the observed price is mainly attributable to microstructure noise, which is also centered at 0.
%Suppose that at tick times $t_i$ and $t_{i+1}$, the efficient price of a stock is fixed at $1.652$.
%Now assume that at time $t_i$ the noise is positive and equal to $0.003$.
%With a minimal tick size of $0.01$ the rounded observed price at time $t_i$ is $1.66$.
%At tick time $t_{i+1}$ conditional on price changes, the most likely price is either $1.67$ or $1.65$ (jumping to other prices requires an extreme positive or negative noise, which occurs with much smaller probability), which occurs with probability $\prob{0.013\leq U_{t_{i+1}}< 0.023}$ and $\prob{-0.007\leq U_{t_{i+1}} <0.003}$, respectively.
%The latter probability usually dominates the former as the tick size is much larger than the (conditional) variance of noise, irrespective of the autocorrelation pattern of noise.
%This explains the alternating pattern of the autocorrelation function of noise in tick time.
%%In Section~\ref{sec:TickTimeSimu}, we provide additional Monte Carlo simulation evidence of the dependence pattern of noise under tick time sampling.
%}

Interestingly, Figure~\ref{fig:CitiIVsTickData} shows that the three estimators of the integrated volatility,
$\widehat{\rm IV}_{\rm step1}$, $\widehat{\rm IV}_{\rm step2}$, and $\widehat{\rm IV}_{n}$, remain close.
It is not surprising to see a close fit of $\widehat{\rm IV}_{\rm step2}$ and $\widehat{\rm IV}_n$ since the data frequency is still quite high.
By contrast, it is not directly obvious why $\widehat{\rm IV}_{\rm step1}$ and $ \widehat{\rm IV}_{\rm step2}$ deliver almost identical estimates,
given the fact that the dependence of noise in this tick time sample is drastically different from i.i.d. noise.
However, a clue is provided by the observation that negatively autocorrelated noise has less impact on the estimation of the integrated volatility,
as the high-order alternating autocovariances partially cancel out, thus
contributing %contribute
less to the asymptotic bias $\sigma^2_U$.\footnote{For a tractable analysis, one may consider AR(1) noise processes.
Let $\rho\in(0,1)$ be the absolute value of the AR(1) coefficient.
When the noise is positively autocorrelated, the asymptotic bias $\sigma^2_U$ corrected by $\widehat{\rm IV}_{\rm step1}$ and $\widehat{\rm IV}_{\rm step2}$ is  $(1-\rho)\var{U}$ and $\frac{1+\rho}{1-\rho}\var{U}$, respectively;
when the noise is negatively autocorrelated, it is $(1+\rho)\var{U}$ and $\frac{1-\rho}{1+\rho}\var{U}$.
Consider $\rho = 0.8$.
Then, $(1-\rho)\var{U} = 0.2\var{U}$ and $\frac{1+\rho}{1-\rho}\var{U} = 9\var{U}$
while $(1+\rho)\var{U} = 1.8\var{U}$ and $\frac{1-\rho}{1+\rho}\var{U} = \frac{1}{9}\var{U}$.
Therefore, the difference in the asymptotic bias is smaller when the noise is negatively autocorrelated;
consequently, the integrated volatility estimates by $\widehat{\rm IV}_{\rm step1}$ and $\widehat{\rm IV}_{\rm step2}$ are close. %
See also Tables~\ref{tab:IV_est_delta=1s} and~\ref{tab:IV_est_delta=0.1s} in our simulation study.}

%	For misspecified i.i.d. noise, positively (with AR(1) coefficient $\rho$) and negatively (with AR(1) coefficient $-\rho$) autocorrelated noise, the asymptotic bias $\sigma^2_U$ (recall~\eqref{eq:sigma2_U}) are $\var{U}-\gamma(1),\frac{1+\rho}{1-\rho}\var{U}$ and $\frac{1-\rho}{1+\rho}\var{U}$, respectively. Compare to the simulation study, we find that $\rho$ should be close to 1 to fit the empirical data. Hence we let $\rho=0.8$ in the sequel. If noise is positively autocorrelated, the asymptotic bias (generated by noise) corrected by $\widehat{\rm IV}_{\rm step2}$ and $\widehat{\rm IV}_{\rm step1}$ are $9\var{U}$ and $0.2\var{U}$; when noise is negatively autocorrelated, the two terms become $\frac{1}{9}\var{U}$ and $1.8\var{U}$, respectively.
%Therefore we get an interesting conclusion that is different from~\cite{griffin2008sampling}:  popular de-noise methods\footnote{\cite{Ait-Sahalia2011DependentNoise} show that TSRV and MSRV are robust in tick time samples.} that require time-independent noise are still quite robust if the integrated volatility is estimated in tick time samples, though noise obviously deviates from the independence assumption.	

\subsection{Economic interpretation and empirical implication}	

The dependence structure of microstructure noise is complex,
and depends on the sampling scheme.
In an original transaction data sample,
noise is likely to be positively autocorrelated as a result of various trading practices that entail continuation in order flows.
The dependence of noise can be reduced by sampling sparsely, say, every few (or more) seconds as we show in Section~\ref{subsubsec:regular_sampling};
noise is close to independent in such sparse subsamples. 	
If, however, we remove all zero returns, thus sample in tick time, noise typically exhibits an alternating autocorrelogram.

Microstructure theories can provide some intuitive economic interpretations of the dynamic properties of microstructure noise recovered in this paper. The positive autocorrelation function displayed in Figure~\ref{fig:RV_LA_2step_fulldata} is consistent with the findings in~\cite{hasbrouck1987order},~\cite{choi1988estimation} and~\cite{huang1997components} that explicitly model the probability of order reversal $\pi$ (or order continuation by $1-\pi$),\footnote{It is the probability that a buy (sell) order follows another sell (buy) order.} so that the deviation of transaction prices from fundamentals becomes an AR(1) process.
Fitting the autocorrelation function recovered by BCRV in Figure~\ref{fig:RV_LA_2step_fulldata} to that of an AR(1) model,
we obtain an estimate of the AR(1) coefficient equal to $\hat{\rho} = 0.75$
and the probability of order continuation is $1-\hat{\pi} = (1+\hat{\rho})/2 = 0.87$.
That is, the estimated probability that a buy (or sell) order follows another buy (or sell) order is 0.87.
In view of the extensive empirical results in~\cite{huang1997components} (see Table 5 therein), this is a reasonable estimate.

One possible interpretation of the positively autocorrelated order flows is that a large \emph{order} is often executed as a series of smaller \emph{trades} to reduce the price impact, or conducted against multiple \emph{trades} from stale limit orders.
However, such positive autocorrelation contradicts the prediction of inventory models,
in which market makers induce negatively autocorrelated order flows to stabilize %equilibrate
inventories; see~\cite{ho1981optimal}.
Consequently, according to inventory models the probability of order reversal would be $\pi>0.5$.
One remedy, suggested by~\cite{huang1997components},
is to collapse multiple \emph{trades} at the same price into one \emph{order},
which is exactly the tick time sampling scheme considered in Section~\ref{subsec:tick_sampling}.
Exploiting the estimates by BCRV presented in Figure~\ref{fig:RV_LA_2step_tickdata},
we obtain an estimate of the probability of order reversal equal to $\hat{\pi} = 0.84$,
which is very close to the average probability $0.87$ in~\cite{huang1997components}.
We emphasize that we recover these probabilities without any prior knowledge or estimates of the order flows.

The dependence structure of microstructure noise,
and hence the choice of sampling scheme,
affect the estimation of integrated volatility.
Popular de-noise methods that assume i.i.d. noise work reasonably well with relatively sparse regular time samples
or tick time samples.
However, this discards a substantial amount of the original transaction data.\footnote{To obtain the Citigroup tick time sample
and the 1-second regular time sample, we delete roughly 70\% and 90\% of the original transaction data, respectively.}
Instead, we can directly estimate the integrated volatility from the original transaction data
using our estimators that explicitly take the potential dependence in noise into account.

In our empirical study, we have also illustrated that bias corrections play
an essential role in recovering the statistical properties of noise
and in estimating the integrated volatility. %when various sampling schemes are used.
Our two-step estimators are specifically designed to conduct such bias corrections,
and have the advantage of being robust to different sampling schemes and frequencies.

\section{Conclusion}\label{sec:conclusion}

In high-frequency financial data
the efficient price is contaminated by microstructure noise,
which is usually assumed to be independently and identically distributed.
This simple distributional assumption is challenged by both microeconomic financial models and various %complex
empirical facts.
In this paper, we deviate from the i.i.d. assumption by allowing noise to be dependent in a general setting.
We then develop econometric tools to recover the dynamic properties of microstructure noise and design improved approaches for %study
 the estimation of the integrated volatility.

This paper makes four contributions.
First, it develops nonparametric estimators of the second moments of microstructure noise in a general setting.
Second, it provides a robust estimator of the integrated volatility,
without assuming serially independent noise.
Third, it reveals the importance of both asymptotic and finite sample bias analysis and develops simple
and readily implementable two-step %and three-step
estimators that are robust to the sampling frequency.
Empirically, it characterizes the dependence structures of noise in several popular sampling schemes and provides %derives
intuitive economic interpretations; it also investigates the impact of the dynamic properties of microstructure noise on integrated volatility estimation.

This paper thus introduces a robust and accurate method to effectively separate the two components of high-frequency financial data
--- the efficient price and microstructure noise.
The robustness lies in its flexibility to accommodate rich dependence structures of microstructure noise motivated by various economic models and trading practices, whereas the accuracy is achieved by the finite sample refinement.
As a result, we discover dynamic properties of microstructure noise consistent with microstructure theory
and obtain accurate volatility estimators that are robust to sampling schemes.

\section*{Acknowledgements}

We are very grateful
to Yacine A\"it-Sahalia, Federico Bandi, Peter Boswijk, Peter Reinhard Hansen, Siem Jan Koopman, Oliver Linton, and Xiye Yang
%and conference and seminar participants at the
%Conference on Financial Econometrics \& Empirical Asset Pricing in Lancester,
%the 10th International Conference on Computational and Financial Econometrics in Sevilla,
%the Tinbergen Institute, and the University of Amsterdam
for their comments and discussions on earlier versions of this paper.
%{\tt Matlab} code to implement the estimators developed in this paper
%is available from the authors upon request.
This research was funded in part by the Netherlands Organization for Scientific Research under grant NWO VIDI 2009 (Laeven).

%\clearpage
\bibliographystyle{econometrics}
\bibliography{reference}

\newpage

\section*{Tables and Figures}

\begin{table}[htb]
\centering
\begin{tabular}{l|ccccc}\hline\hline
	$\rho$                             &-0.7         &-0.3         &0            &0.3           & 0.7      \\\hline
%	$\widehat{\rm IV}_{\rm step1}$     &5.37 (0.45)  &5.60 (0.45)  &5.87 (0.44)  &6.20 (0.47)   & 7.36 (0.55)     \\
$\widehat{\rm IV}_{\rm step1}$  &5.53 (0.46)	  &5.74 (0.46)	 &5.98 (0.47)	   &6.39 (0.49)	&7.57 (0.56)\\
%	$\widehat{\rm IV}_{n}$        &2.89 (0.39)  &2.88 (0.40)  &2.90 (0.39)  &2.86 (0.42)   & 2.71 (0.49)        \\
$\widehat{\rm IV}_{n}$        &3.04 (0.40)  &3.02 (0.40)  &3.02 (0.41)  &3.04 (0.43)   & 2.91 (0.50)        \\
%	$\widehat{\rm IV}_{\rm step2}$     &5.56 (0.61)  &5.66 (0.62)  &5.82 (0.59)  &5.94 (0.64)   & 6.37 (0.75) \\\hline\hline	
$\widehat{\rm IV}_{\rm step2}$     &5.79 (0.61)  &5.87 (0.63)  &5.99 (0.63)  &6.23 (0.67)   & 6.67 (0.76) \\
$\widehat{\rm IV}_{\rm step3}$     &5.92 (0.70)  &5.93 (0.72)  &6.00 (0.72)  &6.13 (0.76)   & 6.22 (0.87) \\\hline\hline
\end{tabular}
%\begin{tabular}{l|ccccc}\hline\hline
%	$\rho$                             &-0.7         &-0.3         &0            &0.3           & 0.7      \\\hline
%	$\widehat{\rm IV}_{\rm step1}$     &5.38 (0.44)  &5.59 (0.46)  &5.86 (0.47)  &6.22 (0.44)   & 7.34 (0.54)     \\
%	$\widehat{\rm IV}_{n}$        &5.10 (0.44)  &4.88 (0.45)  &4.91 (0.46)  &4.88 (0.44)   & 4.89 (0.53)        \\
%	$\widehat{\rm IV}_{\rm step2}$     &5.96 (0.51)  &5.78 (0.53)  &5.85 (0.53)  &5.88 (0.51)   & 6.07 (0.61) \\\hline\hline	
%\end{tabular}
\caption{Estimation of the integrated volatility.
The numbers represent the means of the four estimators of integrated volatility,
$\widehat{\rm IV}_{\rm step1}$, $\widehat{\rm IV}_{n}$, $\widehat{\rm IV}_{\rm step2}$ and $\widehat{\rm IV}_{\rm step3}$,
based on $1\mathord{,}000$ simulations
with standard deviations between parentheses.
The true value of the integrated volatility is given by $\sigma^2 = 6\times 10^{-5}$.
All numbers in the table are multiplied by $10^5$.
We take $\Delta=1$ sec and the number of observations is $23\mathord{,}400$.
The tuning parameter of the RV estimator is $j_n=20$ and $i_n=10$.}
\label{tab:IV_est_delta=1s}
\end{table}

\newpage

\begin{table}[ht]
\centering
\begin{tabular}{l|ccccc}\hline\hline
	$\rho$                             &-0.7         &-0.3         &0            &0.3           & 0.7      \\\hline
	$\widehat{\rm IV}_{\rm step1}$     &5.52 (0.22)  &5.76 (0.21)  &6.00 (0.22)  &6.37 (0.23)   & 7.71 (0.27)     \\
	$\widehat{\rm IV}_{n}$             &5.86 (0.22)  &5.85 (0.21)  &5.85 (0.22)  &5.84 (0.23)   & 5.88(0.27)      \\
	$\widehat{\rm IV}_{\rm step2}$     &5.99 (0.23)  &6.00 (0.22)  &6.00 (0.23)  &6.00 (0.24)   & 6.07 (0.27) \\
	$\widehat{\rm IV}_{\rm step3}$     &6.00 (0.23)  &6.00 (0.22)  &6.00 (0.23)  &5.99 (0.24)   & 6.03 (0.27) \\\hline\hline
\end{tabular}
\caption{Estimation of the integrated volatility with ultra high-frequency data.
The numbers represent the means of the four estimators of integrated volatility,
$\widehat{\rm IV}_{\rm step1}$, $\widehat{\rm IV}_{n}$, $\widehat{\rm IV}_{\rm step2}$ and $\widehat{\rm IV}_{\rm step3}$,
based on $1\mathord{,}000$ simulations
with standard deviations between parentheses.
The true value of the integrated volatility is given by $\sigma^2 = 6\times 10^{-5}$.
All numbers in the table are multiplied by $10^5$.
Different from Table \ref{tab:IV_est_delta=1s}, we now take $\Delta=0.05$ sec and the number of observations is $468\mathord{,}000$.
The tuning parameter of the RV estimator is $j_n=20$ and $i_n=10$.}
\label{tab:IV_est_delta=0.1s}
\end{table}

\newpage

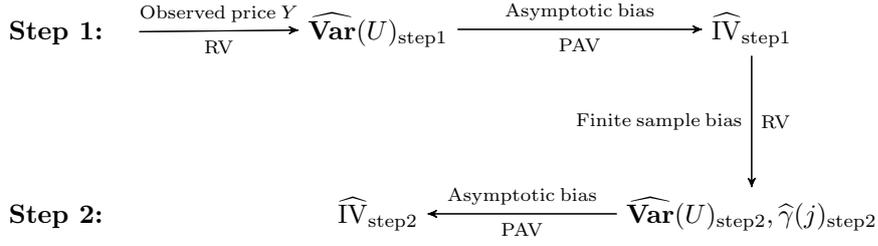
\begin{figure}[ht]
		\begin{tikzpicture}
		\matrix (m) [matrix of math nodes, row sep=5em, column sep=6em]
		{ \textbf{Step 1:}\quad  & \widehat{\mathbf{Var}}(U)_{\rm step1}   & \widehat{\rm IV}_{\rm step1}  \\
			\textbf{Step 2:}\quad  & \widehat{\rm IV}_{\rm step2} &\widehat{\mathbf{Var}}(U)_{\rm step2},\widehat{\gamma}(j)_{\rm step2} \\};
		{ [start chain] %\chainin (m-1-1);
			\chainin (m-1-1);
			\chainin (m-1-2) [join={node[above,labeled] {{\rm Observed\:price\:}Y} node[below,labeled] {\rm RV}}];
			\chainin (m-1-3) [join={node[above,labeled] {\rm Asymptotic \:bias} node[below,labeled] {\rm PAV}}];
			{  	 \chainin (m-2-3)
				[join={node[left,labeled] {\rm Finite\:sample \:bias} node[right,labeled] {\rm RV}}];
			}
			{ [start chain]
				\chainin (m-2-2) [join={node[below,labeled] {\rm PAV} node[above,labeled] {\rm Asymptotic \:bias}}];
			}}
			\end{tikzpicture}
\caption{Illustration of the two-step estimators.
In the first step, we use realized volatility (RV) to obtain an estimator of the variance of (possibly misspecified) i.i.d. noise, $\widehat{\mathbf{Var}}(U)_{\rm step1}$.
Next, this estimator is used to correct the asymptotic bias in the probability limit of the pre-averaging estimator (PAV)
to derive the first-step estimator of the integrated volatility, $\widehat{\rm IV}_{\rm step1}$.
In the second step, we use $\widehat{\rm IV}_{\rm step1}$ to obtain finite sample bias corrected estimators of the variance and covariances of noise, $\widehat{\mathbf{Var}}(U)_{\rm step2}$ and $\widehat{\gamma}(j)_{\rm step2}$,
which are finally used to remove the asymptotic bias in PAV,
leading to the second-step integrated volatility estimator, $\widehat{\rm IV}_{\rm step2}$.}
\label{fig:description_two_step_estimator}
\end{figure}

\newpage

\begin{figure}
\centering
% Minimal pgfplots version: 1.3
%  http://www.mathworks.com/matlabcentral/fileexchange/22022-matlab2tikz
\definecolor{mycolor1}{rgb}{1.00000,0.00000,1.00000}%
\begin{tikzpicture}

\begin{axis}[%
width=4.5in,
height=2.5in,
xlabel = Lags,
at={(0.739556in,4.491039in)},
scale only axis,
separate axis lines,
every outer x axis line/.append style={black},
every x tick label/.append style={font=\color{black}},
xmin=0,
xmax=30,
every outer y axis line/.append style={black},
every y tick label/.append style={font=\color{black}},
ymin=4e-08,
ymax=1.4e-07,
axis x line*=bottom,
axis y line*=left,
legend style={at={(0.05,1)},anchor=north west,legend cell align=left,align=left,draw=black}
]
\addplot [color=green,only marks,mark=asterisk,mark options={solid}]
  table[row sep=crcr]{%
1	1.021e-07\\
2	5.093e-08\\
3	8.6749e-08\\
4	6.16757e-08\\
5	7.922701e-08\\
6	6.6941093e-08\\
7	7.55412349e-08\\
8	6.952113557e-08\\
9	7.3735205101e-08\\
10	7.07853564293e-08\\
11	7.285025049949e-08\\
12	7.1404824650357e-08\\
13	7.24166227447501e-08\\
14	7.17083640786749e-08\\
15	7.22041451449276e-08\\
16	7.18570983985507e-08\\
17	7.21000311210145e-08\\
18	7.19299782152899e-08\\
19	7.20490152492971e-08\\
20	7.1965689325492e-08\\
21	7.20240174721556e-08\\
22	7.19831877694911e-08\\
23	7.20117685613562e-08\\
24	7.19917620070506e-08\\
25	7.20057665950646e-08\\
26	7.19959633834548e-08\\
27	7.20028256315816e-08\\
28	7.19980220578929e-08\\
29	7.2001384559475e-08\\
};
\addlegendentry{True};

\addplot [color=black,solid]
  table[row sep=crcr]{%
1	1.06644735209236e-07\\
2	5.49035328465899e-08\\
3	9.54339830853775e-08\\
4	7.0174080401521e-08\\
5	9.1778554291095e-08\\
6	8.03240638755558e-08\\
7	9.19017215137769e-08\\
8	8.69997250643526e-08\\
9	9.39207419346189e-08\\
10	9.29589932593728e-08\\
11	9.66184326044972e-08\\
12	9.72961280430312e-08\\
13	1.00195520445408e-07\\
14	1.01078008091769e-07\\
15	1.03968966325787e-07\\
16	1.05580118775226e-07\\
17	1.06465106279591e-07\\
18	1.10747189205041e-07\\
19	1.10807992382308e-07\\
20	1.1415828148479e-07\\
21	1.1544662240055e-07\\
22	1.16999283196272e-07\\
23	1.19782354027362e-07\\
24	1.20791741948996e-07\\
25	1.23299498050732e-07\\
26	1.25105452499479e-07\\
27	1.27632382488201e-07\\
28	1.28525131738784e-07\\
29	1.3120879460618e-07\\
};
\addlegendentry{RV};

\addplot [color=red,solid]
  table[row sep=crcr]{%
1	1.04505331674942e-07\\
2	5.06247257780006e-08\\
3	8.90157724824936e-08\\
4	6.16164662643424e-08\\
5	8.10815366196218e-08\\
6	6.74876426697879e-08\\
7	7.69258967737145e-08\\
8	6.98844967899955e-08\\
9	7.46661101259671e-08\\
10	7.15649579164264e-08\\
11	7.30849937272562e-08\\
12	7.16232856314955e-08\\
13	7.23832744995781e-08\\
14	7.11263586116436e-08\\
15	7.18779133113678e-08\\
16	7.13496622265122e-08\\
17	7.00952461965817e-08\\
18	7.22379255877376e-08\\
19	7.01593252307097e-08\\
20	7.13702107988974e-08\\
21	7.05191481803626e-08\\
22	6.99324054417903e-08\\
23	7.05760727385856e-08\\
24	6.94460571259251e-08\\
25	6.98144096933657e-08\\
26	6.94809606078182e-08\\
27	6.98684870622456e-08\\
28	6.86218327785341e-08\\
29	6.91660921116358e-08\\
};
\addlegendentry{RVM};

\addplot [color=blue,solid]
  table[row sep=crcr]{%
1	1.04077450968083e-07\\
2	4.97689643642827e-08\\
3	8.77321303619168e-08\\
4	5.99049434369067e-08\\
5	7.89421330853272e-08\\
6	6.49203584286344e-08\\
7	7.3930731825702e-08\\
8	6.64614511351241e-08\\
9	7.08151837642368e-08\\
10	6.72861508478371e-08\\
11	6.8378305951808e-08\\
12	6.64887171491884e-08\\
13	6.6820825310412e-08\\
14	6.51360287156186e-08\\
15	6.54597027084839e-08\\
16	6.45035709167694e-08\\
17	6.282127417998e-08\\
18	6.45360728642769e-08\\
19	6.202959180039e-08\\
20	6.28125966617188e-08\\
21	6.15336533363251e-08\\
22	6.05190298908939e-08\\
23	6.07348164808303e-08\\
24	5.91769201613108e-08\\
25	5.91173920218925e-08\\
26	5.83560622294861e-08\\
27	5.83157079770545e-08\\
28	5.66411729864841e-08\\
29	5.67575516127269e-08\\
};
\addlegendentry{RVH};

\addplot [color=mycolor1,solid]
  table[row sep=crcr]{%
1	1.04933212381801e-07\\
2	5.14804871917184e-08\\
3	9.02994146030704e-08\\
4	6.33279890917781e-08\\
5	8.32209401539165e-08\\
6	7.00549269109415e-08\\
7	7.9921061721727e-08\\
8	7.33075424448669e-08\\
9	7.85170364876975e-08\\
10	7.58437649850157e-08\\
11	7.77916815027044e-08\\
12	7.67578541138027e-08\\
13	7.79457236887441e-08\\
14	7.71166885076686e-08\\
15	7.82961239142518e-08\\
16	7.8195753536255e-08\\
17	7.73692182131835e-08\\
18	7.99397783111983e-08\\
19	7.82890586610293e-08\\
20	7.99278249360759e-08\\
21	7.95046430244001e-08\\
22	7.93457809926867e-08\\
23	8.04173289963409e-08\\
24	7.97151940905393e-08\\
25	8.05114273648389e-08\\
26	8.06058589861503e-08\\
27	8.14212661474366e-08\\
28	8.0602492570584e-08\\
29	8.15746326105447e-08\\
};
\addlegendentry{RVL};

\end{axis}
\end{tikzpicture}%
\caption{Realized volatility estimators against the number of lags $j$,
based on a single simulated sample,
without and with finite sample bias correction, cf. \eqref{eq:Pconverge_jth_RV} and \eqref{eq:RV_SSBC}.
Here, RV: $\widehat{\RV{Y,Y}}_n(j)$;
RVL: $\widehat{\RV{Y,Y}}_n(j)-\frac{0.8\sigma^2j}{2(n-j+1)}$;
RVM: $\widehat{\RV{Y,Y}}_n(j)-\frac{\sigma^2j}{2(n-j+1)}$; and
RVH: $\widehat{\RV{Y,Y}}_n(j)-\frac{1.2\sigma^2j}{2(n-j+1)}$.
We take $\Delta = 1$ sec, the number of observations is $23\mathord{,}400$, and $\rho=-0.7$.
The designation ``True'' corresponds to the stochastic limit $\var{U}-\gamma(j)$.}% Top panel: $\Delta=1s$; bottom panel $\Delta=0.1s$.}
\label{fig:RVpath}
\end{figure}
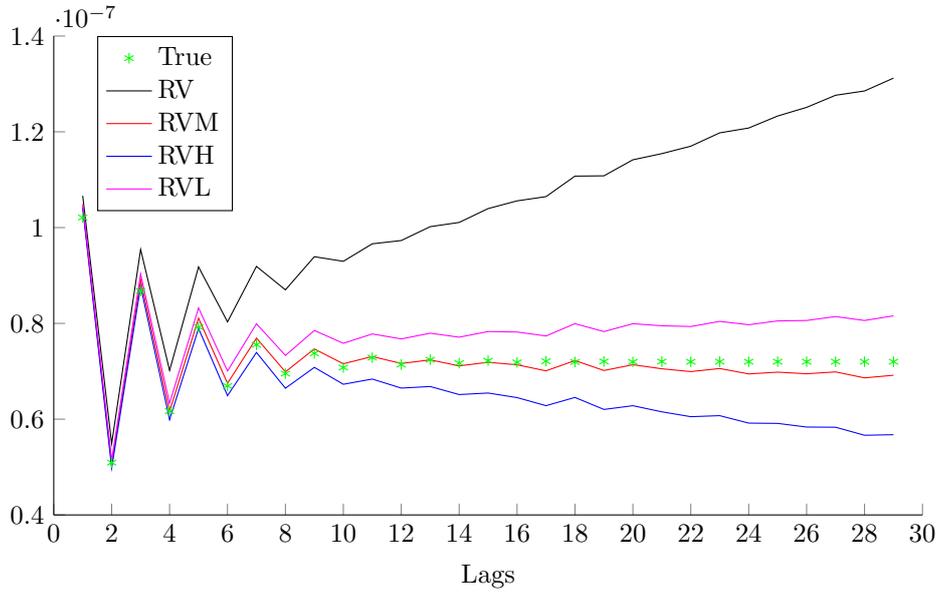

\newpage
	
\begin{sidewaysfigure}[p]
\centering
\input{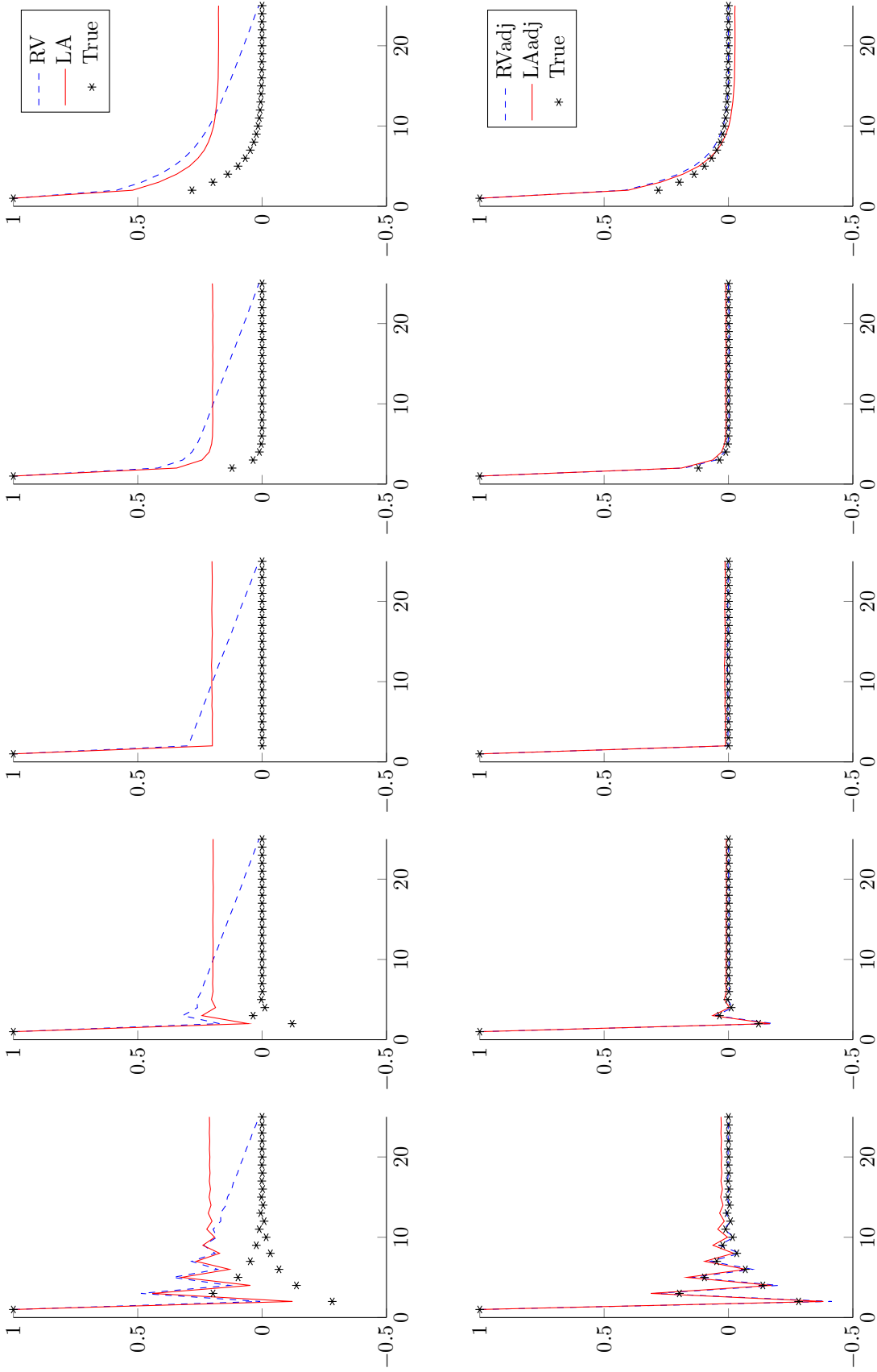}
\caption{Realized volatility (RV) and local averaging (LA) estimators of the autocorrelations of noise against the number of lags $j$,
averaged over $1\mathord{,}000$ simulated samples.
Top panel: RV and LA estimators without finite sample bias corrections.
Bottom panel: RV and LA estimators with finite sample bias corrections.
We take $\Delta=1$ sec and the number of observations is $23\mathord{,}400$.
The tuning parameters of the RV and LA estimators are $j_n=25$ and $K_n=6$, respectively.
The AR(1) coefficient $\rho$ of the noise process is fixed in each column.
From left to right, the values of $\rho$ for each column are given by $-0.7, -0.3, 0, 0.3$, and $0.7$.}
\label{fig:RVvsLA}
\end{sidewaysfigure}

\newpage

\begin{sidewaysfigure}[p]
\centering
\input{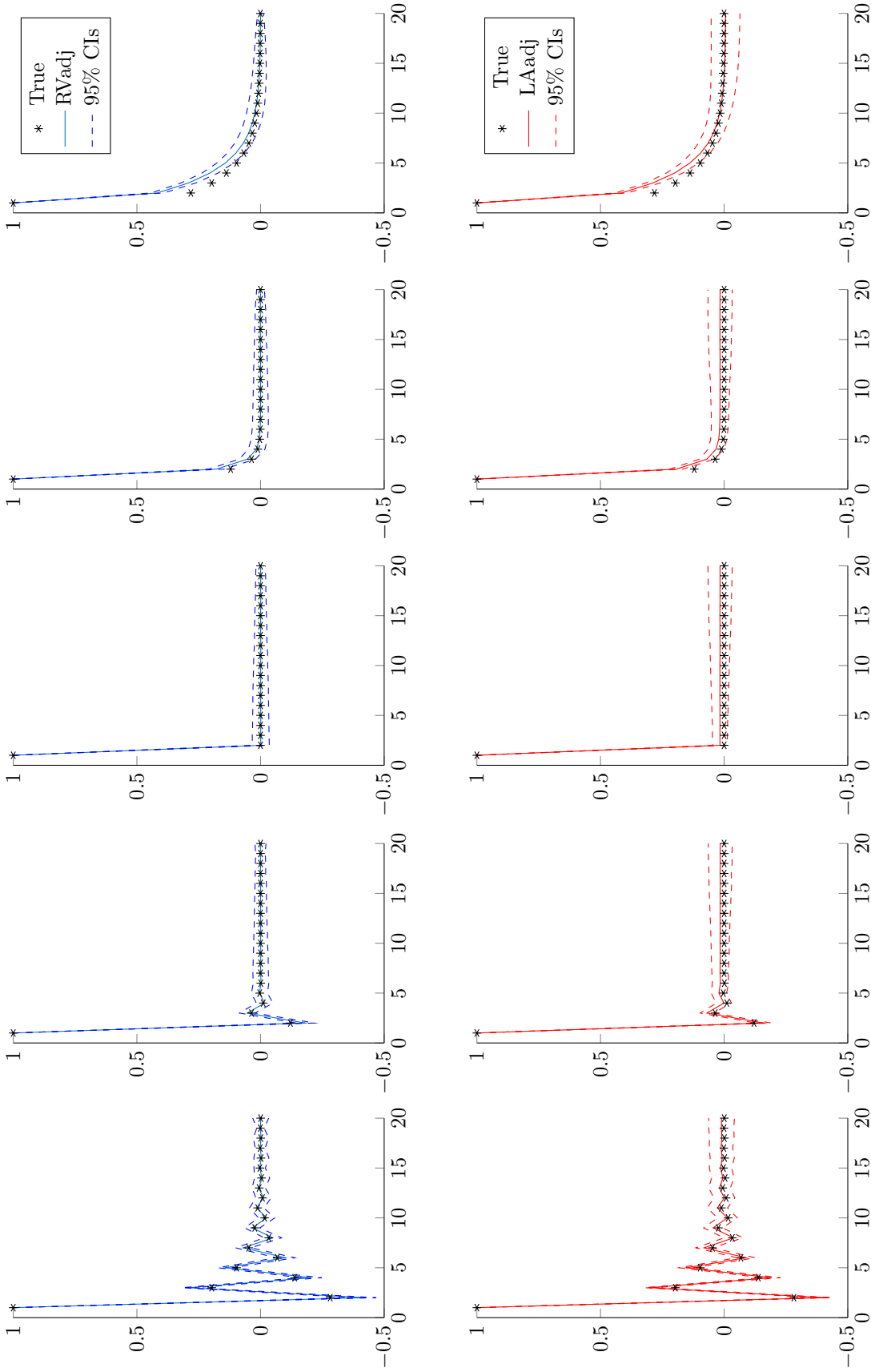}
\caption{Realized volatility (RV) and local averaging (LA) estimators of the autocorrelations of noise against the number of lags $j$,
averaged over $1\mathord{,}000$ simulated samples.
Top panel: RV estimators of the autocorrelations of noise (solid) with 95\% simulated confidence intervals (dashed).
Bottom panel: LA estimators of the autocorrelations of noise (solid) with 95\% simulated confidence intervals (dashed).
The true autocorrelations are displayed in stars.
Both estimators include the finite sample bias correction.
We take $\Delta=1$ sec and the number of observations is $23\mathord{,}400$.
The tuning parameters of the RV and LA estimators are $j_n=20$ and $K_n=6$, respectively.
The AR(1) coefficient $\rho$ of the noise process is fixed in each column.
From left to right, the values of $\rho$ for each column are given by $-0.7, -0.3, 0, 0.3$, and $0.7$.}
\label{fig:RVvsLAWithCIkn6}
\end{sidewaysfigure}

\newpage

\begin{sidewaysfigure}[p]
\centering
\input{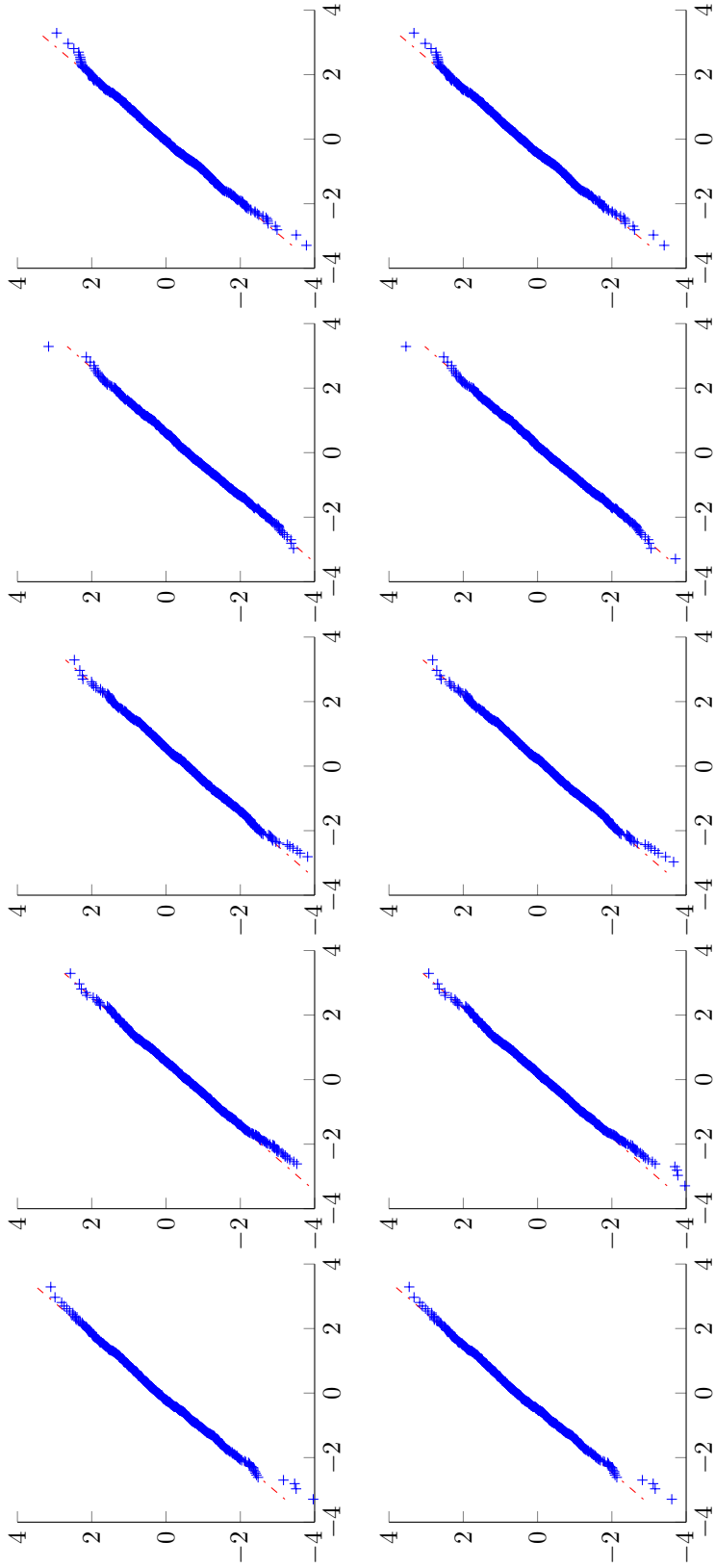}
\caption{Top panel: Standard normal QQ-plots of $n^{1/4}\myp{\widehat{\rm IV}_n-\int_{0}^{1}\sigma^2_s\diff s}/\tau_n$.
Bottom panel: Standard normal QQ-plots of $n^{1/4}\myp{\widehat{\rm IV}_{\rm step2}-\int_{0}^{1}\sigma^2_s\diff s}/\tau_n$.
The AR(1) coefficient $\rho$ of the noise process is fixed in each column.
From left to right, the values of $\rho$ for each column are given by $-0.7, -0.3, 0, 0.3$, and $0.7$.
The number of simulations is $1\mathord{,}000$, the data frequency is $\Delta = 0.05$ sec,
and the number of observations is $468\mathord{,}000$.
The tuning parameter of the RV estimator is $j_n=20$ and $i_n=10$.}
\label{fig:CLT_QQplots}
\end{sidewaysfigure}

\newpage

\begin{sidewaysfigure}[p]
\centering
\input{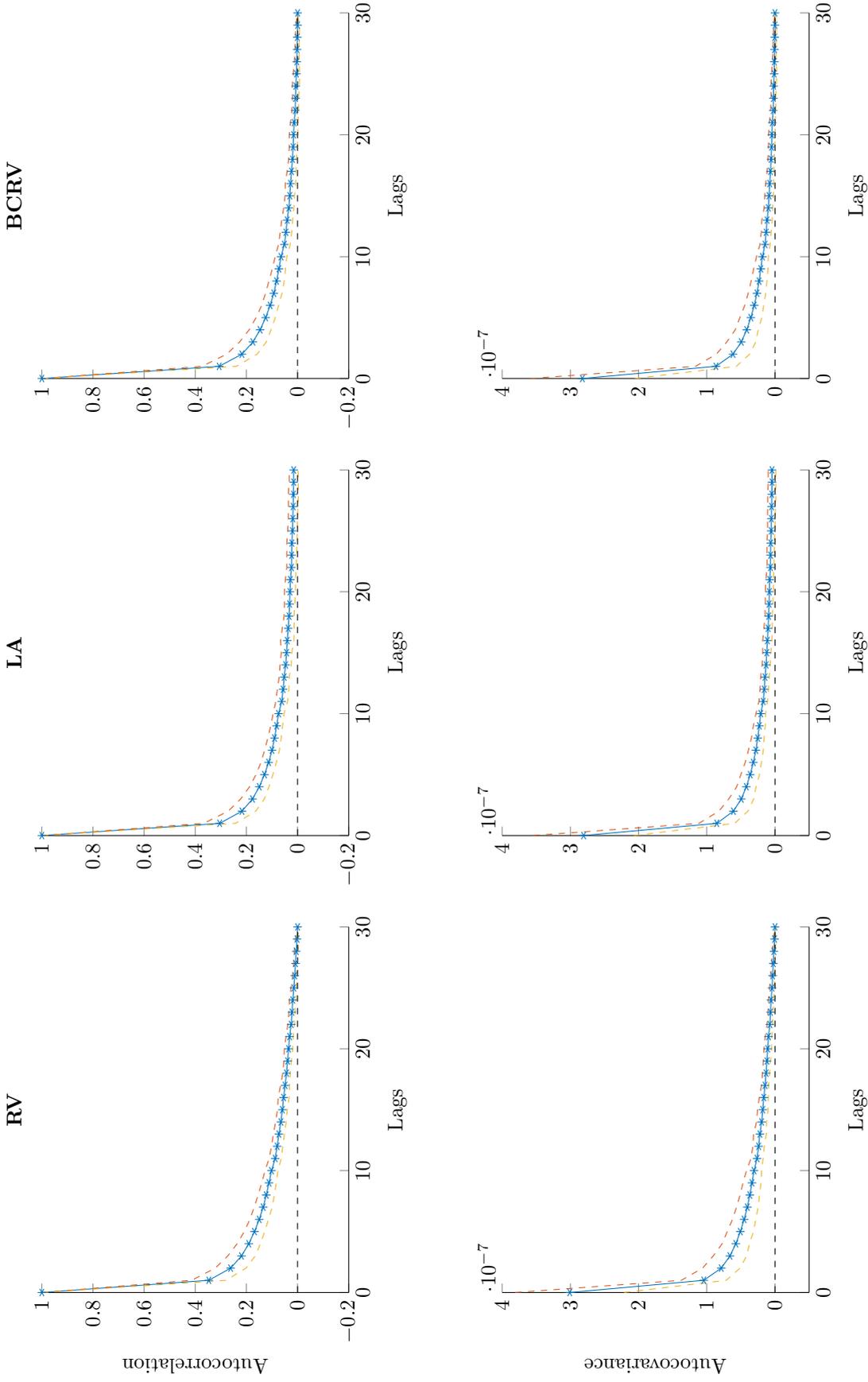}
\caption{
From the left to the right, we display the realized volatility (RV), local averaging (LA),
and the bias corrected realized volatility (BCRV) estimators of the autocorrelations (top panel) and autocovariances (bottom panel) of noise
against the number of lags $j$
based on transaction data for Citigroup.
Sample period: January, 2011.
On average there are 10.5 observations per second in the sample.
The three estimators are applied to and next averaged over each of the 20 trading days.
The stars indicate the means of the 20 estimates.
The dashed lines are 2 standard deviations away from the mean.
The tuning parameter of the RV estimator is $j_n=30$. % and $i_n=15$.
}
\label{fig:RV_LA_2step_fulldata}
\end{sidewaysfigure}

\newpage

\begin{figure}[h]
\centering
% Minimal pgfplots version: 1.3
%  http://www.mathworks.com/matlabcentral/fileexchange/22022-matlab2tikz
\definecolor{mycolor1}{rgb}{0.00000,0.44700,0.74100}%
\definecolor{mycolor2}{rgb}{0.85000,0.32500,0.09800}%
\definecolor{mycolor3}{rgb}{0.92900,0.69400,0.12500}%
\begin{tikzpicture}

\begin{axis}[%
width=5in,
height=2in,
xlabel = Trading days,
ylabel = Integrated volatility,
at={(0.758333in,3.5in)},
scale only axis,
xmin=1,
xmax=20,
ymin=0,
ymax=0.0006,
axis x line*=bottom,
axis y line*=left,
legend style={at={(0.0,1.0)},legend cell align=left,align=left,anchor=north west,draw=white!15!black}
]
\addplot [color=mycolor1,dashed,mark=asterisk,mark options={solid}]
  table[row sep=crcr]{%
1	0.000342767981136732\\
2	0.000309675679890533\\
3	0.000299212452920407\\
4	0.000344871940393436\\
5	0.000338792285141596\\
6	0.000237835791172763\\
7	0.000152344087130008\\
8	0.000226413186665286\\
9	0.000270142242472326\\
10	0.000345705476926557\\
11	0.000538226113225456\\
12	0.000325325376132046\\
13	0.000481535797607275\\
14	0.000260852267434378\\
15	0.000184487508247319\\
16	0.000239417822974756\\
17	0.000270647034327211\\
18	0.000205342327605504\\
19	0.000346136133818423\\
20	0.000217904032329449\\
};
\addlegendentry{$\widehat{{\rm IV}}_{\rm step1}$};

\addplot [color=green,dashed,mark=o,mark options={solid}]
  table[row sep=crcr]{%
1	0.000272949614582903\\
2	0.000230320804729469\\
3	0.000204901174361047\\
4	0.000258177723748505\\
5	0.000268301331027874\\
6	0.000174573819309173\\
7	0.000106520956869723\\
8	0.000162092054262352\\
9	0.000199463086212309\\
10	0.000283237810869355\\
11	0.000445853822902882\\
12	0.00024090452797241\\
13	0.000370020218163353\\
14	0.000195756872467129\\
15	0.000121813964805414\\
16	0.000183616867148079\\
17	0.000204696580135662\\
18	0.000131471243268043\\
19	0.000258132418652664\\
20	0.000143186259430568\\
};
\addlegendentry{$\widehat{{\rm IV}}_{\rm step2}$};

\addplot [color=red,solid,mark=asterisk,mark options={solid}]
  table[row sep=crcr]{%
1	0.000240523891071523\\
2	0.000197636360486083\\
3	0.000173276170669884\\
4	0.000224686046059788\\
5	0.000236909338330909\\
6	0.000139737187571973\\
7	7.8153197820909e-05\\
8	0.000139358056457569\\
9	0.000168918185474758\\
10	0.000254225906025469\\
11	0.000419765455741406\\
12	0.000216334857656741\\
13	0.000324637180357734\\
14	0.000168809233821773\\
15	9.88021803705711e-05\\
16	0.000149247747801338\\
17	0.000167122881317485\\
18	9.60130896422436e-05\\
19	0.000221530064665038\\
20	0.0001036564137206\\
};
\addlegendentry{$\widehat{{\rm IV}}_n$};

\end{axis}

\begin{axis}[%
width=5in,
height=2in,
xlabel = Trading days,
ylabel = Integrated volatility,
at={(0.758333in,0.48125in)},
scale only axis,
xmin=1,
xmax=20,
ymin=0,
ymax=0.0006,
axis x line*=bottom,
axis y line*=left,
legend style={at={(0.0,1.0)},legend cell align=left,align=left,anchor=north west,draw=white!15!black}
]
\addplot [color=mycolor1,dashed,mark=asterisk,mark options={solid}]
  table[row sep=crcr]{%
1	0.000342767981136732\\
2	0.000309675679890533\\
3	0.000299212452920407\\
4	0.000344871940393436\\
5	0.000338792285141596\\
6	0.000237835791172763\\
7	0.000152344087130008\\
8	0.000226413186665286\\
9	0.000270142242472326\\
10	0.000345705476926557\\
11	0.000538226113225456\\
12	0.000325325376132046\\
13	0.000481535797607275\\
14	0.000260852267434378\\
15	0.000184487508247319\\
16	0.000239417822974756\\
17	0.000270647034327211\\
18	0.000205342327605504\\
19	0.000346136133818423\\
20	0.000217904032329449\\
};
\addlegendentry{$\widehat{{\rm IV}}_{\rm step1}$};

\addplot [color=green,dashed,mark=o,mark options={solid}]
  table[row sep=crcr]{%
1	0.000272949614582903\\
2	0.000230320804729469\\
3	0.000204901174361047\\
4	0.000258177723748505\\
5	0.000268301331027874\\
6	0.000174573819309173\\
7	0.000106520956869723\\
8	0.000162092054262352\\
9	0.000199463086212309\\
10	0.000283237810869355\\
11	0.000445853822902882\\
12	0.00024090452797241\\
13	0.000370020218163353\\
14	0.000195756872467129\\
15	0.000121813964805414\\
16	0.000183616867148079\\
17	0.000204696580135662\\
18	0.000131471243268043\\
19	0.000258132418652664\\
20	0.000143186259430568\\
};
\addlegendentry{$\widehat{{\rm IV}}_{\rm step2}$};

\addplot [color=black,dotted]
  table[row sep=crcr]{%
1	0.000307856508318686\\
2	0.000266556966325669\\
3	0.000239028728983702\\
4	0.00029654008140327\\
5	0.00030648667638126\\
6	0.000205360794189778\\
7	0.00012865479993235\\
8	0.000188340323699273\\
9	0.00022970083996152\\
10	0.000319356680117591\\
11	0.000496248629172458\\
12	0.000276886338536168\\
13	0.000509112434830414\\
14	0.000227586685774148\\
15	0.000148535260737221\\
16	0.000216078513009179\\
17	0.000238131946311743\\
18	0.000161193533601001\\
19	0.000295839064978034\\
20	0.000172783551928873\\
};
\addlegendentry{95\% CIs};

\addplot [color=black,dotted,forget plot]
  table[row sep=crcr]{%
1	0.00023804272084712\\
2	0.000194084643133269\\
3	0.000170773619738392\\
4	0.00021981536609374\\
5	0.000230115985674488\\
6	0.000143786844428567\\
7	8.43871138070971e-05\\
8	0.000135843784825432\\
9	0.000169225332463098\\
10	0.000247118941621119\\
11	0.000395459016633306\\
12	0.000204922717408651\\
13	0.000230928001496292\\
14	0.00016392705916011\\
15	9.50926688736078e-05\\
16	0.000151155221286979\\
17	0.000171261213959582\\
18	0.000101748952935084\\
19	0.000220425772327293\\
20	0.000113588966932263\\
};
\end{axis}
\end{tikzpicture}%
\caption{Estimation of the integrated volatility based on transaction data for Citigroup.
Sample period: January, 2011, consisting of 20 trading days.
On average there are 10.5 observations per second in the sample.
The estimators $\widehat{\rm IV}_{\rm step1}$, $\widehat{\rm IV}_{\rm step2}$, and $\widehat{\rm IV}_{n}$
are given by~\eqref{eq:1stStepIV},~\eqref{eq:2ndStepIV}, and~\eqref{eq:consistency_SV_nonpar}.
%The three estimators are applied to each of the 20 trading days.
In the bottom panel, the asymptotic confidence intervals (CIs) are based on the limit distribution in Theorem~\ref{thm:CLT}.
%We set $j_n=30, j_n^* = 15, c=0.2$.%The tuning parameter of the RV estimator is $j_n=30$ and $i_n=15$.
The tuning parameter of the RV estimator is $j_n=30$ and $i_n=15$.
}
\label{fig:CitiIVsFullData}
\end{figure}
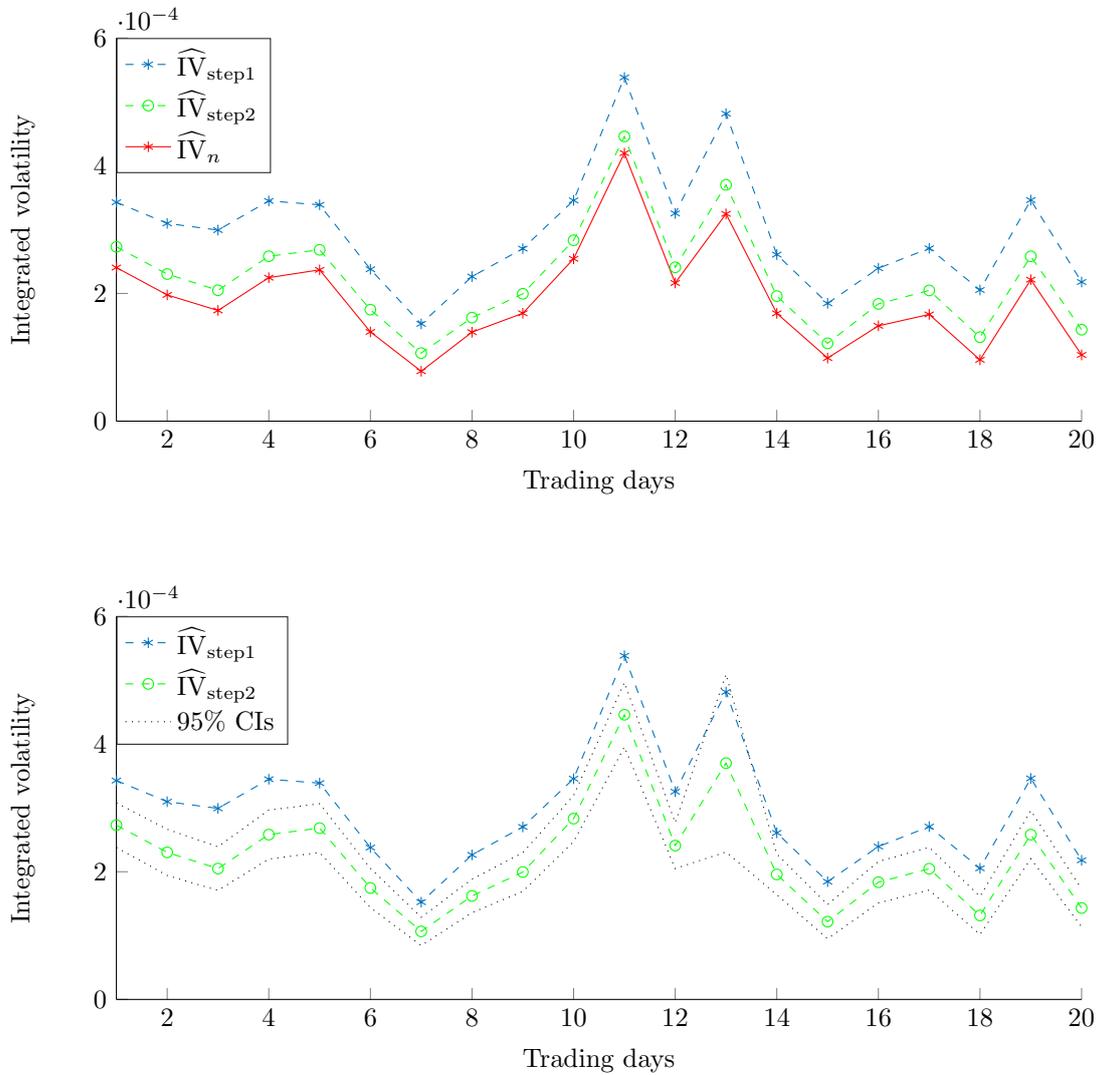

\newpage

\begin{figure}[h]
\centering
\input{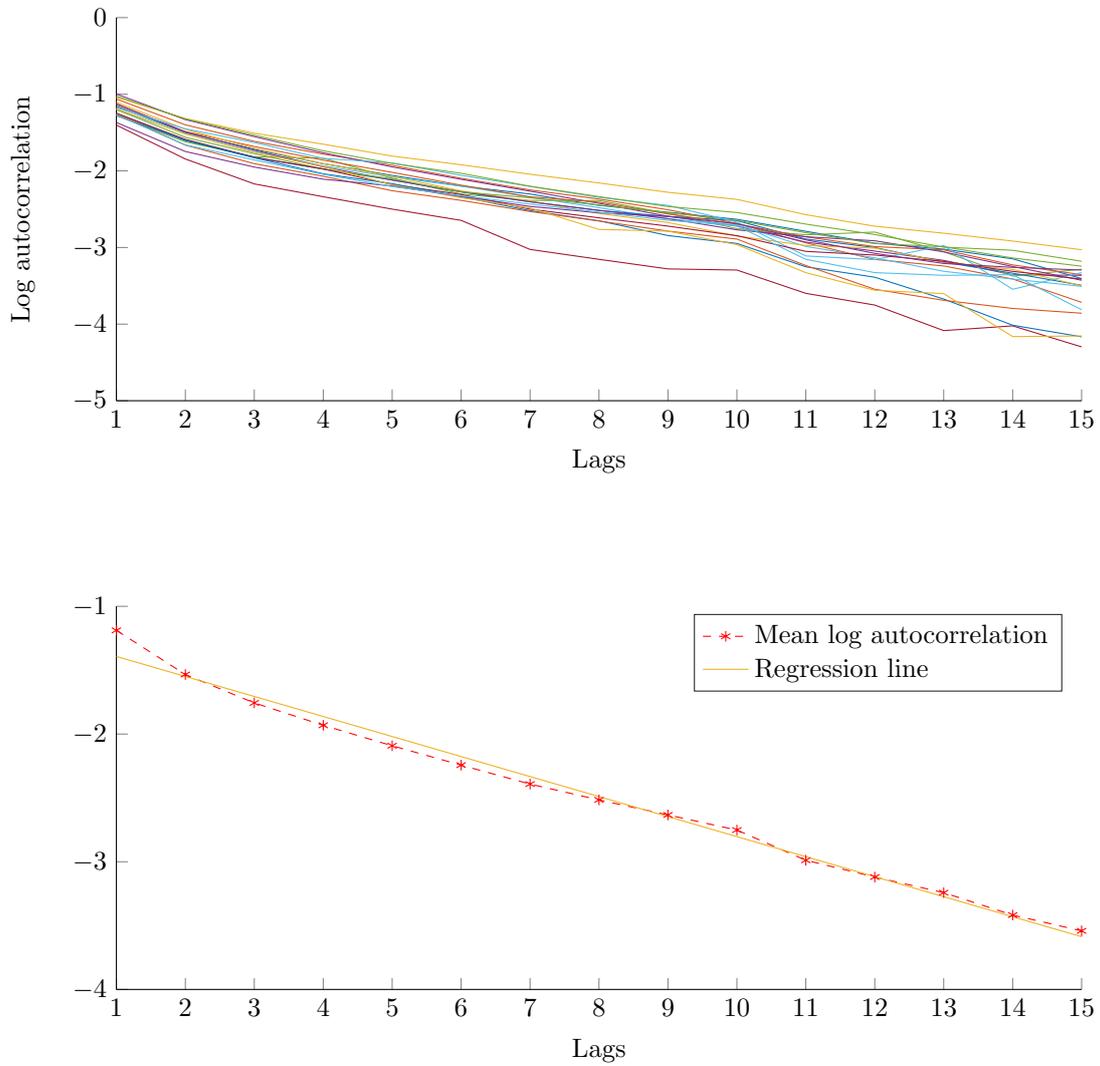}
\caption{Top panel: Logarithmic autocorrelations of noise against the number of lags $j$
estimated by BCRV for each trading day based on transaction data for Citigroup.
Bottom panel: Means of the logarithmic autocorrelations of noise and a linear regression line.
Sample period: January, 2011, consisting of 20 trading days.
On average there are 10.5 observations per second in the sample.
%We set $j_n=30, j_n^* = 15, c=0.2$.%The tuning parameter of the RV estimator is $j_n=30$ and $j_n^* = 15$.
The tuning parameter of the RV estimator is $j_n=30$. % and $i_n=15$.
}
\label{fig:CitiLogCors}
\end{figure}

\newpage

\begin{sidewaysfigure}[h]
\centering
\input{RV_LA_2step_seconddata.tex}
\caption{
From the left to the right, we display the realized volatility (RV), local averaging (LA),
and the bias corrected realized volatility (BCRV) estimators of the autocorrelations (top panel) and autocovariances (bottom panel) of noise
against the number of lags $j$
based on a subsample of the transaction data for Citigroup.
Sample period: January, 2011.
The subsample is recorded on a 1-sec time scale.
The three estimators are applied to and next averaged over each of the 20 trading days.
The stars indicate the means of the 20 estimates.
The dashed lines are 2 standard deviations away from the mean.
%We set $j_n=30, j_n^* = 5, c=0.2$.%The tuning parameter of the RV estimator is $j_n=30$.
The tuning parameter of the RV estimator is $j_n=30$. % and $i_n=5$.
}
\label{fig:RV_LA_2step_seconddata}
\end{sidewaysfigure}

\newpage

\begin{figure}[h]
\centering
% Minimal pgfplots version: 1.3
%  http://www.mathworks.com/matlabcentral/fileexchange/22022-matlab2tikz
\definecolor{mycolor1}{rgb}{0.00000,0.44700,0.74100}%
\definecolor{mycolor2}{rgb}{0.85000,0.32500,0.09800}%
\definecolor{mycolor3}{rgb}{0.92900,0.69400,0.12500}%
\begin{tikzpicture}

\begin{axis}[%
width=5in,
height=2.1in,
xlabel = Trading days,
ylabel = Integrated volatility,
at={(0.758333in,3.554288in)},
scale only axis,
xmin=1,
xmax=20,
ymin=0,
ymax=0.0006,
axis x line*=bottom,
axis y line*=left,
legend style={at={(0.0,1.0)},legend cell align=left,align=left,anchor=north west,draw=white!15!black}
]
\addplot [color=mycolor1,dashed,mark=asterisk,mark options={solid}]
  table[row sep=crcr]{%
1	0.000160120165164057\\
2	0.000205373194846904\\
3	0.000175072389027763\\
4	0.000222301085472496\\
5	0.000301984694486438\\
6	0.000185994454191223\\
7	0.000104691265705434\\
8	9.77547647194857e-05\\
9	0.000143065602508955\\
10	0.000282717029152941\\
11	0.000441851062923783\\
12	0.000140104058892604\\
13	0.000257314798244661\\
14	0.000198668174497754\\
15	0.000114907797822648\\
16	0.000219555104341305\\
17	0.000147693898899805\\
18	3.33944290737129e-05\\
19	0.000351012216825006\\
20	0.000144734139445378\\
};
\addlegendentry{$\widehat{\rm IV}_{\rm step1}$};

\addplot [color=green,dashed,mark=o,mark options={solid}]
  table[row sep=crcr]{%
1	0.000133173292108777\\
2	0.000200177522224605\\
3	0.000174065515000371\\
4	0.000223467253813652\\
5	0.000308888956314976\\
6	0.000184100640717711\\
7	0.000105639868050627\\
8	9.34671822336794e-05\\
9	0.00012302889790761\\
10	0.000268792738451024\\
11	0.00043601365768514\\
12	0.000141469449278672\\
13	0.000240363451263943\\
14	0.000203764305264761\\
15	0.000118762253101443\\
16	0.000243866081462708\\
17	0.000146869409637058\\
18	2.71063843871769e-05\\
19	0.000382926132593878\\
20	0.000138117189103881\\
};
\addlegendentry{$\widehat{\rm IV}_{\rm step2}$};

\addplot [color=red,solid,mark=asterisk,mark options={solid}]
  table[row sep=crcr]{%
1	5.04110030343137e-05\\
2	9.56482839451254e-05\\
3	8.35581452521043e-05\\
4	0.000109683865312266\\
5	0.000156633703088783\\
6	8.13355729923034e-05\\
7	4.83030537443079e-05\\
8	4.47076247458429e-05\\
9	4.90170966197249e-05\\
10	0.000127159581859553\\
11	0.000222162971756412\\
12	7.26241103372983e-05\\
13	0.000110204129839455\\
14	9.91696666391215e-05\\
15	5.73518024747446e-05\\
16	0.000122109706534132\\
17	6.72562385781986e-05\\
18	1.01486769378988e-05\\
19	0.000197518169781938\\
20	5.74958183892419e-05\\
};
\addlegendentry{$\widehat{{\rm IV}}_n$};

\end{axis}

\begin{axis}[%
width=5in,
height=2.2in,
xlabel = Trading days,
ylabel = Integrated volatility,
at={(0.758333in,0.48125in)},
scale only axis,
xmin=1,
xmax=20,
ymin=-0.0001,
ymax=0.0006,
axis x line*=bottom,
axis y line*=left,
legend style={at={(0.0,1.0)},legend cell align=left,align=left,anchor=north west,draw=white!15!black}
]
\addplot [color=mycolor1,dashed,mark=asterisk,mark options={solid}]
  table[row sep=crcr]{%
1	0.000160120165164057\\
2	0.000205373194846904\\
3	0.000175072389027763\\
4	0.000222301085472496\\
5	0.000301984694486438\\
6	0.000185994454191223\\
7	0.000104691265705434\\
8	9.77547647194857e-05\\
9	0.000143065602508955\\
10	0.000282717029152941\\
11	0.000441851062923783\\
12	0.000140104058892604\\
13	0.000257314798244661\\
14	0.000198668174497754\\
15	0.000114907797822648\\
16	0.000219555104341305\\
17	0.000147693898899805\\
18	3.33944290737129e-05\\
19	0.000351012216825006\\
20	0.000144734139445378\\
};
\addlegendentry{$\widehat{\rm IV}_{\rm step1}$};

\addplot [color=red,solid,mark=asterisk,mark options={solid}]
  table[row sep=crcr]{%
1	5.04110030343137e-05\\
2	9.56482839451254e-05\\
3	8.35581452521043e-05\\
4	0.000109683865312266\\
5	0.000156633703088783\\
6	8.13355729923034e-05\\
7	4.83030537443079e-05\\
8	4.47076247458429e-05\\
9	4.90170966197249e-05\\
10	0.000127159581859553\\
11	0.000222162971756412\\
12	7.26241103372983e-05\\
13	0.000110204129839455\\
14	9.91696666391215e-05\\
15	5.73518024747446e-05\\
16	0.000122109706534132\\
17	6.72562385781986e-05\\
18	1.01486769378988e-05\\
19	0.000197518169781938\\
20	5.74958183892419e-05\\
};
\addlegendentry{$\widehat{\rm IV}_{n}$};

\addplot [color=black,dotted]
  table[row sep=crcr]{%
1	0.000112106298011992\\
2	0.000152995154132794\\
3	0.000162590032908091\\
4	0.000192527380446363\\
5	0.000241486853869392\\
6	0.000137913215551924\\
7	8.56562816214041e-05\\
8	7.46561000534136e-05\\
9	9.19826888270062e-05\\
10	0.000193478393253096\\
11	0.000346535865982098\\
12	0.000113831012920566\\
13	0.00019733099853325\\
14	0.000162734337394795\\
15	0.000106579221390722\\
16	0.000212589430217391\\
17	0.000109082942782461\\
18	2.32679307722068e-05\\
19	0.000302266971003457\\
20	0.000107308842038\\
};
\addlegendentry{95\% CIs};

\addplot [color=black,dotted,forget plot]
  table[row sep=crcr]{%
1	-1.12842919433641e-05\\
2	3.83014137574572e-05\\
3	4.52625759611725e-06\\
4	2.68403501781684e-05\\
5	7.17805523081742e-05\\
6	2.47579304326827e-05\\
7	1.09498258672116e-05\\
8	1.47591494382722e-05\\
9	6.05150441244355e-06\\
10	6.0840770466009e-05\\
11	9.77900775307258e-05\\
12	3.14172077540309e-05\\
13	2.30772611456602e-05\\
14	3.56049958834479e-05\\
15	8.12438355876763e-06\\
16	3.16299828508734e-05\\
17	2.54295343739362e-05\\
18	-2.97057689640912e-06\\
19	9.27693685604179e-05\\
20	7.68279474048343e-06\\
};
\end{axis}
\end{tikzpicture}%
\caption{
Estimation of the integrated volatility based on a subsample of the transaction data for Citigroup.
Sample period: January, 2011, consisting of 20 trading days.
The subsample is recorded on a 1-sec time scale.
The estimators $\widehat{\rm IV}_{\rm step1}$, $\widehat{\rm IV}_{\rm step2}$, and $\widehat{\rm IV}_{n}$
are given by~\eqref{eq:1stStepIV},~\eqref{eq:2ndStepIV}, and~\eqref{eq:consistency_SV_nonpar}.
%The three estimators are applied to each of the 20 trading days.
In the bottom panel, the asymptotic confidence intervals (CIs) are based on the limit distribution in Theorem~\ref{thm:CLT}.
%We set $j_n=30, j_n^* = 5, c=0.2$.%The tuning parameter of the RV estimator is $j_n=30$ and $i_n=5$.
The tuning parameter of the RV estimator is $j_n=30$ and $i_n=5$.
}
\label{fig:CitiIVsSecondData}
\end{figure}
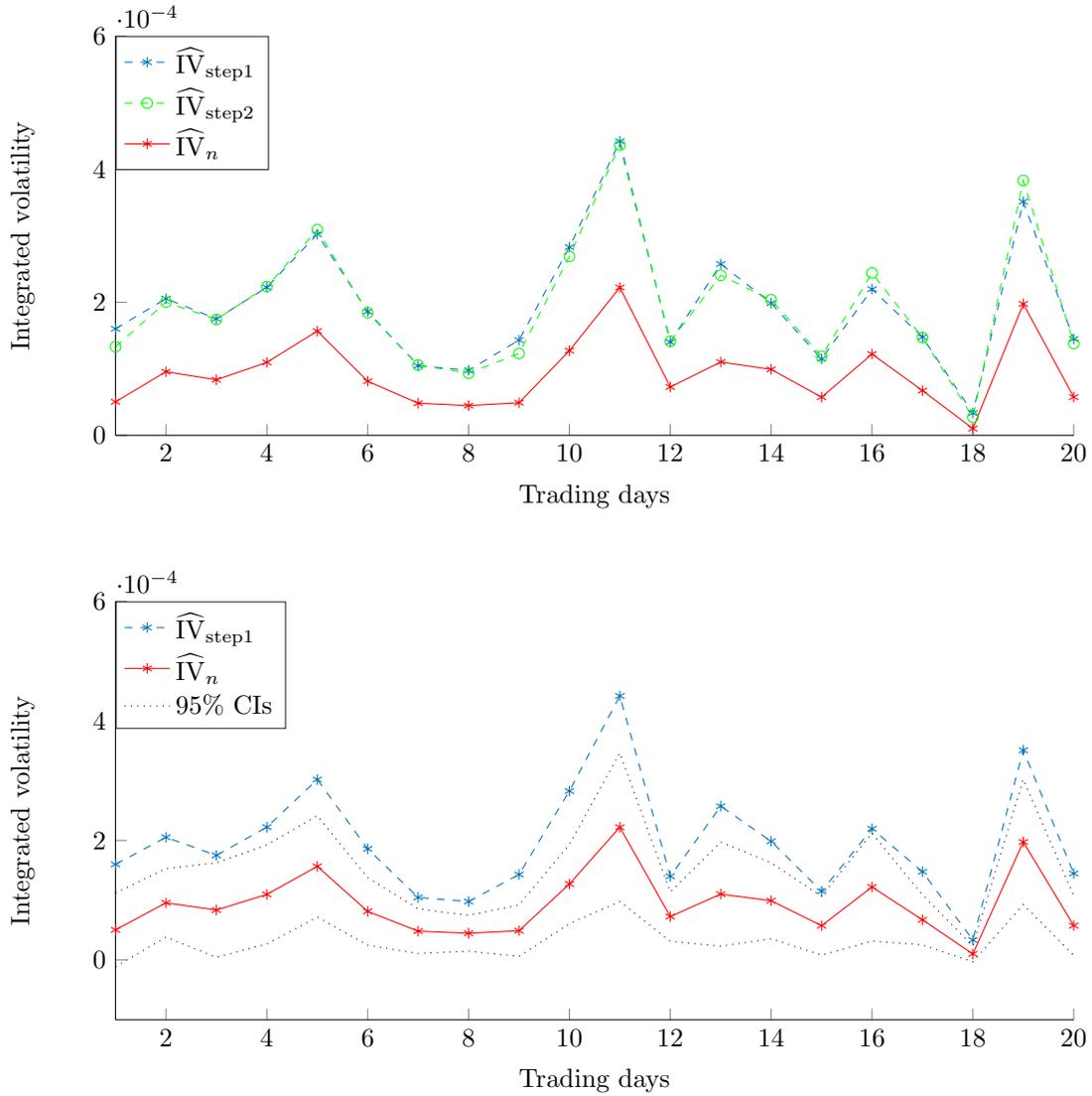

\newpage

\begin{sidewaysfigure}[h]
\centering
\input{RV_LA_2step_tickdata.tex}
\caption{
From the left to the right, we display the realized volatility (RV), local averaging (LA),
and the bias corrected realized volatility (BCRV) estimators of the autocorrelations (top panel) and autocovariances (bottom panel) of noise
against the number of lags $j$
based on a subsample of the transaction data for Citigroup.
Sample period: January, 2011.
The subsample is recorded at tick time.
On average there are 3.2 observations per second in the sample.
The three estimators are applied to and next averaged over each of the 20 trading days.
The stars indicate the means of the 20 estimates.
The dashed lines are 2 standard deviations away from the mean.
%We set $j_n=30, j_n^* = 10, c=0.2$.
The tuning parameter of the RV estimator is $j_n=30$. % and $i_n=10$.
}
\label{fig:RV_LA_2step_tickdata}
\end{sidewaysfigure}

\newpage

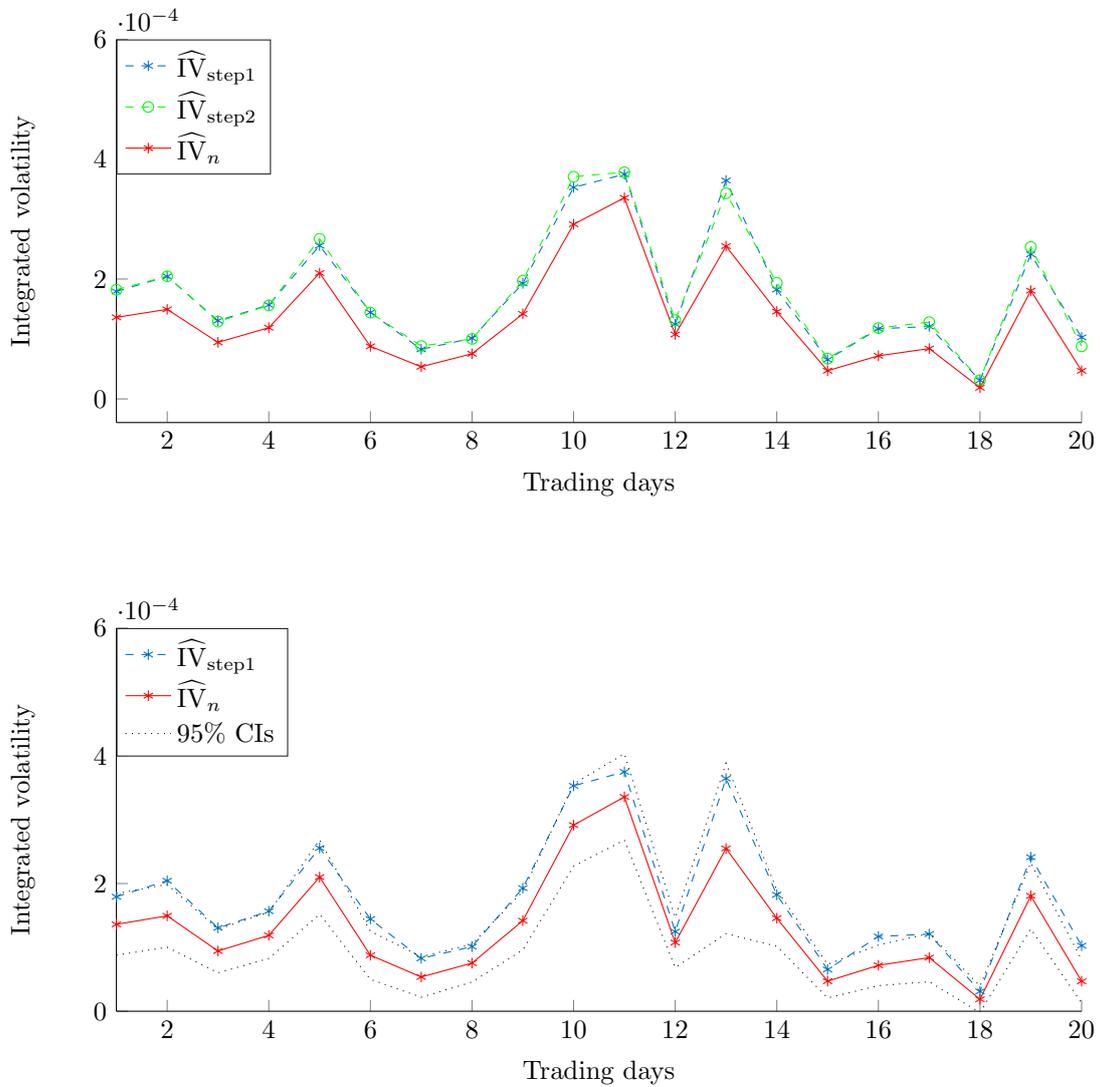
\begin{figure}[h]
\centering
% Minimal pgfplots version: 1.3
%  http://www.mathworks.com/matlabcentral/fileexchange/22022-matlab2tikz
\definecolor{mycolor1}{rgb}{0.00000,0.44700,0.74100}%
\definecolor{mycolor2}{rgb}{0.85000,0.32500,0.09800}%
\definecolor{mycolor3}{rgb}{0.92900,0.69400,0.12500}%
\begin{tikzpicture}

\begin{axis}[%
width=5in,
height=2in,
xlabel = Trading days,
ylabel = Integrated volatility,
at={(0.758333in,3.554288in)},
scale only axis,
xmin=1,
xmax=20,
ymax=0.0006,
axis x line*=bottom,
axis y line*=left,
legend style={at={(0.0,1.0)},legend cell align=left,align=left,anchor=north west,draw=white!15!black}
]
\addplot [color=mycolor1,dashed,mark=asterisk,mark options={solid}]
  table[row sep=crcr]{%
1	0.000179305118053875\\
2	0.000204381293891562\\
3	0.00013053835534511\\
4	0.000156768730509865\\
5	0.000255698919789147\\
6	0.000144341143105837\\
7	8.27167429620677e-05\\
8	0.000101176733991231\\
9	0.000192809423877079\\
10	0.000353012707681033\\
11	0.000374909052321031\\
12	0.000124694714529091\\
13	0.000364501061655364\\
14	0.000182404030452263\\
15	6.5387418475669e-05\\
16	0.000117032439053008\\
17	0.000120853323959626\\
18	3.09996737087025e-05\\
19	0.000241258245412906\\
20	0.000102701226353366\\
};
\addlegendentry{$\widehat{{\rm IV}}_{\rm step1}$};

\addplot [color=green,dashed,mark=o,mark options={solid}]
  table[row sep=crcr]{%
1	0.00018223313211677\\
2	0.000204613016446565\\
3	0.000128929278009242\\
4	0.000156023545299798\\
5	0.000267278689664566\\
6	0.000143838801392631\\
7	8.83201662558958e-05\\
8	0.000100103382780624\\
9	0.000197286813290372\\
10	0.000370808592033311\\
11	0.000378302217800778\\
12	0.000131342304694089\\
13	0.000343188187529943\\
14	0.000193792918183755\\
15	6.72660911169053e-05\\
16	0.000118262498919433\\
17	0.000128158981978209\\
18	3.12484759050565e-05\\
19	0.000253636472400188\\
20	8.77199927649441e-05\\
};
\addlegendentry{$\widehat{{\rm IV}}_{\rm step2}$};

\addplot [color=red,solid,mark=asterisk,mark options={solid}]
  table[row sep=crcr]{%
1	0.000136108951524483\\
2	0.000149408599252215\\
3	9.43714749838871e-05\\
4	0.000118830324837183\\
5	0.000210053119630892\\
6	8.76885315044032e-05\\
7	5.36482606977297e-05\\
8	7.5370178892563e-05\\
9	0.00014185005784352\\
10	0.000291489158984523\\
11	0.000335666401724383\\
12	0.000107532586153212\\
13	0.000254970855275166\\
14	0.000145729677886864\\
15	4.71195900707798e-05\\
16	7.19427806864547e-05\\
17	8.39966642842e-05\\
18	1.86839499083425e-05\\
19	0.000180658487342312\\
20	4.6927466000763e-05\\
};
\addlegendentry{$\widehat{\rm IV}_{n}$};

\end{axis}

\begin{axis}[%
width=5in,
height=2in,
at={(0.758333in,0.48125in)},
scale only axis,
xlabel = Trading days,
ylabel = Integrated volatility,
xmin=1,
xmax=20,
ymin=0,
ymax=0.0006,
axis x line*=bottom,
axis y line*=left,
legend style={at={(0.0,1.0)},legend cell align=left,align=left,anchor=north west,draw=white!15!black}
]
\addplot [color=mycolor1,dashed,mark=asterisk,mark options={solid}]
  table[row sep=crcr]{%
1	0.000179305118053875\\
2	0.000204381293891562\\
3	0.00013053835534511\\
4	0.000156768730509865\\
5	0.000255698919789147\\
6	0.000144341143105837\\
7	8.27167429620677e-05\\
8	0.000101176733991231\\
9	0.000192809423877079\\
10	0.000353012707681033\\
11	0.000374909052321031\\
12	0.000124694714529091\\
13	0.000364501061655364\\
14	0.000182404030452263\\
15	6.5387418475669e-05\\
16	0.000117032439053008\\
17	0.000120853323959626\\
18	3.09996737087025e-05\\
19	0.000241258245412906\\
20	0.000102701226353366\\
};
\addlegendentry{$\widehat{{\rm IV}}_{\rm step1}$};

\addplot [color=red,solid,mark=asterisk,mark options={solid}]
  table[row sep=crcr]{%
1	0.000136108951524483\\
2	0.000149408599252215\\
3	9.43714749838871e-05\\
4	0.000118830324837183\\
5	0.000210053119630892\\
6	8.76885315044032e-05\\
7	5.36482606977297e-05\\
8	7.5370178892563e-05\\
9	0.00014185005784352\\
10	0.000291489158984523\\
11	0.000335666401724383\\
12	0.000107532586153212\\
13	0.000254970855275166\\
14	0.000145729677886864\\
15	4.71195900707798e-05\\
16	7.19427806864547e-05\\
17	8.39966642842e-05\\
18	1.86839499083425e-05\\
19	0.000180658487342312\\
20	4.6927466000763e-05\\
};
\addlegendentry{$\widehat{\rm IV}_{n}$};

\addplot [color=black,dotted]
  table[row sep=crcr]{%
1	0.000184283989285277\\
2	0.000198642820023909\\
3	0.000128509273474775\\
4	0.000155001443417079\\
5	0.000268157019871519\\
6	0.000125405827772324\\
7	8.53416259257125e-05\\
8	0.000105003007073494\\
9	0.000187636210576267\\
10	0.000355968196561921\\
11	0.000403220243158482\\
12	0.000146690209007868\\
13	0.000388374268301363\\
14	0.000189756177857323\\
15	7.33763752287332e-05\\
16	0.000103787256837474\\
17	0.000121748231690376\\
18	3.9481784882493e-05\\
19	0.000232253872440713\\
20	8.08211614853254e-05\\
};
\addlegendentry{95\% CIs};

\addplot [color=black,dotted,forget plot]
  table[row sep=crcr]{%
1	8.79339137636883e-05\\
2	0.000100174378480521\\
3	6.02336764929991e-05\\
4	8.2659206257288e-05\\
5	0.000151949219390264\\
6	4.99712352364822e-05\\
7	2.19548954697468e-05\\
8	4.5737350711632e-05\\
9	9.60639051107733e-05\\
10	0.000227010121407126\\
11	0.000268112560290284\\
12	6.83749632985558e-05\\
13	0.000121567442248969\\
14	0.000101703177916405\\
15	2.08628049128264e-05\\
16	4.00983045354356e-05\\
17	4.62450968780244e-05\\
18	-2.11388506580795e-06\\
19	0.000129063102243911\\
20	1.30337705162006e-05\\
};
\end{axis}
\end{tikzpicture}%
\caption{
Estimation of the integrated volatility based on a subsample of the transaction data for Citigroup.
Sample period: January, 2011, consisting of 20 trading days.
The subsample is recorded at tick time.
On average there are 3.2 observations per second in the sample.
The estimators $\widehat{\rm IV}_{\rm step1}$, $\widehat{\rm IV}_{\rm step2}$, and $\widehat{\rm IV}_{n}$
are given by~\eqref{eq:1stStepIV},~\eqref{eq:2ndStepIV}, and~\eqref{eq:consistency_SV_nonpar}.
%The three estimators are applied to each of the 20 trading days.
In the bottom panel, the asymptotic confidence intervals (CIs) are based on the limit distribution in Theorem~\ref{thm:CLT}.
%We set $j_n=30$ and $i_n=10$.
The tuning parameter of the RV estimator is $j_n=30$ and $i_n=10$.
}
\label{fig:CitiIVsTickData}
\end{figure}

\newpage

\clearpage

\appendix

\renewcommand\theequation{\thesection.\arabic{equation}}
\renewcommand\thefigure{\thesection.\arabic{figure}}
\renewcommand\thetable{\thesection.\arabic{table}}

\vskip 2cm
\begin{center}
{\LARGE Supplementary Material to}\\[5mm]
{\LARGE ``Dependent Microstructure Noise and Integrated Volatility Estimation from High-Frequency Data''}
\end{center}

\newpage

\section*{Appendix}\label{sec:Appendix}

Sections \ref{appendix:prop_RV_Estimate_var+cov(1)}--\ref{appendix:thm_CLT} in this appendix
contain detailed technical proofs of our results.
In Sections%~\ref{sec:TickTimeSimu},
~\ref{sec:SVsimu} and~\ref{sec:GE_Empirical},
we provide additional Monte Carlo simulation and empirical results.
In the proofs that follow the constants $C$ and $\delta\in(0,1)$ may vary from line to line.
We add a subscript $q$ if they depend on some parameter $q$.

\begin{appendices}

\setcounter{equation}{0}
	
\section{Proof of Proposition~\ref{prop:RV_Estimate_var+cov(1)}}
\label{appendix:prop_RV_Estimate_var+cov(1)}

\begin{proof}
Adopting the standard localization procedure
(see e.g.,~\cite{jacod2011discretization} for further details),
we may assume that the processes $a$ and $\sigma$ are bounded by constants $C_{a},C_\sigma>0$.
This yields for any such continuous It\^o semimartingale $X$ and stopping times $S\leq T$ that
\begin{align}\label{eq:Classic_Est_X_p}
\conexp{\abs{X_T-X_S}^p}{\mathcal{F}_S}\leq C_p\conexp{T-S}{\mathcal{F}_S},\quad \forall p\geq 2.
\end{align}
%Then the established result will remain valid under our actual assumptions of local boundedness.

Let $\Delta_n = 1/n$.
For any process $V$, we write $\Delta^n_{i,j}V := V^n_{i+j}-V^n_{i}$, $j=1,2,\ldots,n-i$.
Then, for the log-price process $Y$,
\begin{equation}
			[Y,Y]^{j}_n := \sum_{i=0}^{n-j}(\Delta^n_{i,j} Y)^2 = \sum_{i=0}^{n-j}(\Delta^n_{i,j} X)^2 + 2\sum_{i=0}^{n-j}\Delta^n_{i,j} X\ \Delta^n_{i,j} U + \sum_{i=0}^{n-j}(\Delta^n_{i,j} U)^2.
\label{eq:decomp_[Y,Y]^j_n}
\end{equation}
We now analyze the asymptotic properties of the three components on the right-hand side of \eqref{eq:decomp_[Y,Y]^j_n}:
\begin{enumerate}
[(i)]\item First note that $\sum_{i=0}^{n-j}(\Delta^n_{i,j} X)^2/j\Pconverge  [X,X]$, where $[X,X]$ is the quadratic variation of $X$.
%\item We apply Lemma 1 from~\cite{Ait-Sahalia2011DependentNoise} to conclude that $\sum_{i=0}^{n-j}\Delta^n_{i,j} X\ \Delta^n_{i,j} U = O_p(j^{1/2})$.
\item By the independence of $X$ and $U$, we have
\begin{align}\label{eq:dX2_dU2_bound}
\sum_{i=0}^{n-j}\expect{\myp{\Delta^n_{i,j} X\ \Delta^n_{i,j} U }^2}= \sum_{i=0}^{n-j}\expect{\myp{\Delta^n_{i,j} X}^2}\expect{\myp{\Delta^n_{i,j} U }^2}\leq Cj.
\end{align}
The last inequality follows from the fact that $U$ has bounded moments and from an application of~\eqref{eq:Classic_Est_X_p}.
Next,
\begin{equation}\label{eq:dX2_dU2_cross_bound}
\begin{split}
&\sum_{i,i': i< i'}\expect{\Delta^n_{i,j} X\ \Delta^n_{i,j} U\ \Delta^n_{i',j} X\ \Delta^n_{i',j} U }\\
=&\sum_{i,i':i< i'}\expect{\Delta^n_{i,j} X\  \Delta^n_{i',j} X}\expect{ \Delta^n_{i,j} U\ \Delta^n_{i',j} U }\\
\leq& Cj\Delta_n\myp{\sum_{i,i': i+j< i'}\expect{ \Delta^n_{i,j} U\ \Delta^n_{i',j} U }+\sum_{i,i': i+j\geq  i'>i }\expect{ \Delta^n_{i,j} U\ \Delta^n_{i',j} U }}\\
\leq&  Cj^2.
\end{split}
\end{equation}
The first inequality follows from the Cauchy-Schwarz inequality and~\eqref{eq:Classic_Est_X_p}.
To see the second inequality, we apply the Cauchy-Schwarz inequality, Lemma VIII 3.102 of~\cite{jacod1987limit}
(hereafter abbreviated as JS-Lemma), and the fact that $v>2$ to obtain
\begin{equation}\label{eq:Apply_JS_Lemma}
\begin{split}
\sum_{i,i':i+j<i'}\expect{ \Delta^n_{i,j} U\ \Delta^n_{i',j} U } & = \sum_{i,i':i+j<i'}\expect{\Delta^n_{i,j} U\ \conexp{  \Delta^n_{i',j} U }{\infor{(i+j)\Delta_n}}}\\
&\leq C\sum_i\sum_{i':i+j<i'}\sqrt{\expect{ \myp{\conexp{  \Delta^n_{i',j} U }{\infor{(i+j)\Delta_n}}}^2} }\\
& \leq C\sum_i\sum_{i':i+j<i'}(i'-(i+j))^{-v/2}\leq C\Delta_n^{-1}.
\end{split}
\end{equation}
Eqns. \eqref{eq:dX2_dU2_bound} and~\eqref{eq:dX2_dU2_cross_bound} imply that $\expect{\myp{\sum_{i=0}^{n-j}\Delta^n_{i,j} X\ \Delta^n_{i,j} U }^2}\leq Cj^2$, thus
\begin{equation}\label{eq:dXdU_order}
\sum_{i=0}^{n-j}\Delta^n_{i,j} X\ \Delta^n_{i,j} U = O_p(j).
\end{equation}

\item Turning to the last sum of \eqref{eq:decomp_[Y,Y]^j_n},
let $\nu_j := \expect{(U^n_{i+j} - U^n_i)^2}= 2(\var{U} - \gamma(j))$.
%Note that the moment restriction $\expect{|U|^{4+\epsilon}}<\infty$ implies $\var{(U^n_{j} - U^n_0)^2}<\infty$.
%Hence,
For $i>j$, we obtain the following in a similar way in which we derived~\eqref{eq:Apply_JS_Lemma}:
\begin{align*}
				\abs{\cov{(U^n_{j} - U^n_0)^2,(U^n_{i+j} - U^n_i)^2}}	\leq C(i-j)^{-v/2},
			%	=& \abs{\expect{[(U^n_{i+j} - U^n_i)^2-\nu_j][(U^n_{j} - U^n_0)^2 - \nu_j]}}\\
			%	= &\abs{\expect{[(U^n_{j} - U^n_0)^2 - \nu_j]\conexp{(U^n_{i+j} - U^n_i)^2-\nu_j}{\mathcal{F}_{\frac{j}{n}} }}}\\
			%	\leq & \sqrt{\var{(U^n_{j} - U^n_0)^2}} \sqrt{\expect{\myp{\conexp{(U^n_{i+j} - U^n_i)^2-\nu_j}{\mathcal{F}_{\frac{j}{n}} }}^2 } } \\
				%\leq & \abs{c_0\tilde{\rho}^{\frac{i-j}{2}} \sqrt{\var{(U^n_{j} - U^n_0)^2}}\expect{((U^n_{j} - U^n_0)^2 - \nu_j)}}\\
			 %\var{(U^n_{j} - U^n_0)^2}.
\end{align*}
				%for some constant $c_0$ and $\tilde{\rho}\in(0,1)$.
				%The last second inequality.
%The last inequality follows from Lemma~\ref{lemma:Jacod_Shiryave_VIII3_102}.
%Therefore, $\expect{\myp{\sum_{i=0}^{n-j}\myp{(\Delta^n_{i,j}U)^2-\nu_j} }^2}$ is uniformly bounded.
				%This proves the series $$\Omega^j_\infty := \var{(U^n_{j} - U^n_0)^2} + 2\sum_{i=1}^{\infty} \cov{(U^n_{j} - U^n_0)^2,(U^n_{i+j} - U^n_i)^2}$$ is absolutely summable, and it validates the following limit distribution\footnote{This is a well-known result, see, among others,~\cite{hamilton1994time} Section 7.2.}:
				%we follow the standard formulas of mixing sums (see also~\cite{Ait-Sahalia2011DependentNoise}, Section 5) to get
			%	\begin{equation}\label{eq:Asy_distri_RV(U;j)}
			%	\sqrt{(n-j+1)}\myp{\frac{\sum_{i=0}^{n-j}(\Delta^n_{i,j}U)^2}{n-j+1} - 2(\var{U} - \gamma(j))}\convergeL % \normdist{0}{\Omega^j_{\infty}}.
			%	\end{equation}
which implies
\begin{equation}
\expect{\myp{\sum_{i=0}^{n-j}\myp{(\Delta^n_{i,j}U)^2-\nu_j}}^2}\leq C\Delta_n^{-1}j.
\label{eq:Asy_Orders_RV(U)_j}
\end{equation}
%whence
%\[
%\frac{\sum_{i=0}^{n-j}(\Delta^n_{i,j}U)^2}{2(n-j+1)}\Pconverge \var{U}-\gamma(j).
%\]
%
%			%Now we see that among the three summands in~\eqref{prop:RV_Estimate_var+cov(1)}, the first two have stochastic order $1/n$. Thus it follows that
%Now it follows immediately from Step (i) and~\eqref{eq:dXdU_order} that
%\begin{equation*}
%\frac{1}{2(n-j+1)}[Y,Y]^{j}_n\Pconverge \var{U} - \gamma(j).
%\end{equation*}
For any fixed $j$, any $j_n$ satisfying $\Delta_nj_n\rightarrow 0, j_n\rightarrow \infty$, we have by~\eqref{eq:dXdU_order},~\eqref{eq:Asy_Orders_RV(U)_j} and~\eqref{eq:rho_strong_mixing} that
\begin{equation}\label{eq:RV{Y,Y}_Op}
\begin{split}
&\widehat{\RV{Y,Y}}_n(j)-\myp{ \var{U}-\gamma(j)} = O_p\myp{\sqrt{\Delta_nj}};\\
&\widehat{\RV{Y,Y}}_n(j_n)-\var{U} = O_p\myp{\max\left\{\sqrt{\Delta_nj_n},j_n^{-v/2}\right\}}.
\end{split}
\end{equation}Now the stated results follow from~\eqref{eq:Asy_condi_RV_consistency}.
\end{enumerate}
\end{proof}

\setcounter{equation}{0}
		
%\section{Proof of Proposition~\ref{prop:consistent_estimate_var_U_dependence}}\label{sec:prop:consistent_estimate_var_U_dependence}
%
%\begin{proof}
%The proof basically follows from the proof of Proposition~\ref{prop:RV_Estimate_var+cov(1)} and the asymptotic conditions imposed on $j_n$ and $n$.
%A decomposition of $[Y,Y]^{j_n}_n$ into three parts as in~\eqref{eq:decomp_[Y,Y]^j_n} is of course still true; $\sum_{i=0}^{n-j_n}(\Delta^n_{i,j_n} X)^2/j_n\Pconverge [X,X]$ holds; $\sum_{i=0}^{n-j_n}\Delta^n_{i,j_n} X\ \Delta^n_{i,j_n} U = O_p(\sqrt{{j_n}})$ holds as well;
%the asymptotic orders as in~\eqref{eq:Asy_Orders_RV(U)_j} are still true.
%Our asymptotic conditions $\rho^{j_n}\sqrt{n} = O(1)$ and $n/j_n\rightarrow \infty$ imply $\gamma(j_{n}) = O_p((n-j_n+1)^{-1/2})$.
%This finishes the proof.
%%s due to the fact $\frac{n}{j_n}\rightarrow \infty$. The proof is complete.
%\end{proof}

\setcounter{equation}{0}
		
\section{Proof of Proposition~\ref{prop:Finite_Sample_Bias_Correction}}
\label{sec:prop:Finite_Sample_Bias_Correction}

\begin{proof}
%Recall
%\begin{align*}
%[Y,Y]^j_n = [X,X]^j_n + [U,U]^j_n
%\end{align*}
Let $k= \lfloor \frac{n}{j}\rfloor$.
We will adopt the square bracket notation in \eqref{eq:decomp_[Y,Y]^j_n} for $X$ and $U$ as well.
By It\^o's isometry, we have
\begin{align*}
\expectsigma{[X,X]^j_{kj-1}} &= \sum_{i=0}^{j-1}\int_{i\Delta_n}^{\myp{(k-1)j+i}\Delta_n}\sigma^2_s\diff s =\sum_{i=0}^{j-1}\myp{\int_{0}^{kj\Delta_n}\sigma^2_s\diff s -\int_{0}^{i\Delta_n} \sigma^2_s\diff s - \int_{\myp{(k-1)j+i}\Delta_n}^{kj\Delta_n} \sigma^2_s\diff s}\\
& = j\int_{0}^{kj\Delta_n}\sigma^2_s\diff s +O_p(j^2\Delta_n).
\end{align*}
Hence, we have
\begin{align*}
\expectsigma{[X,X]^j_n} = j\int_{0}^{1}\sigma^2_s\diff s + O_p(j^2\Delta_n),
\end{align*}
where the stochastic orders follow from the regularity conditions of the volatility path at 0 and 1.
Furthermore, it is immediate that $\expectsigma{[U,U]^j_n} = 2(n-j+1)(\var{U}-\gamma(j)).$
Thus, we have, by the independence of $X$ and $U$,
\begin{align*}
\expectsigma{\widehat{\RV{Y,Y}}_n(j)} = \frac{j\int_{0}^{1}\sigma^2_s\diff s}{2(n-j+1)} + \var{U} - \gamma(j) + O_p(j^2\Delta_n^2).
\end{align*}
\end{proof}

\setcounter{equation}{0}
		
\section{Proof of Proposition~\ref{prop:pre-averaged noise}}\label{appendix:prop:pre-averaged noise}

\begin{proof}[Proof of Proposition~\ref{prop:pre-averaged noise}]
Recall that
\begin{align*}
	\preavg{U} & = \frac{1}{k_n+1}\sum_{i=(2m-2)k_n}^{(2m-1)k_n}\myp{U^n_{i+k_n} - U^n_{i}} \\&= \frac{1}{k_n+1}\myp{\sum_{i=(2m-1)k_n}^{2mk_n}U^n_{i}
		- \sum_{i=(2m-2)k_n}^{(2m-1)k_n}U^n_{i}}.
\end{align*} %Note that by our construction of $k_n$ in~\eqref{eq:k_nM_condition}, $\frac{(m-1)n}{M}+k_n = 2mk_n - k_n$.
Also recall that $U$ is symmetrically distributed around 0,  whence $\preavg{U}$ is equal to the following in distribution:
\begin{align}\label{eq:pre_U_order}
	\preavg{U} \overset{d}{ = }\frac{1}{k_n+1}\myp{\sum_{i=(2m-2)k_n}^{2mk_n}U^n_{i}} + O_p(\sqrt{\Delta_n}).
\end{align}
Since $v>2$, we have $\sigma_U^2<\infty$, and an application of Corollary VIII 3.106 of~\cite{jacod1987limit} yields
%The mixing coefficients of the noise process are decaying exponentially,
%hence we can apply Lemma~\ref{thm:clt-alpha-mixing} to $\sum_{i=(2m-2)k_n}^{2mk_n}U^n_{i}$:
\[
	\frac{1}{\sqrt{2k_n + 1}}\sum_{i=(2m-2)k_n}^{2mk_n}U^n_{i}\convergeL \normdist{0}{\sigma^2_U},
\]
whence
\[
	n^{1/4}\preavg{U}\convergeL \normdist{0}{2\sigma^2_U/c}.
\]
\end{proof}

%		\section{Proof of Proposition~\ref{prop:consistent_estimate_sigma2U_general}}\label{appendix:prop:consistent_estimate_sigma2U_general}
%		\begin{proof}
%			Assumption 1 enables us to apply the Inverse Mapping Theorem. The result is then validated by applying the Continuous Mapping Theorem (see Theorem 2.3 in~\cite{VanderVaart_AsymptoticStatistics}) over the results from Proposition~\ref{prop:RV_Estimate_var+cov(1)} and Proposition~\ref{prop:consistent_estimate_var_U_dependence}, which follows from Assumption 2.
%		\end{proof}

\setcounter{equation}{0}
		
\section{Proof of Proposition~\ref{prop:consistent_nonpar_sigma2U}}
\label{appendix:prop:consistent_nonpar_sigma2U}

\begin{proof}
For any fixed $j$, ~\eqref{eq:RV{Y,Y}_Op} implies $\widehat{\gamma(j)}_n - \gamma(j) =O_p\myp{\max\left\{\sqrt{\Delta_nj_n},j_n^{-v/2}\right\}}$.
%From the proofs of Propositions~\ref{prop:RV_Estimate_var+cov(1)} and~\ref{prop:consistent_estimate_var_U_dependence},
%we know that, for $j\leq i_n$,
%	\begin{align*}
%			&\widehat{\RV{Y,Y}}_n(j) = \var{U} - \gamma(j) + O_p(\sqrt{\Delta_n});\quad\widehat{\RV{Y,Y}}_n(i_n) = \var{U}  + O_p(\sqrt{\Delta_n}).
%	\end{align*}
Therefore,
\begin{align*}
\widehat{\sigma^2_U} -  \sum_{j=-i_n}^{i_n}\gamma(j) = O_p\myp{\max\left\{\sqrt{\Delta_nj_ni_n^2},j_n^{-v/2}i_n\right\}}.
\end{align*}
Now the result follows given that $\Delta_nj_n^3\rightarrow 0, i_n\leq j_n,i_n\rightarrow \infty, v>2.$
%			\begin{align*}
%			\widehat{\sigma^2_U} & = (2i_n+1)\widehat{\RV{Y,Y}}_n(i_n)-2\sum_{j=1}^{i_n}\widehat{\RV{Y,Y}}_n(j)\\
%			& = \var{U}+2\sum_{j=1}^{i_n}\gamma(j)+O_p\myp{i_n/\sqrt{n}}\\
%			& = \var{U} + 2\sum_{j=1}^{\infty}\gamma(j) + O(\rho^{i_n})+O_p\myp{i_n/\sqrt{n}}\\
%			& = \sigma^2_U  +O_p\myp{i_n/\sqrt{n}}.
%			\end{align*}
%Hence, we conclude that $\widehat{\sigma^2_U} \Pconverge\sigma^2_U$.
		\end{proof}

\setcounter{equation}{0}

\section{Proof of Theorem~\ref{thm:consistency}}\label{appendix:thm_consistency}

The proof of this theorem basically follows~\cite{podolskij2009pre-averaging-1},
but we need to deal with generally dependent noise.
			
First, we introduce some notation:
\begin{align}
		&\beta^n_m := n^{1/4}\myp{\sigma_{\frac{m-1}{M_n}}\preavg{W} + \preavg{U}}\label{eq:beta^n_m};\\
		%&\widetilde{\beta^n_m}  = n^{1/4}\myp{\sigma_{\frac{m-1}{M_n}}\preavglag{W} + \preavglag{U}};\\
		&\xi^n_m :=n^{1/4}\preavg{Y} - \beta^n_m;\label{eq:xi^n_m}\\
		%&\widetilde{\xi^n_m} =n^{1/4}\preavglag{Y} - \widetilde{\beta^n_m};\\
		&\eta^n_m :=\frac{n^{r/4}}{2c}\conexp{\abs{\preavg{Y}}^r}{\mathcal{F}_{\frac{m-1}{M_n}}};\\
		&\widetilde{\eta^n_m} :=\frac{\mu_r}{2c}\myp{\frac{2c}{3}\sigma^2_{\frac{m-1}{M_n}} + \frac{2}{c}\sigma^2_U}^{\frac{r}{2}};\\
		&\Pbv^n:=\sum_{m=1}^{M_n}\eta^n_m;\\
		&\widetilde{\Pbv}^n :=\sum_{m=1}^{M_n}\widetilde{\eta^n_m}.
\end{align}
		%Note that the $\beta^n_m$ is an approximation of $n^{1/4}\preavg{Y}$ while $\xi^n_m$ is a measure of the approximation error; similar interpretation applies to $\widetilde{\beta^n_m} $ and $\widetilde{\xi^n_m}$.
Then, we state the following lemma:
\begin{lemma}\label{lemma:stochastic_Order_X_Y}
For any $q>0$, there is some constant $C_q>0$ (depending on $q$), such that $\forall m$:
	\begin{equation}
	\label{eq:StochasticOrderXbar}
	\expect{\abs{\xi^n_m}^q} + \expect{\abs{n^{1/4}\preavg{X}}^q}<C_q;
	\end{equation}
and the following holds for $q\in (0,2r+\varepsilon)$ with $\varepsilon$ as defined in Theorem~\ref{thm:consistency}:
	\begin{equation}
	\label{eq:StochasticOrderYbar}
	\expect{\abs{\beta^n_m}^q} + \expect{\abs{n^{1/4}\preavg{Y}}^q}<C_q.
	\end{equation}
\end{lemma}
\begin{proof}
[Proof of Lemma~\ref{lemma:stochastic_Order_X_Y}]
The boundedness of moments of $\xi^n_m$ and $n^{1/4}\preavg{X}$ (which don't depend on the noise)
follows from Lemma 1 in~\cite{podolskij2009pre-averaging-1}.
			
Now we show the boundedness of $\expect{\abs{n^{1/4}\preavg{Y}}^q}$ for $0<q<2r+\varepsilon$.
We note (see Proposition 3.8 in~\cite{white2000asymptotic}) that there is some $C_q$ so that the following is true:
\[
		\expect{\abs{n^{1/4}\preavg{Y}}^q}\leq C_q\myp{\expect{\abs{n^{1/4}\preavg{X}}^q} + \expect{\abs{n^{1/4}\preavg{U}}^q}}.
\]
Boundedness of $\expect{\abs{n^{1/4}\preavg{X}}^q}$ has already been established,
while $\expect{\abs{n^{1/4}\preavg{U}}^q}$ is bounded by Proposition~\ref{prop:pre-averaged noise}
and a well known fact that convergence in distribution implies convergence in moments under uniformly bounded moments condition, see, e.g., Theorem 4.5.2 of~\cite{chung2001course}. %Lemma~\ref{lemma:Thm452_Chung} and the condition on $q$. %Note that this is the key difference from the original proofs in~\cite{podolskij2009pre-averaging-1} and~\cite{jacod2009pre-averaging-2}.
A similar proof holds for $\expect{\abs{\beta^n_m}^q}$.% and $\expect{\abs{\widetilde{\beta^n_m}}^q}$.
\end{proof}
		
\begin{proof}[Proof of Theorem~\ref{thm:consistency}]
We present the proof in several steps.
\begin{enumerate}[(i)]
\item We first prove that
\begin{equation}
\label{eq:Mgl_diff_Pconverge_0}
\Pbv(Y,r)_n\ - \frac{1}{M_n}\Pbv^n\Pconverge 0.
\end{equation}
First, recall our choice of $M_n = \left\lfloor\frac{\sqrt{n}}{2c}\right\rfloor$.
Next, observe that the difference on the left-hand side of \eqref{eq:Mgl_diff_Pconverge_0} is in fact a sum of \emph{martingale differences}:
\begin{align*}
&\Pbv(Y,r)_n\ - \frac{1}{M_n}\Pbv^n \\
=& \sum_{m=1}^{M_n}\frac{1}{\sqrt{n}}\myp{\abs{n^{\frac{1}{4}}\preavg{Y}}^r-\conexp{\abs{n^{\frac{1}{4}}\preavg{Y}}^r }{\mathcal{F}_{\frac{m-1}{M_n}}}}.
\end{align*}
In light of Lemma 2.2.11 in~\cite{jacod2011discretization}, it suffices to show that
\begin{align}\label{eq:squared_Mgl_Pconverge_0}
\frac{1}{n}\sum_{m=1}^{M_n}\conexp{\abs{n^{\frac{1}{4}}\preavg{Y}}^{2r}}{\mathcal{F}_{\frac{m-1}{M_n}}}\Pconverge 0.
\end{align}
But this follows from the boundedness established in Lemma~\ref{lemma:stochastic_Order_X_Y} and the choice of $M_n$.
%				An application of H\"older's inequality yields
%				\begin{align*}
%				\expect{\abs{n^{\frac{1}{4}}\preavg{Y}}^{2r}}
%				\leq \myp{\expect{\abs{n^{\frac{1}{4}}\preavg{Y}}^{2(r+l)}}}^{\frac{l}{r+l}}\myp{\expect{\abs{n^{\frac{1}{4}}\preavglag{Y}}^{2(r+l)}}}^{\frac{r}{r+l}}.
%				\end{align*}And we already proved that the two terms
 %The RHS of the inequality are uniformly bounded in $m$ in Lemma~\ref{lemma:stochastic_Order_X_Y}. Since all terms are nonnegative, the conditional expectation is also uniformly bounded whence the LHS of~\eqref{eq:squared_Mgl_Pconverge_0} is of order $O_p(n^{-1/2})$ by our choice of $M$. This proves~\eqref{eq:Mgl_diff_Pconverge_0}.
\item Next, we prove that
\begin{equation}\label{eq:MBV-MBVtilde=0}
\frac{1}{M_n}\Pbv^n-\frac{1}{M_n}\widetilde{\Pbv}^n\Pconverge 0.
\end{equation}
To prove this, we proceed in several steps:
\begin{enumerate}
\item We first note that the error of approximating $n^{1/4}\preavg{Y}$ by $\beta^n_m$,
denoted by $\xi^n_m$ in~\eqref{eq:xi^n_m}, is small in the sense that
\begin{equation}
\label{eq:approxi_err_small}
\frac{1}{M_n}\sum_{m=1}^{M_n}\expect{\abs{\xi^n_m}^2}\rightarrow 0.
\end{equation} %and
%					\begin{equation}
%					\label{eq:approxi_lag_err_small}
%					\frac{1}{M_n}\sum_{m=1}^{M_n}\expect{\abs{\xi^n_{m+1}}^2}\rightarrow 0
%					\end{equation}
For a detailed proof, see~\cite{podolskij2009pre-averaging-1}.
(Note that our assumptions on the noise process are different from~\cite{podolskij2009pre-averaging-1},
but the noise terms don't appear in $\xi^{n}_{m}$.)
%					\item Now we show \begin{equation}
%					\label{eq:approxi_err2_small}
%					\frac{1}{M_n}\sum_{m=1}^{M_n}\expect{\abs{\beta^n_{m+1}-\widetilde{\beta^n_m}}^2}\rightarrow 0.
%					\end{equation}This is easily done by
%					\begin{align*}
%					\frac{1}{M_n}\sum_{m=1}^{M_n}\expect{\abs{\beta^n_{m+1}-\widetilde{\beta^n_m}}^2}
%					= & \frac{1}{M_n}\sum_{m=1}^{M_n}\expect{\abs{n^{1/4}\myp{\sigma_{\frac{m}{M_n}} - \sigma_{\frac{m-1}{M_n}} }\preavglag{W}}^2}\\
%					\overset{\eqref{eq:asym_var_preW_preU}}{\leq } & \frac{C}{M_n}\sum_{m=1}^{M_n}\expect{\abs{\sigma_{\frac{m}{M_n}} - \sigma_{\frac{m-1}{M_n}} }^2}\\
%					\leq  & \frac{C}{M_n}\sum_{m=1}^{M_n}\myp{\expect{\abs{\sigma_{\frac{m}{M_n}} - \sigma_s}^2}+\expect{\abs{\sigma_{\frac{m-1}{M_n}}-\sigma_s }^2}}.
%					\end{align*}Then~\eqref{eq:approxi_err2_small} follows by rewriting the integration and  applying the Lebesgue's Dominated Convergence Theorem again, as we did in previous part.
\item Next, define the approximation error %of by modulated bi-power variation as
\begin{align*}
\zeta^n_m := \frac{\abs{n^{1/4}\preavg{Y}}^r - \abs{\beta^n_m}^r}{2c}.
\end{align*}
We note that this error is also small:
\begin{equation}\label{eq:approxi_err_zeta}
\frac{1}{M_n}\sum_{m=1}^{M_n}\expect{\abs{\zeta^n_m}} \rightarrow 0,
\end{equation}
which follows from
\begin{equation}\label{eq:approxi_err_zeta2}
\frac{1}{M_n}\sum_{m=1}^{M_n}\expect{\abs{\zeta^n_m}^2} \rightarrow 0.
\end{equation}
This, in turn, can be proved following~\cite{podolskij2009pre-averaging-1}.
					%Now we can apply Lemma 5.4 in~\cite{barndorff2006central} since we proved~\eqref{eq:approxi_err_small},~\eqref{eq:approxi_lag_err_small}, %, and~\eqref{eq:approxi_err2_small};
					 %also we need the results from Lemma~\ref{lemma:stochastic_Order_X_Y}.
\eqref{eq:approxi_err_zeta} then follows, and it implies
\begin{equation}\label{eq:zeta_P_converge_0}
\frac{1}{M_n}\sum_{m=1}^{M_n}\conexp{\zeta^n_m}{\mathcal{F}_{\frac{m-1}{M_n}}} \Pconverge 0,
\end{equation}
by the Markov inequality.
					
\item Now we show the following:
\begin{equation}
\label{eq:conditional_PBV_op(1)}
\conexp{\abs{\beta^n_m}^r}{\mathcal{F}_{\frac{m-1}{M_n}}} ={\mu_r}\myp{\frac{2c}{3}\sigma^2_{\frac{m-1}{M_n}} + \frac{2\sigma^2_U}{c}}^{\frac{r}{2}}+o_p(1),
\end{equation}
which holds uniformly in $m$.
Recall that $r\geq 2$ is an even integer.
Let $r_n\rightarrow\infty$ but $r_n=o(n^{1/2})$.
Denote %To simplify notations, we let $s^n_m:=(2m-2)k_n+r_n$ and denote
\begin{align*}
\overline{\beta}^{n}_{m-1,r_n} & = \frac{n^{1/4}}{k_n+1}\myp{\sum_{i=(2m-2)k_n}^{(2m-2)k_n+r_n}\sigma_{\frac{m-1}{M_n}}\myp{W^n_{i+k_n}   - W^n_{i}} +\myp{U^n_{i+k_n}  -U^n_{i} }}\\
&=: n^{1/4}\myp{\sigma_{\frac{m-1}{M_n}} \overline{W}^{n}_{m-1,r_n} + \overline{U}^{n}_{m-1,r_n} };\\
\overline{\beta}^{n}_{r_n,m} & = \frac{n^{1/4}}{k_n+1}\myp{\sum_{i=(2m-2)k_n+r_n+1}^{(2m-1)k_n}\sigma_{\frac{m-1}{M_n}}\myp{W^n_{i+k_n}   - W^n_{i}} +\myp{U^n_{i+k_n}  -U^n_{i} }}\\
&=: n^{1/4}\myp{\sigma_{\frac{m-1}{M_n}} \overline{W}^{n}_{r_n,m} + \overline{U}^{n}_{r_n,m} }.
%\widetilde{\beta}^{n}_{m,r_n} & = \frac{n^{1/4}}{k_n+1}\myp{\sum_{i=mk_n}^{2mk_n+r_n}\sigma_{\frac{m-1}{M_n}}\myp{W^n_{i+k_n}   - W^n_{i}} +\myp{U^n_{i+k_n}  -U^n_{i} }}\\
					%&=: n^{1/4}\myp{\sigma_{\frac{m-1}{M_n}} \overline{W}^{n}_{m,r_n} + \overline{U}^{n}_{m,r_n} }.\\
					%		 \widetilde{\beta}^{n}_{r_n,m+1} & = \frac{n^{1/4}}{k_n+1}\myp{\sum_{i=2mk_n+r_n+1}^{(2m+1)k_n}\sigma_{\frac{m-1}{M_n}}\myp{W^n_{i+k_n}   - W^n_{i}} +\myp{U^n_{i+k_n}  -U^n_{i} }}\\
					%		 &=: n^{1/4}\myp{\sigma_{\frac{m-1}{M_n}} \overline{W}^{n}_{r_n,m+1} + \overline{U}^{n}_{r_n,m+1} }.		
\end{align*}
Then, we have $\beta^n_m =\overline{\beta}^{n}_{m-1,r_n} + \overline{\beta}^{n}_{r_n,m}$.
Furthermore, by our construction, $\overline{\beta}^{n}_{m-1,r_n}=o_p(1)$ and $\overline{\beta}^{n}_{r_n,m}$ has the same asymptotic distribution as $\beta^n_m$, which can be derived from the asymptotic distributions of $n^{1/4}\overline{U}^n_m$ and $n^{1/4}\overline{W}^n_m$, and the independence assumption between $X$ and $U$.
                    %Intuitively $\overline{\beta}^{n}_{m-1,r_n}$ is a negligible term, which can be dropped to render some asymptotic independence to deal with dependent noise.
%					
					
By the Mean Value Theorem, we have
\begin{align*}
\conexp{\myp{\beta^n_m }^r-\myp{\overline{\beta}^n_{r_n,m} }^r}{\infor{\frac{m-1}{M_n}}}
=   \conexp{r\myp{\overline{\beta}^n_{r_n,m}}^{r-1}\myp{\overline{\beta}^n_{m-1,r_n}} }{\infor{\frac{m-1}{M_n}}} + o_p(1).
%& + C \conexp{\abs{\overline{\beta}^n_{r_n,m} }^r\myp{\abs{\widebar{\beta}^n_{m,r_n}}^l + \abs{\widetilde{\beta}^n_{r_n,m+1}}^{l-1}\abs{\widetilde{\beta}^n_{m,r_n}} } }{\infor{\frac{m-1}{M_n}}}
\end{align*}
The moment conditions and an application of Cauchy-Schwarz inequality yields %$\expect{\abs{\overline{\beta}^n_{m-1,r_n}}^r}\rightarrow 0$ and by
%we have
$$\conexp{\myp{\overline{\beta}^n_{r_n,m}}^{r-1}\myp{\overline{\beta}^n_{m-1,r_n}} } {\infor{\frac{m-1}{M_n}}}=o_p(1).$$
Thus,
\begin{equation}\label{eq:r_big_blk_op(1)_1}
\conexp{\myp{\beta^n_m }^r}{\infor{\frac{m-1}{M_n}}} = \conexp{\myp{\overline{\beta}^n_{r_n,m} }^r}{\infor{\frac{m-1}{M_n}}}  + o_p(1).
\end{equation}
					
For any $l\leq r$, define $\overline{U}^{n,l}_{r_n,m}:=\myp{n^{1/4}\overline{U}^n_{r_n,m}}^{l}$, and let
$$C_l: = \expect{\myp{\conexp{\overline{U}^{n,l}_{r_n,m}}{\infor{\frac{m-1}{M_n}}} - \expect{\overline{U}^{n,l}_{r_n,m}}}^2}.$$
By the JS-Lemma, we have $C_l\leq Cr_n^{-v}$. Let $$\Lambda_l: = \frac{\conexp{\overline{U}^{n,l}_{r_n,m}}{\infor{\frac{m-1}{M_n}}} - \expect{\overline{U}^{n,l}_{r_n,m}}}{\sqrt{C_l}};$$ note that $\expect{\Lambda_l^2} = 1$. Thus,
\begin{equation}
\conexp{\overline{U}^{n,l}_{r_n,m}}{\infor{\frac{m-1}{M_n}}}=\expect{\overline{U}^{n,l}_{r_n,m}} + \sqrt{C_l}\Lambda_l.
\end{equation}
Therefore, we can substitute the conditional moments by the unconditional moments and we obtain the following ($C^k_r =\frac{r!}{k!(r-k)!}$ denotes the binomial coefficient):
\begin{align*}
&\conexp{\myp{\overline{\beta}^n_{r_n,m} }^r}{\infor{\frac{m-1}{M_n}}}\\
=& \conexp{\sum_{k=0}^{r} C^k_r\sigma^k_{\frac{m-1}{M_n}} \myp{n^{1/4}\overline{W}^n_{r_n,m}}^k \myp{n^{1/4}\overline{U}^n_{r_n,m}}^{r-k}}{\infor{\frac{m-1}{M_n}}} \\
=& \sum_{k=0}^{r} C^k_r\sigma^k_{\frac{m-1}{M_n}} \conexp{\myp{n^{1/4}\overline{W}^n_{r_n,m}}^k}{\sigma_{\frac{m-1}{M_n}}} \conexp{\myp{n^{1/4}\overline{U}^n_{r_n,m}}^{r-k}}{\infor{\frac{m-1}{M_n}}}\\
=& \conexp{\myp{\overline{\beta}^n_{r_n,m} }^r}{{\sigma_{\frac{m-1}{M_n}}}} +\sum_{k=0}^{r} C_k^r\sigma^k_{\frac{m-1}{M_n}} \conexp{\myp{n^{1/4}\overline{W}^n_{r_n,m}}^k}{\sigma_{\frac{m-1}{M_n}}} \sqrt{C_{r-k}}\Lambda_{r-k}.
\end{align*}
Clearly, the last term is $o_p(1)$, and together with~\eqref{eq:r_big_blk_op(1)_1}, we have
\begin{equation}\label{eq:key_condi_approxi}
\begin{split}
\conexp{\myp{\beta^n_m }^r}{\infor{\frac{m-1}{M_n}}} & = \conexp{\myp{\overline{\beta}^n_{r_n,m} }^r}{\sigma_{\frac{m-1}{M_n}}} +o_p(1)\\
&  = \mu_r\myp{\frac{2c}{3}\sigma^2_{\frac{m-1}{M_n}} + \frac{2\sigma^2_U}{c}}^{\frac{r}{2}}+o_p(1).
\end{split}
\end{equation}
The last equality is a consequence of the asymptotic distribution of $\beta^n_m$.%, the moments condition and Lemma~\ref{lemma:Thm452_Chung}.
					
\item Now \eqref{eq:MBV-MBVtilde=0} follows from~\eqref{eq:zeta_P_converge_0} and~\eqref{eq:key_condi_approxi}.
\end{enumerate}
\item Following Proposition 2.2.8 in~\cite{jacod2011discretization}, we see that the Riemann approximation converges:
\begin{equation}
\frac{1}{M_n}\sum_{m=1}^{M_n}\widetilde{\Pbv^n}\Pconverge \Pbv(Y,r).
\label{eq:Riem}
\end{equation}
Recall that we already proved that
\[
\Pbv(Y,r)_n\ - \frac{1}{M_n}\Pbv^n\Pconverge 0;\quad\mathrm{and}\quad
\frac{1}{M_n}\Pbv^n-\frac{1}{M_n}\widetilde{\Pbv}^n\Pconverge 0;
\]
in previous steps.
Now it is immediate to conclude that
\[
\Pbv(Y,r)_n\Pconverge \Pbv(Y,r).
\]
This finalizes the proof of Theorem~\ref{thm:consistency}.
\end{enumerate}
\end{proof}

\setcounter{equation}{0}

\section{Robustness to Irregular Sampling}\label{sec:IrregularSampling}

In this section, we show that the consistency results for integrated volatility
in Theorem~\ref{thm:consistency} and Corollary~\ref{corollary:consistency_IV}
can be extended to irregular sampling times for the case $r=2$,
by adapting the approach in Appendix C of~\cite{christensen2014jump-high-fre}
to allow for serially dependent noise in our general setting (recall $Y^n_i = X_{t^n_i} + U^n_i$).
Let $f:[0,1]\mapsto[0,1]$ be a strictly increasing map with Lipschitz continuous first order derivatives.  Let $f(0)=0$ and $f(1)=1$.
Suppose that the observation times are $\{t^n_i=f(i/n):0\leq i\leq n \}$.
Let $C_f' = \max_{x\in[0,1]}\abs{f'(x)}$.
Note that $C'_f<\infty$ by the continuity of $f'$.

First, we note that the asymptotic results related to the noise process we derived so far
still hold under irregular sampling,
because the noise is indexed by $i$ rather than by $t_{i}$ in our setting.
The proof then proceeds in several steps:

\begin{enumerate}
\item We first provide the analogs of Lemma~\ref{lemma:stochastic_Order_X_Y}
and step (i) in the proof of Theorem~\ref{thm:consistency}.
Assume $q\geq 1$.
Then,
\begin{align*}
	\expect{\abs{\xi^n_m}^q}& = \expect{\abs{\frac{n^{1/4}}{k_n + 1}\sum_{i=(2m-2)k_n}^{(2m-1)k_n}X^n_{i+k_n} - X^n_{i} -\sigma_{t^n_{(2m-2)k_n}} \myp{W^n_{i+k_n} - W^n_{i}} }^q }\\
	&\leq\frac{n^{\frac{q}{4}}}{k_n+1}\sum_{i=(2m-2)k_n}^{(2m-1)k_n}\expect{\abs{X^n_{i+k_n} - X^n_{i} -\sigma_{t^n_{(2m-2)k_n}} \myp{W^n_{i+k_n} - W^n_{i}}}^q}\\
	&=\frac{n^{\frac{q}{4}}}{k_n+1}\sum_{i=(2m-2)k_n}^{(2m-1)k_n}\expect{\abs{
			\int_{t^n_i}^{t^n_{i+k_n}} \myp{\alpha_s\diff s + \myp{\sigma_s - \sigma_{t^n_{(2m-2)k_n}} }\diff W_s}}^q}\\
	&\leq C_\alpha (C_f')^q n^{-\frac{q}{4}}+ \frac{C_qn^{\frac{q}{4}}}{k_n+1}\sum_{i=(2m-2)k_n}^{(2m-1)k_n}\expect{\abs{
			\int_{t^n_i}^{t^n_{i+k_n}}  \myp{\sigma_s - \sigma_{t^n_{(2m-2)k_n}} }\diff W_s}^q}\\
	&\leq C+ \frac{C_qn^{\frac{q}{4}}}{k_n+1}\sum_{i=(2m-2)k_n}^{(2m-1)k_n}\expect{\myp{
			\int_{t^n_i}^{t^n_{i+k_n}}  \abs{\sigma_s - \sigma_{t^n_{(2m-2)k_n}} }^2\diff s}^{q/2}}\\
	&\leq C.
\end{align*}
The second inequality follows from the boundedness of $\alpha$ and $C_f'$.
The third inequality is an application of the Burkholder-Davis-Gundy inequality.
The last inequality follows from the fact that $\sigma$ is bounded.
Similarly, we can prove that $\expect{\abs{n^{1/4}\preavg{X}}^q}$ is bounded.
For $q\in(0,1)$, the result is immediate using Jensen's inequality.
Now the boundedness of $\expect{\abs{n^{1/4}\preavg{Y}}^q}$, $q\in (0,2r+\varepsilon)$,
is obvious as the asymptotic distribution of the pre-averaged noise (which is indexed by $i$) does not change under irregular sampling.
\item Next, we prove the analog of step (ii) item (a) in the proof of Theorem~\ref{thm:consistency}.
We have that
\begin{align*}
	&\expect{\abs{\xi^n_m}^2}\\
	\leq& \sum_{i=(2m-2)k_n}^{(2m-1)k_n}\frac{\expect{\abs{n^{\frac{1}{4}}\myp{  \myp{X^n_{i+k_n} - X^n_{i} } -\sigma_{t_{(2m-2)k_n}^n}\myp{W^n_{i+k_n} - W^n_{i} }} }^2}}{k_n+1}\\
	=& \sum_{i=(2m-2)k_n}^{(2m-1)k_n}\frac{\expect{\abs{n^{\frac{1}{4}}\myp{  \int_{t^n_i}^{t^n_{i+k_n}} \alpha_s\diff s + \int_{t^n_i}^{t^n_{i+k_n}} \myp{\sigma_s -\sigma_{t_{(2m-2)k_n}^n}} \diff W_s }} }^2}{k_n+1}\\
	\leq &\sum_{i=(2m-2)k_n}^{(2m-1)k_n}\frac{2\expect{n^{\frac{1}{2}}\myp{  \int_{t^n_i}^{t^n_{i+k_n}} \alpha_s\diff s}^2 +n^{1/2} \int_{t^n_i}^{t^n_{i+k_n}} \myp{\sigma_s -\sigma_{t_{(2m-2)k_n}^n}}^2 \diff s } }{k_n+1}\\
	%\leq &\frac{{C_f'}^2C_\alpha}{\sqrt{n}} +\frac{ \sum_{i=(2m-2)k_n}^{(2m-1)k_n} 2\expect{n^{1/2}\int_{t^n_i}^{t^n_{i+k_n}} \myp{\sigma_s -\sigma_{t_{(2m-2)k_n}^n}}^2 \diff s } }{k_n+1}\\
	\leq &\frac{{C_f'}^2C_\alpha}{\sqrt{n}} + 2n^{1/2}\expect{\int_{t^n_{(2m-2)k_n}}^{t_{2mk_n}^{n}} \myp{\sigma_s -\sigma_{t_{(2m-2)k_n}^n}}^2 \diff s } .
\end{align*}
The second inequality is due to the Cauchy's inequality and It\^o's isometry.
The third inequality is a consequence of the boundedness of $\alpha,|f'|$ and our choice of $k_n$;
it is obtained by taking $i$ to be the lower and upper bound.
Now we have
\begin{align*}
	\frac{1}{M_n}\sum_{m=1}^{M_n}\expect{\abs{\xi^n_m}^2}&\leq O(1/\sqrt{n}) + \frac{2n^{1/2}}{M_n}\sum_{m=1}^{M_n}\expect{\int_{t^n_{(2m-2)k_n}}^{t_{2mk_n}^{n}} \myp{\sigma_s -\sigma_{t_{(2m-2)k_n}^n}}^2 \diff s }\\
	& = O(1/\sqrt{n}) + 4c\int_{0}^1\expect{ \myp{\sigma_s -\sigma_{\frac{\lfloor M_ns \rfloor}{M_n}}}^2} \diff s.
\end{align*}
Since $\sigma_{\frac{\lfloor M_ns \rfloor}{M_n}}\rightarrow \sigma_s $-a.s., and $\sigma$ is bounded,
upon applying Lebesgue's Dominated Convergence Theorem, we obtain the analog of~\eqref{eq:approxi_err_small}.
We note that the analog of item (b) of step (ii) in the proof of Theorem~\ref{thm:consistency}
is directly obtained because (6.10) in~\cite{podolskij2009pre-averaging-1} holds.
\item We now provide the analog of~\eqref{eq:key_condi_approxi}.
First, we note that all the steps in proving~\eqref{eq:key_condi_approxi} hold
except those pertaining to the conditional variance of the pre-averaging Brownian motion.
Next, we show that
\begin{equation*}
\var{n^{1/4}\preavg{W}} = f'((2m-2)k_n/n)\frac{2c}{3}+o(1).
\end{equation*}
%In the following, we apply Mean Value Theorem to $f$:
By the Lipschitz continuity of $f'$ we obtain:
\begin{align*}
	&\var{\sum_{i=2(m-1)k_n}^{(2m-1)k_n}\myp{W^n_{i+k_n} - W^n_{i}}}\\
	 = &\sum_{i=2(m-1)k_n}^{(2m-1)k_n}\var{W^n_{i+k_n} - W^n_{i}} + \sum_{i\neq j}\cov{W^n_{i+k_n} - W^n_{i},W^n_{j+k_n} - W^n_{j}}\\
	 = &\sum_{i=(2m-2)k_n}^{(2m-1)k_n}(t^n_{i+k_n} - t^n_i) + 2\sum_{i=(2m-2)k_n}^{(2m-1)k_n-1}\sum_{j>i}(t^n_{i+k_n}-t^n_j)\\
	 = &\sum_{i=(2m-2)k_n}^{(2m-1)k_n}\myp{f'(i/n)\frac{k_n}{n}+o(k_n/n)} + 2\sum_{i=(2m-2)k_n}^{(2m-1)k_n-1}\sum_{j>i} \myp{f'(j/n)\frac{i+k_n-j}{n}+o(k_n/n)}\\
	 = & f'\myp{\frac{(2m-2)k_n}{n}}\sum_{i=(2m-2)k_n}^{(2m-1)k_n}\myp{\frac{k_n}{n}+o(k_n/n)} \\
	   & +  2\sum_{i=(2m-2)k_n}^{(2m-1)k_n-1}\sum_{j>i} \myp{\frac{i+k_n-j}{n}+o(k_n/n)}\\
	 =& f'\myp{\frac{(2m-2)k_n}{n}}\frac{2c^3\sqrt{n}}{3} + o(\sqrt{n}).
\end{align*}
Now the analog of~\eqref{eq:key_condi_approxi} (with $r=2$) is
\begin{equation}
	\conexp{(\beta^n_m)^2}{\infor{t^n_{(2m-2)k_n}}} = \myp{f'\myp{\frac{(2m-2)k_n}{n}}\sigma^2_{f\myp{\frac{(2m-2)k_n}{n}}}\frac{2c}{3} + \frac{2\sigma^2_U}{c}}+o_p(1).
\end{equation}
\item Finally, Riemann integrability yields the analog of~\eqref{eq:Riem}:
\begin{align*}
	\Pbv(Y,2)_n\Pconverge \int_{0}^{1}\myp{f'(s)\sigma^2_{f(s)}\frac{2c}{3} + \frac{2\sigma^2_U}{c}} \diff s = \int_{0}^{1}\myp{\frac{2c}{3}\sigma^2_t + \frac{2\sigma^2_U}{c}}\diff t.
\end{align*}
The last equality is due to the change of variable $f(s) = t$.
\end{enumerate}

\setcounter{equation}{0}

\section{Proof of Theorem~\ref{thm:CLT}}\label{appendix:thm_CLT}

We will first prove three lemmas.
Then Theorem~\ref{thm:CLT} follows as a consequence.
\begin{lemma}\label{lemma:1st_approxi}
We have that
\begin{align}\label{eq:CLT1stapproxi}
&\conexp{\myp{\beta^n_m}^2}{\infor{\frac{m-1}{M_n}}} = \myp{\frac{2c}{3}\sigma^2_{\frac{m-1}{M_n}}+\frac{2}{c}\sigma^2_U} + o_p(n^{-1/4}).
\end{align}
\end{lemma}
			
\begin{proof}
Let $r_n$ satisfy
\begin{equation}
\label{eq:r_n_Asy_cond}
r_n \asymp n^{\vartheta}, \quad \frac{1}{4v}<\vartheta<\frac{1}{4}.
\end{equation}
To simplify notation, we let $s^n_m:=(2m-2)k_n+r_n$,
and we recall our earlier notation used in the proof of Theorem~\ref{thm:consistency}:
\begin{align*}
	\overline{\beta}^{n}_{m-1,r_n} & = \frac{n^{1/4}}{k_n+1}\myp{\sum_{i=(2m-2)k_n}^{(2m-2)k_n+r_n}\sigma_{\frac{m-1}{M_n}}\myp{W^n_{i+k_n}   - W^n_{i}} +\myp{U^n_{i+k_n}  -U^n_{i} }}\\
	&=: n^{1/4}\myp{\sigma_{\frac{m-1}{M_n}} \overline{W}^{n}_{m-1,r_n} + \overline{U}^{n}_{m-1,r_n} };\\
	\overline{\beta}^{n}_{r_n,m} & = \frac{n^{1/4}}{k_n+1}\myp{\sum_{i=(2m-2)k_n+r_n+1}^{(2m-1)k_n}\sigma_{\frac{m-1}{M_n}}\myp{W^n_{i+k_n}   - W^n_{i}} +\myp{U^n_{i+k_n}  -U^n_{i} }}\\
	&=: n^{1/4}\myp{\sigma_{\frac{m-1}{M_n}} \overline{W}^{n}_{r_n,m} + \overline{U}^{n}_{r_n,m} },
\end{align*}
where $\overline{\beta}^{n}_{m-1,r_n}+\overline{\beta}^{n}_{r_n,m}=\beta^n_m $.
The proof consists of three steps:
\begin{enumerate}
\item We start by showing that
\begin{align}\label{eq:small_BLk_n=o_p(n^-1/4)}
\conexp{\myp{\beta^n_m}^2}{\infor{\frac{m-1}{M_n}}}-\conexp{\myp{\overline{\beta}^{n}_{r_n,m}}^2}{\infor{\frac{m-1}{M_n}}} = o_p(n^{-1/4}).
\end{align}
To prove~\eqref{eq:small_BLk_n=o_p(n^-1/4)}, we first prove that
\begin{equation}
\label{eq:CLT_small_blk2_op(n^-1/4)}
\conexp{\myp{\overline{\beta}^{n}_{m-1,r_n}}^2}{\infor{\frac{m-1}{M_n}}}=o_p(n^{-1/4}).
\end{equation}
%, which can be proved in the same way as we prove Proposition 4.4.
For this purpose, we show the following for any $k\leq i< j$:
\begin{equation}\label{eq:conditional_cov_decay_exp}
\expect{\abs{\conexp{U^n_{i}  U^n_{j}  }{\infor{\frac{k}{n}}}}}\leq C\myp{j-i}^{-v/2}.
\end{equation}
To see this, we apply JS-Lemma
to obtain that
\begin{equation*}
c_{ij}:=\expect{\myp{\conexp{U^n_j}{\infor{\frac{i}{n}}}}^2}\leq C\myp{j-i}^{-v}.
\end{equation*}
Then,
\begin{align*}
					\expect{\abs{\conexp{U^n_{i}U^n_{j}}{\infor{\frac{k}{n}}}}}& \leq \sqrt{C\myp{j-i}^{-v}} \expect{\abs{\conexp{U^n_i\frac{\conexp{U^n_j}{\infor{\frac{i}{n}}}}{\sqrt{c_{ij}}}  }{\infor{\frac{k}{n}}}}}.
\end{align*}
Now applying the Cauchy-Schwarz inequality
and using the fact that the variance of noise is bounded,
we obtain~\eqref{eq:conditional_cov_decay_exp}.
					%Since the unconditional covariance also decays exponentially, we see that
From~\eqref{eq:conditional_cov_decay_exp} and some simple algebra we find that
\begin{equation*}
\conexp{\myp{\sum_{i=(2m-2)k_n}^{s^n_m}   \sigma_{\frac{m-1}{M_n}}  \myp{W^n_{i+k_n} - W^n_{i}  } }^2}{\infor{\frac{m-1}{M_n}}}
\end{equation*}
is asymptotically much smaller than
\begin{align}
\conexp{\myp{\sum_{i=(2m-2)k_n}^{s^n_m}    \myp{U^n_{i+k_n} - U^n_{i}  } }^2}{\infor{\frac{m-1}{M_n}}} =O_p(r_n)=o_p(n^{1/4}),
\label{eq:CLT_small_blk2_noise_op(n^-1/4)}
\end{align}
%following the algebra in the proof of Proposition~\ref{prop:asym_var_preW_preU}. Now
whence~\eqref{eq:CLT_small_blk2_op(n^-1/4)} holds.
					
Next, we prove that
\begin{equation}\label{eq:CLT_cross_product_big&small_blk}
		\conexp{\myp{\overline{\beta}^{n}_{r_n,m}}\myp{\overline{\beta}^{n}_{m-1,r_n}}}{\infor{\frac{m-1}{M_n}}}=o_p(n^{-1/4}).
\end{equation}
				%	Due to the independence assumptions, we only need to evaluate
				%	\begin{enumerate}
				%	\item
(Note that the left-hand side of \eqref{eq:small_BLk_n=o_p(n^-1/4)}
is equal to the left-hand side of \eqref{eq:CLT_small_blk2_op(n^-1/4)}
plus twice the left-hand side of \eqref{eq:CLT_cross_product_big&small_blk}).
To show that
\begin{align*}
%\label{eq:CLTCross_small_big_blk_noise_op(n^-1/4)}
		\frac{n^{1/2}}{(k_n+1)^2}
		\conexp{
		\myp{\sum_{i=(2m-2)k_n}^{s^n_m}U^n_{i+k_n}  -U^n_{i} }
		\myp{\sum^{(2m-1)k_n}_{i=s^n_m+1}U^n_{i+k_n}  -U^n_{i}}}{\infor{\frac{m-1}{M_n}}}=o_p(n^{-1/4}),
\end{align*}
we first evaluate
\begin{align*}
	&	\frac{n^{1/2}}{(k_n+1)^2}\abs{\conexp{\myp{\sum_{i=(2m-2)k_n}^{s^n_m}U^n_{i+k_n}  }\myp{\sum^{(2m-1)k_n}_{j=s^n_m+1}U^n_{j+k_n} }   }{\infor{\frac{m-1}{M_n}}}}\\			 \leq&\frac{n^{1/2}}{(k_n+1)^2}\sum_{i=(2m-2)k_n}^{s^n_m}\sum^{(2m-1)k_n}_{j=s^n_m+1}\abs{\conexp{U^n_{i+k_n}U^n_{j+k_n}}{\infor{\frac{m-1}{M_n}}}}.
			         %\sum_{i=(2m-2)k_n}^{s^n_m}\sum^{(2m-1)k_n}_{j=s^n_m+1}c_1\delta^{i+j+2k_n-4(m-1)k_n}
					%\leq \frac{c_1\delta^{2k_n+r_n+1}}{(1-\delta)^2}.
\end{align*} %The second inequality follows from the proof of Proposition 4.5.
Now apply~\eqref{eq:conditional_cov_decay_exp} and by the fact that $v>4$, we have
\begin{align*}
\sum_{i=(2m-2)k_n}^{s^n_m}\sum^{(2m-1)k_n}_{j=s^n_m+1}\expect{\abs{\conexp{U^n_{i+k_n}U^n_{j+k_n}}{\infor{\frac{m-1}{M_n}}}}}&\overset{\eqref{eq:conditional_cov_decay_exp}}{\leq } \sum_{i=(2m-2)k_n}^{s^n_m}\sum^{(2m-1)k_n}_{j=s^n_m+1}C (j-i)^{-v/2}\\
&\leq C\sum_{\ell = 1}^{r_n}\ell^{1-\frac{v}{2}}\leq C	.
 %= O_p(1).
\end{align*}
					%the above terms are $o_p(n^{-1/4})$.
Similarly, we can prove that the other three cross products have the same order.
It is also easy to verify that
$$\frac{\sqrt{n}}{(k_n+1)^2}\expect{\sum_{i=(2m-2)k_n}^{s^n_m}(W^n_{i+k_n}-W^n_{i}) \sum^{(2m-1)k_n}_{j=s^n_m+1}(W^n_{j+k_n}-W^n_{j})} = O(r_n/\sqrt{n}).$$
Now~\eqref{eq:CLT_cross_product_big&small_blk} is proved
and consequently~\eqref{eq:small_BLk_n=o_p(n^-1/4)} follows
from~\eqref{eq:CLT_small_blk2_op(n^-1/4)} and~\eqref{eq:CLT_cross_product_big&small_blk}.
\item Next, we prove that
\begin{align}\label{eq:Diff_cond_uncond_o_p(n^-1/4)}
\conexp{\myp{\overline{\beta}^n_{r_n,m}}^2 }{\infor{\frac{m-1}{M_n}}} - \conexp{\myp{\overline{\beta}^n_{r_n,m}}^2}{\sigma_{\frac{m-1}{M_n}}} = o_p(n^{-1/4}).
\end{align}
For this purpose, we note that
\begin{align*}
&\frac{(k_n+1)^2}{\sqrt{n}} \abs{\conexp{\myp{\overline{\beta}^n_{r_n,m}}^2 }{\infor{\frac{m-1}{M_n}}} - \conexp{\myp{\overline{\beta}^n_{r_n,m}}^2}{\sigma_{\frac{m-1}{M_n}}}} \\
=&\abs{\myp{\conexp{\myp{\sum_{i=s^n_m+1}^{(2m-1)k_n} \myp{U^n_{i+k_n}  -U^n_{i} }}^2}{\infor{\frac{m-1}{M_n}}}-\expect{\myp{\sum_{i=s^n_m+1}^{(2m-1)k_n} \myp{U^n_{i+k_n}  -U^n_{i} }}^2}}}.
\end{align*}
Applying again the JS-Lemma, we find that
\begin{equation*}
\conexp{\myp{\overline{\beta}^n_{r_n,m}}^2 }{\infor{\frac{m-1}{M_n}}} - \conexp{\myp{\overline{\beta}^n_{r_n,m}}^2}{\sigma_{\frac{m-1}{M_n}}} =O_p(r_n^{-v}),
\end{equation*}
whence~\eqref{eq:Diff_cond_uncond_o_p(n^-1/4)} follows from~\eqref{eq:r_n_Asy_cond}.%exponentially by some $c\delta^{r_n}$.
					
\item Finally, we show that
\begin{align}
\label{eq:Diff_uncond_o_p(n^-1/4)}
\conexp{\myp{\overline{\beta}_{r_n,m}}^2}{\sigma_{\frac{m-1}{M_n}}} = \myp{\frac{2c}{3}\sigma^2_{\frac{m-1}{M_n}}+\frac{2}{c}\sigma^2_U} + o_p(n^{-1/4}).
\end{align}
This follows from the following equalities, which are straightforward:
\begin{align*}
\conexp{\myp{\frac{n^{1/4}}{k_n+1}\myp{\sum_{i=s^n_m+1}^{(2m-1)k_n}\sigma_{\frac{m-1}{M_n}}\myp{W^n_{i+k_n}   - W^n_{i}}}}^2}{\sigma_{\frac{m-1}{M_n}}}= \frac{2c}{3}\sigma^2_{\frac{m-1}{M_n}}+o_p(n^{-1/4}),
\end{align*}
\begin{align*}
\conexp{\myp{\frac{n^{1/4}}{k_n+1}\myp{\sum_{i=s^n_m+1}^{(2m-1)k_n}\myp{U^n_{i+k_n}   - U^n_{i}}}}^2}{\sigma_{\frac{m-1}{M_n}}}= \frac{2\sigma
^2_U}{c}+o_p(n^{-1/4}).
\end{align*}%can be obtained in a similar way.% and the unconditional version of~\eqref{eq:CLT_small_blk2_noise_op(n^-1/4)} and~\eqref{eq:CLTCross_small_big_blk_noise_op(n^-1/4)}.
\end{enumerate}
Now~\eqref{eq:CLT1stapproxi} follows from~\eqref{eq:small_BLk_n=o_p(n^-1/4)},~\eqref{eq:Diff_cond_uncond_o_p(n^-1/4)} and~\eqref{eq:Diff_uncond_o_p(n^-1/4)}, and the proof is complete.
\end{proof}			
			
\begin{lemma}
Let
\begin{align*}
L_n & :=n^{-1/4} \sum_{m=1}^{M_n}\myp{\myp{\beta^n_m}^2 -\conexp{\myp{\beta^n_m}^2}{\infor{\frac{m-1}{M_n}}}}.
\end{align*}
Then, we have the following stable convergence in law:
\begin{align}
\label{eq:CLT_stable_convergence}
L_n\LsConverge \sqrt{\frac{1}{c}}\int_0^1 \myp{\frac{2c}{3}\sigma_s^2+\frac{2\sigma^2_U}{c}}\diff W_s',
\end{align}
where $W'$ is a standard Wiener process independent of $\mathcal{F}$.	
\end{lemma}
					
\begin{proof}
Let $\theta^n_m  := n^{-1/4}\myp{\myp{{\beta^n_m}}^2 -\myp{\frac{2c}{3}\sigma^2_{\frac{m-1}{M_n}}+\frac{2}{c}\sigma^2_U} }.$
Then,
\begin{align*}
L_n = \sum_{m=1}^{M_n}\theta_m + o_p(1),
\end{align*}
by Lemma~\ref{lemma:1st_approxi}.
We also have
\begin{align}\label{eq:CLT_sum=P=>0}
	\sum_{m=1}^{M_n}\conexp{\theta^n_m}{\infor{\frac{m-1}{M_n}}}\Pconverge 0,
\end{align}
again by Lemma~\ref{lemma:1st_approxi} and
\begin{align*}
\sum_{m=1}^{M_n}\conexp{\myp{\theta^n_m}^2}{\infor{\frac{m-1}{M_n}}} =& \frac{1}{2cM_n}\sum_{m=1}^{M_n}    \conexp{\myp{{\beta^n_m}}^4}{\infor{\frac{m-1}{M_n}}}+\frac{1}{2cM_n}\sum_{m=1}^{M_n} \myp{\frac{2c}{3}\sigma^2_{\frac{m-1}{M_n}} +\frac{2\sigma^2_U}{c}}^2  \\
			& - \frac{1}{cM_n}\sum_{m=1}^{M_n}\conexp{\myp{{\beta^n_m}}^2}{\infor{\frac{m-1}{M_n}}}\myp{\frac{2c}{3}\sigma^2_{\frac{m-1}{M_n}} +\frac{2\sigma^2_U}{c}}.
\end{align*}
Now it follows from~\eqref{eq:conditional_PBV_op(1)} and a Riemann approximation that
\begin{align}\label{eq:CLT_sum=P=>Ct}
			\sum_{m=1}^{M_n}\conexp{\myp{\theta^n_m}^2}{\infor{\frac{m-1}{M_n}}}\Pconverge \frac{1}{c}\int_{0}^{1}\myp{\frac{2c}{3}\sigma^2_u+\frac{2\sigma^2_U}{c}}^2\diff u.
\end{align}

Next, denote $\widebar{\Delta^n_m V} =  V^n_{(2m-1)k_n} - V^n_{{2(m-1)k_n}}$, for any process $V$.			
We will show that
\begin{equation}
	\sum_{m=1}^{M_n}\conexp{\theta^n_m\widebar{\Delta^n_m N}}{\mathcal{F}^n_{2(m-1)k_n} }\Pconverge 0,
\label{eq:CLT_MGL_ortho}
\end{equation}
for any bounded martingale $N$ defined on the same probability space, where $\mathcal{F}^n_i = \mathcal{F}_{i/n}$ whence $\mathcal{F}^n_{2(m-1)k_n} = \infor{\frac{m-1}{M_n}}$.
To complete the proof, it is convenient to specify the respective probability spaces as follows.
(We can always extend the probability space --- whether the noise process and the efficient price process are defined on the same probability space or not --- see e.g., the detailed arguments in~\cite{jacod2013StatisticalPropertyMMN}.)
The efficient price process lives on $(\Omega', \mathcal{F}', (\mathcal{F}'_t)_{t\in\reals},\mathbb{P}')$.
The noise process $(U_i)_{i\in\mathbb{N}}$ is defined on $(\Omega^{''}, \mathcal{F}^{''}, (\mathcal{F}^{''}_i)_{i\in\mathbb{N}},\mathbb{P}^{''})$,
where the filtration is defined by $\mathcal{F}^{''}_i = \sigma\myp{U_j,j\leq i, j\in\mathbb{N}}$ and $\mathcal{F}^{''}=\bigvee_{i\in\mathbb{N}}\mathcal{F}^{''}_i$.
Let
\begin{equation}
\Omega=\Omega'\times\Omega^{''},\quad  \mathcal{F} = \mathcal{F}'\otimes\mathcal{F}^{''},\quad \mathbb{P}\myp{\diff \omega',\diff\omega^{''}} = \mathbb{P}'\myp{\diff \omega'}\mathbb{P}^{''}\myp{\diff \omega^{''}}.
\end{equation}
For a realization of observation times $(t^n_i)_{0\leq i\leq n}$, we introduce $\mathcal{F}^n_i = \mathcal{F}'_{t^n_i}\otimes\mathcal{F}^{''}_i$.

%With the above notations, we can go through the proofs on p.44 and p.45. In particular, $\widetilde{\mathcal{F}'_t} = \cap_{s\geq t}\mathcal{F}'_s\otimes\mathcal{F}^{''}$ still makes sense (as a \emph{continuous} filtration). $\mathcal{F}_{\frac{m-1}{M_n}}$ should be replaced by $\mathcal{F}^n_{2(m-1)k_n}$, and it is true $\mathcal{F}_{{2(m-1)k_n}/{n}}\subset \widetilde{\mathcal{F}'}_{{2(m-1)k_n}/{n}}$.

%Let $(\Omega',\mathcal{F}',(\mathcal{F}'_t), \mathbb{P}')$ and $(\Omega'',\mathcal{F''},(\mathcal{F}''_t)_t, \mathbb{P}'')$ be the stochastic basis on which the efficient price and the noise processes are defined.
%We can further restrict $\mathcal{F}'_t = \sigma\{U_s:0\leq s\leq t \}$ for $t\in[0,1]$ and set $\mathcal{F}'' = \mathcal{F}''_1.$
%An extension of the probability spaces can be obtained so that the two processes are defined in the same probability space as we specified earlier.
%\begin{align*}
%			\Omega = \Omega'\times\Omega'',\quad \mathcal{F} = \mathcal{F'}\otimes \mathcal{F''},\quad \prob{\diff \omega',\diff \omega''} = \mathbb{P}'(\diff \omega')\mathbb{P''}(\diff \omega'');\quad \mathcal{F}_t = \bigcap_{s>t}\mathcal{F}'_s\otimes\mathcal{F}''_s.
%\end{align*}
			
According to~\cite{jacod2009pre-averaging-2} and the proof of Theorem IX 7.28 of~\cite{jacod1987limit}
it suffices to consider martingales in $\mathcal{N}^0$ or $\mathcal{N}^1$, where $\mathcal{N}^0$ is the set of all bounded martingales on $(\Omega',\mathcal{F}',\mathbb{P}')$, orthogonal to $W$, and $\mathcal{N}^1$ is the set of all martingales
having a limit $N_\infty = f(Y_{t_1},\ldots,Y_{t_q})$,
where $f$ is any bounded Borel function on $\reals^q$, $t_1<\ldots<t_q$ and $q\geq 1$.
			
First, let $N\in\mathcal{N}^0$ and let $\mathcal{\widetilde{F}}'_t=\bigcap_{s>t}\mathcal{F}'_s\otimes\mathcal{F}''$.
Then, for any $t>\frac{m-1}{M_n}$, $\overline{\theta}^n_m(t):=\conexp{\theta_m^n}{\mathcal{\widetilde{F}'}_t}$,
conditional on $\sigma_{\frac{m-1}{M_n}}$,
is a martingale with respect to the filtration generated by $\{W_t-W_{\frac{m-1}{M_n}}|t>\frac{m-1}{M_n}\}$.
By the martingale representation theorem, we have $\overline{\theta}^n_m(t)=\overline{\theta}^n_m(\frac{m-1}{M_n}) + \int_{\frac{m-1}{M_n}}^{t}\gamma_u\diff W_u$ for some predictable process $\gamma$.
Now it follows from the orthogonality of $W,N$ and the martingale property of $N$ that
\begin{align*}
			\conexp{\theta^n_m\widebar{\Delta^n_mN}}{\widetilde{\mathcal{F}}'_{\frac{m-1}{M_n}}}	= \conexp{\myp{\theta^n_m - \overline{\theta}^n_m\myp{\frac{m-1}{M_n}}}\widebar{\Delta^n_mN} + \overline{\theta}^n_m\myp{\frac{m-1}{M_n}}\widebar{\Delta^n_mN}}{\widetilde{\mathcal{F}}'_{\frac{m-1}{M_n}}}=0,
			%+\conexp{}{\widetilde{\mathcal{F}}'_{\frac{m-1}{M_n}}},
\end{align*}
which leads to
\begin{align}
			\conexp{\theta^n_m\widebar{\Delta^n_mN}}{ {\mathcal{F}}^n_{2(m-1)k_n}}= 0,
\end{align}
since $\mathcal{F}_t\subset\widetilde{\mathcal{F}}'_t$.
			
Next, assume that $N\in\mathcal{N}^1$.
It can be shown (see~\cite{jacod2009pre-averaging-2}) that there exists some $\hat{f}_t$
such that $t\in[t_l,t_{l+1})$, $N_t = \hat{f}_t(Y_{t_0}, Y_{t_1},\ldots,Y_{t_{l}})$ with $t_0=0,t_{q+1}=\infty$, and such that it is measurable in $(Y_{t_1},\ldots,Y_{t_l})$.
Hence, $\widebar{\Delta^n_m N}=0$ if %the difference $\widebar{\Delta^n_m N}$
it does not cover any of the points $t_1,\ldots,t_{q+1}$.
But such intervals (to compute %the difference
$\widebar{\Delta^n_m N}$) that contain any of $t_1,\ldots,t_{q+1}$ are at most finite in number.
Furthermore, by the boundedness of $N$ and the conditional Cauchy-Schwarz inequality,
we have the following:
\begin{align*}
\conexp{\abs{\theta^n_m\widebar{\Delta^n_m N}}}{\mathcal{F}^n_{2(m-1)k_n}}\leq \sqrt{\conexp{\myp{\theta^n_m}^2}{\mathcal{F}^n_{2(m-1)k_n}}}\sqrt{\conexp{\myp{\widebar{\Delta^n_m N}}^2}{\mathcal{F}^n_{2(m-1)k_n}}}=O_p(n^{-1/4}).
\end{align*}
Now~\eqref{eq:CLT_MGL_ortho} follows since there are at most finitely many such intervals.

The following is also trivial:
\begin{align}
\conexp{\theta^n_m\widebar{\Delta^n_m W}}{\mathcal{F}^n_{2(m-1)k_n}} = 0,
\label{eq:CLT_W_ortho}
\end{align}
since $\theta^n_m$ is an even functional of $U$ and $W$ and $(U,W)$ are distributed symmetrically.

From~\eqref{eq:key_condi_approxi}, we know that $(\theta^n_m)^2\idfun{\abs{\theta^n_m}>\varepsilon}=o_p(n^{-1/2})$ for any $\varepsilon>0$.
We then have
\begin{align}
\label{eq:CLT_tight}
\sum_{m=1}^{M_n}\conexp{(\theta^n_m)^2\idfun{\abs{\theta^n_m}>\varepsilon}}{ \mathcal{F}^n_{2(m-1)k_n} }\Pconverge0.
\end{align}
Now the proof is complete in view of~\eqref{eq:CLT_sum=P=>0},~\eqref{eq:CLT_sum=P=>Ct},~\eqref{eq:CLT_MGL_ortho},~\eqref{eq:CLT_W_ortho} and~\eqref{eq:CLT_tight}, and Theorem IX 7.28 of~\cite{jacod1987limit}.
\end{proof}
\begin{lemma}
We have that
\begin{equation}
			\label{eq:third_approxi}
			\sum_{m=1}^{M_n}\myp{\preavg{Y}}^2- \frac{1}{\sqrt{n}}\sum_{m=1}^{M_n}\myp{\beta^n_m}^2 = o_p(n^{-1/4}).
\end{equation}
\end{lemma}
\begin{proof}
Denote
\begin{equation}
			\widetilde{Y}^{n}_m = \sigma_{\frac{m-1}{M_n}}\preavg{W}+\preavg{U}.
\end{equation}
Then,
\begin{align*}
			\expect{\abs{\sum_{m=1}^{M_n}\myp{\preavg{Y}}^2- \frac{1}{\sqrt{n}}\sum_{m=1}^{M_n}\myp{\beta^n_m}^2}}\leq \sum_{m=1}^{M_n}\sqrt{\expect{\myp{\preavg{Y}-\widetilde{Y}^n_m}^2}}\sqrt{\expect{\myp{\preavg{Y}+\widetilde{Y}^n_m}^2}}.
\end{align*}
Since $\sqrt{\expect{\myp{\preavg{Y}+\widetilde{Y}^n_m}^2}}=O(n^{-1/4})$,
the result follows if
\begin{align}
	\sum_{m=1}^{M_n}\sqrt{\expect{\myp{\preavg{Y}-\widetilde{Y}^n_m}^2}}\rightarrow 0.
\end{align}
But this follows directly from Lemma 7.8 in~\cite{barndorff2006central}.%; see also Lemma 4.6 of~\cite{Vetter2008Thesis}.
\end{proof}
		
\begin{proof}
[Proof of Theorem~\ref{thm:CLT}]
Now the proof of Theorem~\ref{thm:CLT} is complete in view of~\eqref{eq:CLT_stable_convergence} and~\eqref{eq:third_approxi},
and our consistency result in~\eqref{eq:LLN}.%, and Remark~\ref{rmk:converge_rate_HAC_estimator}.
\end{proof}

\setcounter{equation}{0}

\section{Simulation Study under Stochastic Volatility}\label{sec:SVsimu}
In this section, we provide additional simulation results
in the presence of stochastic volatility.
%We choose the parameters of interest to mimic the estimates obtained in our empirical studies.
%Furthermore,
We simulate the microstructure noise process employing various combinations of
dependence structure and sampling frequency.
%that are consistent with our empirical findings.
% We also simulate noise processes with different dependence structures

We assume that the efficient log-price is generated by the following dynamics:
\begin{equation*}
\diff X_t = -\delta (X_t-\mu_1)\diff t + \sigma_t\diff W_t,
\qquad \diff \sigma^2_t = \kappa\myp{ \mu_2- \sigma^2_t}\diff t + \gamma \sigma_{t}\diff B_t,
%\qquad {\rm \mathbf{Cor}}\myp{W,B}=\varrho.
%\label{eq:Eff_Pirce}
\end{equation*}
where $B$ is a standard Brownian motion
and its quadratic covariation with the standard Brownian motion $W$ is $\varrho t$.
We set the parameters as follows:
$\delta = 0.5$, $\mu_1 = 1.6,\kappa =5/252,\mu_ 2 = 0.04/252 ,\gamma =0.05/252$, and $\varrho = -0.5$.
We employ the same noise process as in~\eqref{eq:ARnoise}.
We set $\expect{V^2} = 1.9\times 10^{-7}$, and $\expect{\epsilon^2} = 1.3\times 10^{-7}$.
Note that these parameters are slightly different from those in Section \ref{sec:simulation},
which were based on~\cite{Ait-Sahalia2011DependentNoise}.
They are chosen to mimic the results of our empirical studies.  %We set $\rho =0.7,0,-0.7$ with different choices of $\Delta_n$ to mimic the dependence structure of microstructure noise revealed in our empirical studies.

Figure~\ref{fig:SVsimu} presents the estimates of the second moments of noise.
Clearly, the bias correction can be important,
potentially yielding significantly improved results.
Turning to the estimation of the integrated volatility
using $\widehat{\rm IV}_{\rm step1}$, $\widehat{\rm IV}_{n}$, $\widehat{\rm IV}_{\rm step2}$ and $\widehat{\rm IV}_{\rm step3}$,
we observe from Table~\ref{tab:SVsimu} similar results under stochastic volatility
as in our previous simulation studies that assumed deterministic volatility:
the two-step estimators of the integrated volatility have much smaller bias
and only slightly larger standard deviations when noise is dependent.
One more iteration of bias corrections further improves the performance
when noise is %tends to be
serially correlated.
They also deliver reliable estimates when noise turns out to be independent.
\begin{sidewaysfigure}[p]
\centering
\input{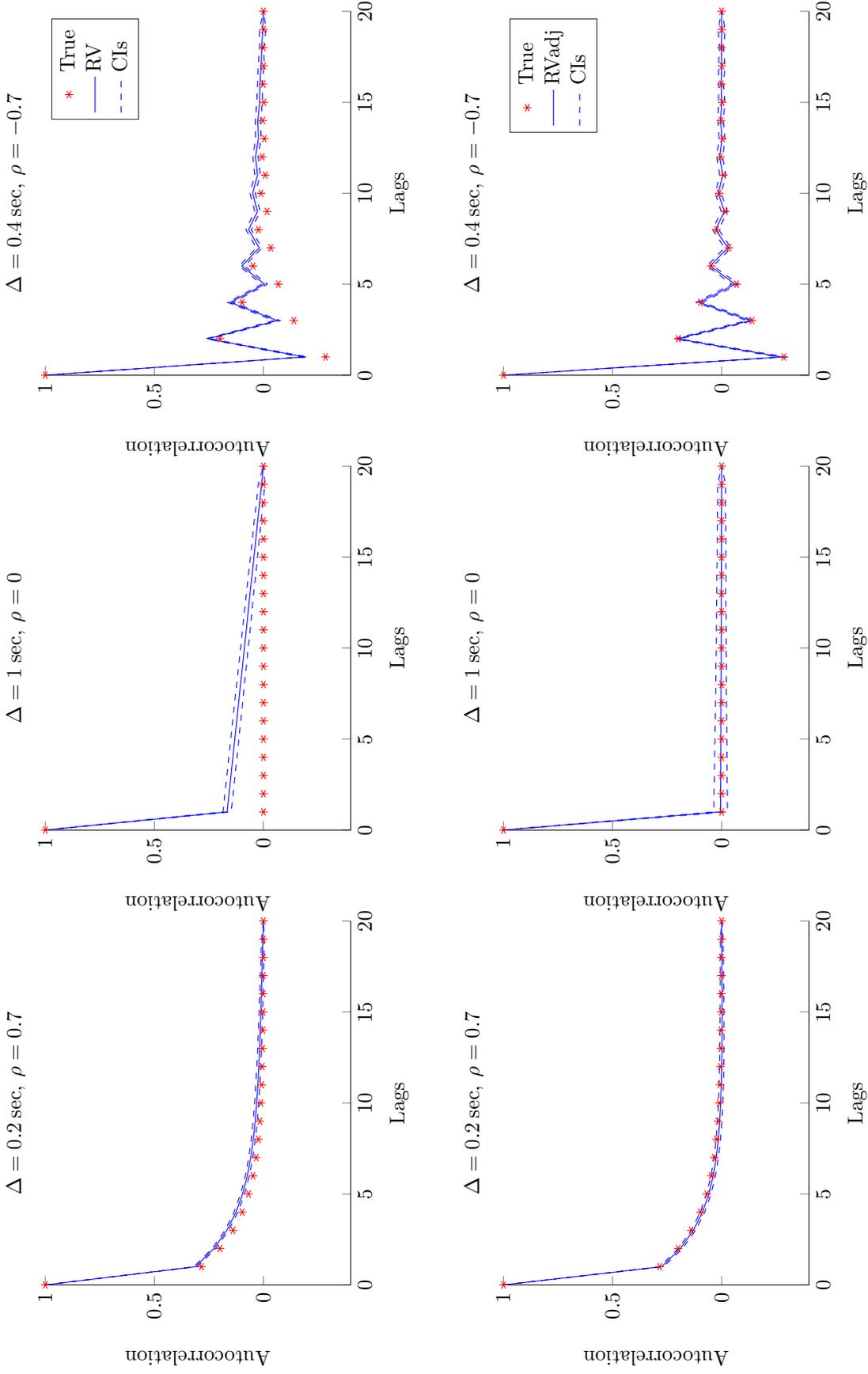}
\caption{Realized volatility (RV) and bias corrected realized volatility (RVadj) estimators
of the autocorrelations of noise in the presence of stochastic volatility
against the number of lags $j$,
averaged over $1\mathord{,}000$ simulated samples.
Top panel: RV estimators of the autocorrelations of noise (solid).
Bottom panel: RV estimators of the autocorrelations of noise with finite sample bias correction (solid).
The true autocorrelations are displayed in stars and
the 95\% simulated confidence intervals are dashed.
From the left to the right,
the three combinations of $\rho,\Delta_n$ mimic transaction time sampling, regular time sampling (at 1 sec scale), and tick time sampling.
The tuning parameters are set as follows: $j_n=20$, $i_n=10$ and $c=0.2$. }
\label{fig:SVsimu}
\end{sidewaysfigure}

\begin{table}[h]
	\centering
	\begin{tabular}{l|c|c|c}\hline\hline
		$\rho,\Delta_n$                             &	$\rho=0.7$, $\Delta_n=0.2$ sec             &	$\rho=0$, $\Delta_n=1$ sec                      & 	 $\rho=-0.7$, $\Delta_n=0.4$ sec      \\\hline
		$\widehat{\rm IV}_{\rm step1}-\int_{0}^{1}\sigma^2_t\diff t$     &5.02e-5 (1.10e-5)  &4.33e-7 (1.32e-5)   & -1.50e-5 (9.97e-6)     \\
		$\widehat{\rm IV}_{n}-\int_{0}^{1}\sigma^2_t\diff t$             &-1.64e-5 (1.09e-5)  &-7.82e-5 (1.18e-5)  &-3.17e-5 (9.77e-6)     \\
		$\widehat{\rm IV}_{\rm step2}-\int_{0}^{1}\sigma^2_t\diff t$     &4.32e-6 (1.20e-5)  &9.94e-7 (1.79e-5)  &-3.15e-6 (1.17e-5)  \\
		$\widehat{\rm IV}_{\rm step3}-\int_{0}^{1}\sigma^2_t\diff t$     &-2.32e-7 (1.21e-5)  &1.27e-6 (2.06e-5)  &-8.05e-7 (1.21e-5)  \\\hline\hline	
	\end{tabular}
\caption{Estimation of the integrated volatility in the presence of stochastic volatility
and under various combinations of noise dependence structure and sampling frequency.
We report the means of the bias of the four integrated volatility estimators:
$\widehat{\rm IV}_{\rm step1}-\int_{0}^{1}\sigma^2_t\diff t, \widehat{\rm IV}_{n}-\int_{0}^{1}\sigma^2_t\diff t$, $\widehat{\rm IV}_{\rm step2}-\int_{0}^{1}\sigma^2_t\diff t$
and $\widehat{\rm IV}_{\rm step3}-\int_{0}^{1}\sigma^2_t\diff t$,
based on $1\mathord{,}000$ simulations with standard deviations between parentheses.
From the left to the right, the three combinations of $\rho,\Delta_n$ mimic
transaction time sampling, regular time sampling (at 1 sec scale), and tick time sampling.
%The tuning parameter of the RV estimator is $j_n=20$ and $j_{n}^{*}=10$.
The tuning parameters are set as follows: $j_n=20$, $i_n=10$ and $c=0.2$.}
\label{tab:SVsimu}
\end{table}

\setcounter{equation}{0}
\setcounter{figure}{0}

\section{Empirical Study of Transaction Data for General Electric}
\label{sec:GE_Empirical}
We collect $2\mathord{,}721\mathord{,}475$ transaction prices of General Electric (GE)
over the month January 2011.
On average there are 5.8 observations per second.
In contrast to the analysis of Citigroup transaction prices
in Sections~\ref{subsec:estimate_2nd_moments_noise_Citi} and~\ref{subsec:IV_Citi},
bias correction plays a very pronounced role here.
Despite the high data frequency,
the finite sample bias can be very significant
if the underlying noise-to-signal ratio is small
(recall Remark~\ref{rmk:why_correct_bias}).
This is indeed the case as Figure~\ref{fig:GE_C_obv_n2s} reveals:
compared with Citigroup, the data frequency of the General Electric sample is typically lower
but the noise-to-signal ratio is also (much) smaller.
While the data frequency is immediately available, the noise-to-signal ratio is latent.
Therefore, one should always be wary to rely solely on asymptotic theory in practice.

The top panel of Figure~\ref{fig:GE_autocorr_IVs} shows that both the realized volatility (RV) and local averaging (LA) estimators
indicate that the noise is strongly autocorrelated,
while the bias corrected realized volatility (BCRV) estimator reveals that the noise is only weakly dependent.
Such a pattern also appears in our simulation study,
where we have seen that it is the finite sample bias that induces this discrepancy.
The bottom panel of Figure~\ref{fig:GE_autocorr_IVs} plots two estimators of the integrated volatility,
$\widehat{\rm IV}_{n}$ and $\widehat{\rm IV}_{\rm step2}$,
to illustrate that the finite sample bias correction is particularly essential.
%Again, we obtain similar findings as in our simulation study:
%ignoring the finite sample bias yields lower estimates (see $\widehat{\rm IV}_{n}$),
%while ignoring the dependence in noise yields higher estimates (see $\widehat{\rm IV}_{\rm step1}$),
%and our two-step estimator $\widehat{\rm IV}_{\rm step2}$ lies in between.
If one would solely rely on asymptotic theory,
then one would end up with much lower estimates and narrow confidence intervals that may well exclude the true values!

\begin{figure}[h]
\centering
% Minimal pgfplots version: 1.3
%  http://www.mathworks.com/matlabcentral/fileexchange/22022-matlab2tikz
\begin{tikzpicture}

\begin{axis}[%
width=5in,
height=2in,
at={(0.758333in,0.48125in)},
scale only axis,
every outer x axis line/.append style={black},
every x tick label/.append style={font=\color{black}},
xmin=1,
xmax=20,
xlabel={Trading days},
every outer y axis line/.append style={black},
every y tick label/.append style={font=\color{black}},
ymin=0,
ymax=0.006,
ylabel={Noise-to-Signal Ratio},
axis x line*=bottom,
axis y line*=left,
legend style={at={(0.0,1.0)},legend cell align=left,align=left,anchor=north west,draw=white!15!black}
]
\addplot [color=blue,dashed,mark=o,mark options={solid}]
  table[row sep=crcr]{%
1	0.00248908950516677\\
2	0.0032773856242139\\
3	0.00418468575260701\\
4	0.00314866911451027\\
5	0.00262798526812869\\
6	0.00359756164280749\\
7	0.0050483204015567\\
8	0.00399327400141307\\
9	0.00343237530516356\\
10	0.00215536686392205\\
11	0.00202783589314785\\
12	0.00345773761538953\\
13	0.00290747316022016\\
14	0.00338402035716524\\
15	0.00537975308640994\\
16	0.00319439722858327\\
17	0.00330171297154659\\
18	0.00585847293709716\\
19	0.00315724274331206\\
20	0.00537786102601347\\
};
\addlegendentry{C};

\addplot [color=red,dashed,mark=o,mark options={solid}]
  table[row sep=crcr]{%
1	0.00139822600872148\\
2	0.000901099570723452\\
3	0.000790332423834797\\
4	0.00097052797397897\\
5	0.000281494787176808\\
6	0.000285376781942547\\
7	0.00062387911889265\\
8	0.000898038832699581\\
9	0.000165151202901355\\
10	0.000425541426877564\\
11	0.000926113906330832\\
12	0.000465284248332994\\
13	0.000263988947673748\\
14	0.000282755205588858\\
15	0.000613705686409644\\
16	0.000332221225044094\\
17	0.000802058406232685\\
18	0.000399774355084484\\
19	0.000115431626977943\\
20	0.000689684924578138\\
};
\addlegendentry{GE};

\addplot [color=blue,solid,mark=o,mark options={solid},forget plot]
  table[row sep=crcr]{%
1	0.00248908950516677\\
2	0.0032773856242139\\
3	0.00418468575260701\\
4	0.00314866911451027\\
5	0.00262798526812869\\
6	0.00359756164280749\\
7	0.0050483204015567\\
8	0.00399327400141307\\
9	0.00343237530516356\\
10	0.00215536686392205\\
11	0.00202783589314785\\
12	0.00345773761538953\\
13	0.00290747316022016\\
14	0.00338402035716524\\
15	0.00537975308640994\\
16	0.00319439722858327\\
17	0.00330171297154659\\
18	0.00585847293709716\\
19	0.00315724274331206\\
20	0.00537786102601347\\
};
\addplot [color=red,solid,mark=o,mark options={solid},forget plot]
  table[row sep=crcr]{%
1	0.00139822600872148\\
2	0.000901099570723452\\
3	0.000790332423834797\\
4	0.00097052797397897\\
5	0.000281494787176808\\
6	0.000285376781942547\\
7	0.00062387911889265\\
8	0.000898038832699581\\
9	0.000165151202901355\\
10	0.000425541426877564\\
11	0.000926113906330832\\
12	0.000465284248332994\\
13	0.000263988947673748\\
14	0.000282755205588858\\
15	0.000613705686409644\\
16	0.000332221225044094\\
17	0.000802058406232685\\
18	0.000399774355084484\\
19	0.000115431626977943\\
20	0.000689684924578138\\
};
\addplot [color=blue,solid,mark=o,mark options={solid},forget plot]
  table[row sep=crcr]{%
1	0.00248908950516677\\
2	0.0032773856242139\\
3	0.00418468575260701\\
4	0.00314866911451027\\
5	0.00262798526812869\\
6	0.00359756164280749\\
7	0.0050483204015567\\
8	0.00399327400141307\\
9	0.00343237530516356\\
10	0.00215536686392205\\
11	0.00202783589314785\\
12	0.00345773761538953\\
13	0.00290747316022016\\
14	0.00338402035716524\\
15	0.00537975308640994\\
16	0.00319439722858327\\
17	0.00330171297154659\\
18	0.00585847293709716\\
19	0.00315724274331206\\
20	0.00537786102601347\\
};
\addplot [color=red,solid,mark=o,mark options={solid},forget plot]
  table[row sep=crcr]{%
1	0.00139822600872148\\
2	0.000901099570723452\\
3	0.000790332423834797\\
4	0.00097052797397897\\
5	0.000281494787176808\\
6	0.000285376781942547\\
7	0.00062387911889265\\
8	0.000898038832699581\\
9	0.000165151202901355\\
10	0.000425541426877564\\
11	0.000926113906330832\\
12	0.000465284248332994\\
13	0.000263988947673748\\
14	0.000282755205588858\\
15	0.000613705686409644\\
16	0.000332221225044094\\
17	0.000802058406232685\\
18	0.000399774355084484\\
19	0.000115431626977943\\
20	0.000689684924578138\\
};
\end{axis}

\begin{axis}[%
width=5in,
height=2in,
at={(0.758333in,3.5in)},
scale only axis,
every outer x axis line/.append style={black},
every x tick label/.append style={font=\color{black}},
xmin=1,
xmax=20,
xlabel={Trading days},
every outer y axis line/.append style={black},
every y tick label/.append style={font=\color{black}},
ymin=50000,
ymax=550000,
ylabel={Number of Observations},
axis x line*=bottom,
axis y line*=left,
legend style={at={(0.0,1.0)},legend cell align=left,align=left,anchor=north west,draw=white!15!black}
]
\addplot [color=blue,dashed,mark=asterisk,mark options={solid}]
  table[row sep=crcr]{%
1	273557\\
2	245194\\
3	244846\\
4	266477\\
5	279287\\
6	176689\\
7	138993\\
8	257731\\
9	228877\\
10	308362\\
11	533860\\
12	342645\\
13	274582\\
14	250505\\
15	207478\\
16	180283\\
17	186416\\
18	149881\\
19	244727\\
20	142669\\
};
\addlegendentry{C};

\addplot [color=red,dashed,mark=asterisk,mark options={solid}]
  table[row sep=crcr]{%
1	96369\\
2	127409\\
3	109119\\
4	90890\\
5	100177\\
6	122286\\
7	90394\\
8	96403\\
9	81959\\
10	89181\\
11	116668\\
12	154569\\
13	135325\\
14	342971\\
15	185937\\
16	192489\\
17	133464\\
18	129608\\
19	221653\\
20	104604\\
};
\addlegendentry{GE};

\addplot [color=blue,dashed,mark=asterisk,mark options={solid},forget plot]
  table[row sep=crcr]{%
1	273557\\
2	245194\\
3	244846\\
4	266477\\
5	279287\\
6	176689\\
7	138993\\
8	257731\\
9	228877\\
10	308362\\
11	533860\\
12	342645\\
13	274582\\
14	250505\\
15	207478\\
16	180283\\
17	186416\\
18	149881\\
19	244727\\
20	142669\\
};
\addplot [color=red,dashed,mark=asterisk,mark options={solid},forget plot]
  table[row sep=crcr]{%
1	96369\\
2	127409\\
3	109119\\
4	90890\\
5	100177\\
6	122286\\
7	90394\\
8	96403\\
9	81959\\
10	89181\\
11	116668\\
12	154569\\
13	135325\\
14	342971\\
15	185937\\
16	192489\\
17	133464\\
18	129608\\
19	221653\\
20	104604\\
};
\addplot [color=blue,dashed,mark=asterisk,mark options={solid},forget plot]
  table[row sep=crcr]{%
1	273557\\
2	245194\\
3	244846\\
4	266477\\
5	279287\\
6	176689\\
7	138993\\
8	257731\\
9	228877\\
10	308362\\
11	533860\\
12	342645\\
13	274582\\
14	250505\\
15	207478\\
16	180283\\
17	186416\\
18	149881\\
19	244727\\
20	142669\\
};
\addplot [color=red,dashed,mark=asterisk,mark options={solid},forget plot]
  table[row sep=crcr]{%
1	96369\\
2	127409\\
3	109119\\
4	90890\\
5	100177\\
6	122286\\
7	90394\\
8	96403\\
9	81959\\
10	89181\\
11	116668\\
12	154569\\
13	135325\\
14	342971\\
15	185937\\
16	192489\\
17	133464\\
18	129608\\
19	221653\\
20	104604\\
};
\end{axis}
\end{tikzpicture}%
\caption{Number of daily observations of transaction prices (top panel)
and noise-to-signal ratio (bottom panel) for Citigroup (C) and General Electric (GE).
Sample period: January, 2011, consisting of 20 trading days.
In the bottom panel, the noise-to-signal ratio, $\frac{\sigma^2_U}{\int_{0}^{1}\sigma_s^2\diff s}$,
is estimated by $\frac{\hat{\sigma}^2_{U,{\rm step2}}}{\widehat{{\rm IV}}_{\rm step2}}$,
where $\hat{\sigma}^2_{U,{\rm step2}}$ and $\widehat{{\rm IV}}_{\rm step2}$ are defined in \eqref{eq:2ndStepSigma2U} and~\eqref{eq:2ndStepIV},
respectively.
We set $j_n=30$, $i_n = 15$ and $c=0.2$.%The tuning parameter of the RV estimator is $j_n=30$ and $j_{n}^{*}=15$.
}
\label{fig:GE_C_obv_n2s}
\end{figure}
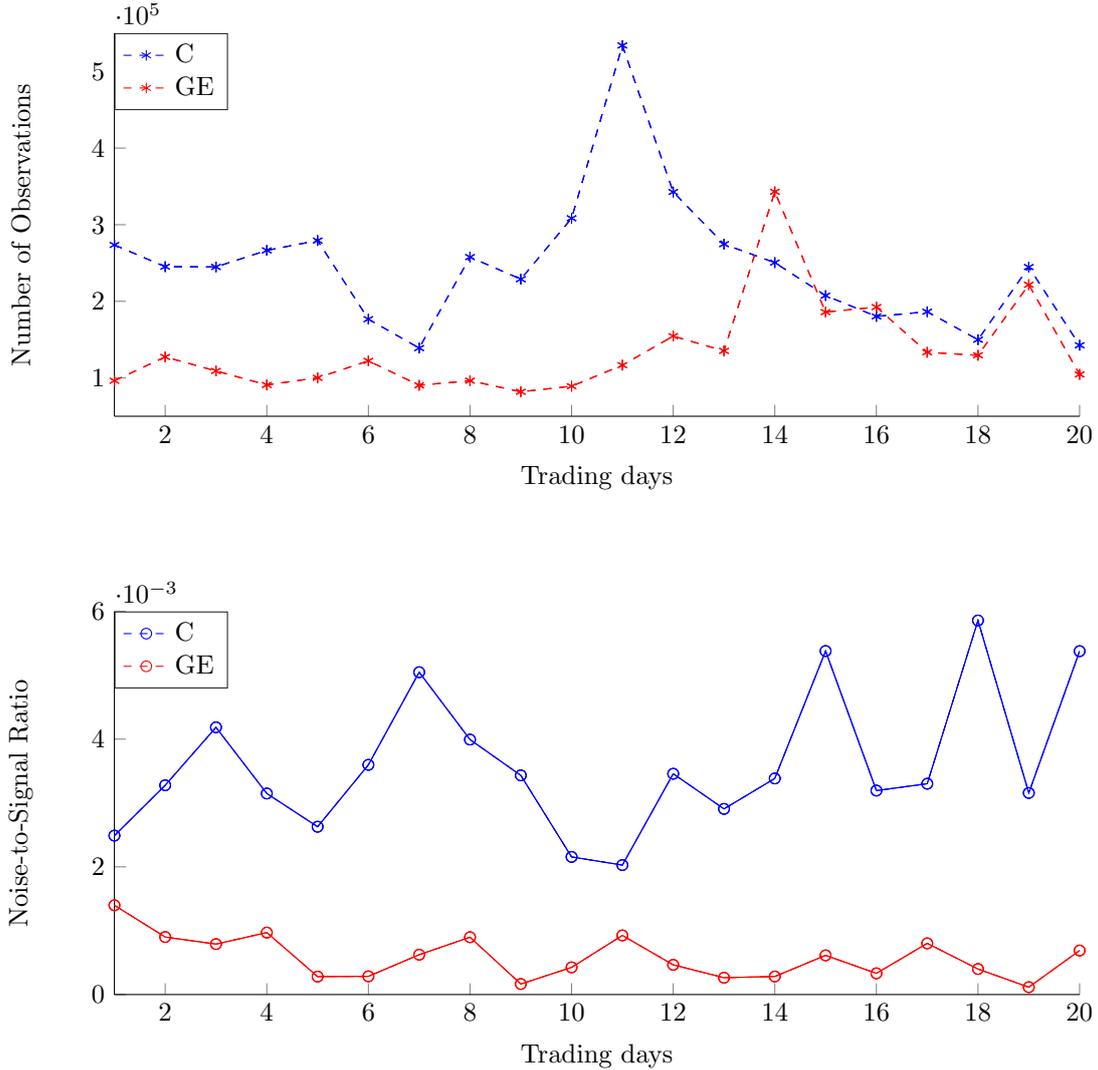

\begin{sidewaysfigure}[h]
\centering
\input{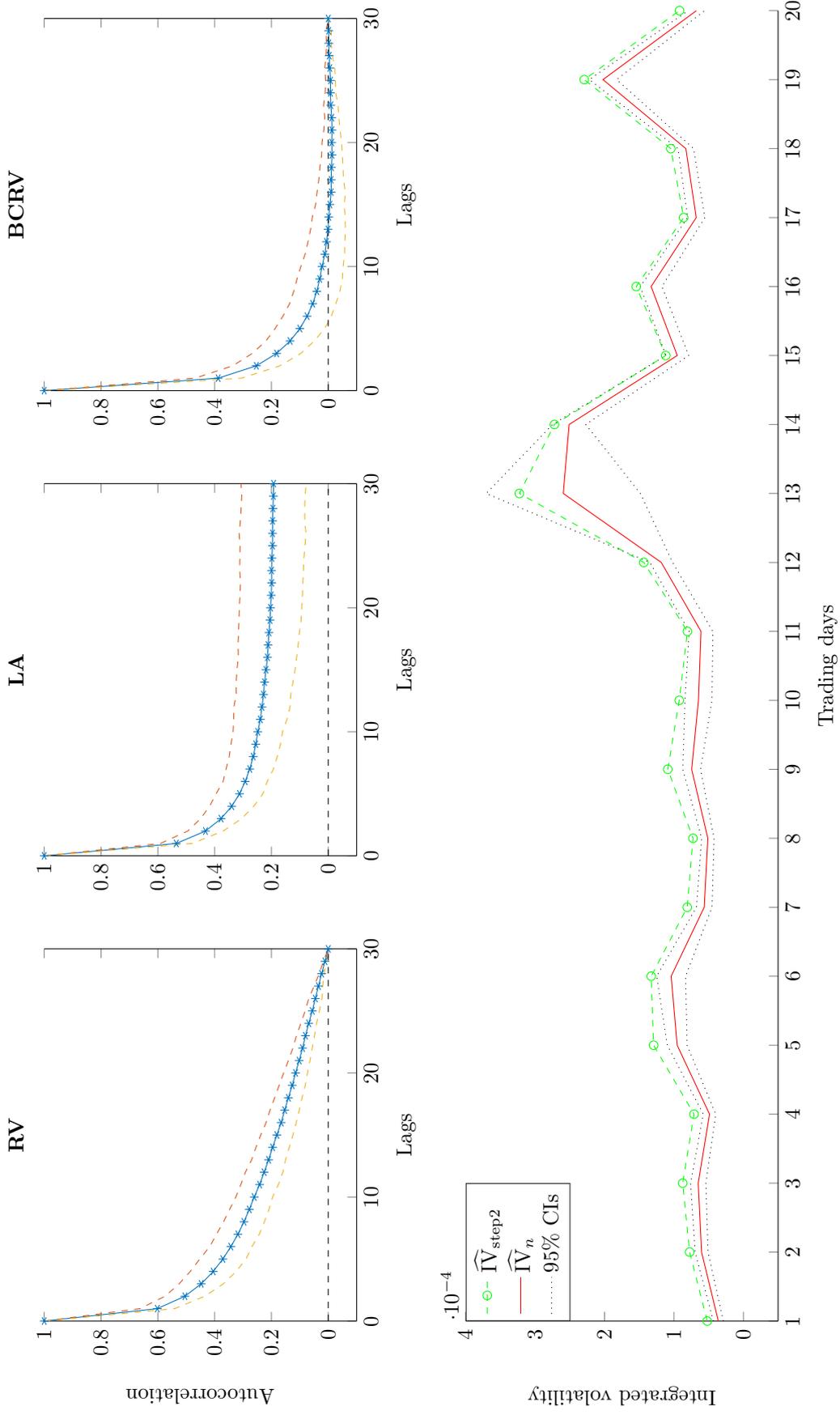}
\caption{Autocorrelations of noise and integrated volatility based on transaction data for General Electric (GE).
Sample period: January, 2011, consisting of 20 trading days.
On average there are 5.8 observations per second in the sample.
Top panel: From the left to the right, we display the realized volatility (RV), local averaging (LA),
and the bias corrected realized volatility (BCRV) estimators of the autocorrelations of noise against the number of lags $j$.
The three estimators are applied to and next averaged over each of the 20 trading days.
The stars indicate the means of the 20 estimates.
The dashed lines are 2 standard deviations away from the mean.
Bottom panel: Estimation of the integrated volatility.
The estimators %$\widehat{\rm IV}_{\rm step1}$,
$\widehat{\rm IV}_{\rm step2}$ and $\widehat{\rm IV}_{n}$
are given by%~\eqref{eq:1stStepIV},
~\eqref{eq:2ndStepIV} and~\eqref{eq:consistency_SV_nonpar}.
The asymptotic confidence intervals (CIs) are based on the limit distribution in Theorem~\ref{thm:CLT}.
%The tuning parameter of the RV estimator is $j_{n}=30$ in both panels and $j_{n}^{*}=15$ in the bottom panel.
We set $j_n=30$, $i_n = 15$ and $c=0.2$.}
\label{fig:GE_autocorr_IVs}
\end{sidewaysfigure}
\end{appendices}

\end{document}